\documentclass[10pt]{article}
\usepackage{amsmath}
\usepackage{latexsym}
\usepackage{amsthm}
\usepackage{amssymb}
\usepackage[normalem]{ulem}
\usepackage{graphicx}
\usepackage{xr}
\usepackage{color}
\usepackage{epsfig}

\newtheorem{definition}{Definition}[subsection]
\newtheorem{theorem}{Theorem}[section]
\newtheorem{proposition}{Proposition}[subsection]
\newtheorem{lemma}{Lemma}[subsection]
\newtheorem{corollary}{Corollary}[subsection]
\newtheorem{remark}{Remark}[subsection]

\renewcommand{\labelenumi}{(\roman{enumi})}
\newcommand{\gothic}{\mathfrak}
\renewcommand{\bold}[1]{\medskip \noindent {\bf #1 }\nopagebreak}
\long\def\symbolfootnote[#1]#2{\begingroup%
\def\thefootnote{\fnsymbol{footnote}}\footnote[#1]{#2}\endgroup}

\begin{document}

\title{Coarse differentiation and quasi-isometries of a class of solvable Lie groups II}
\date{}
\author{Irine Peng}
\maketitle

\begin{abstract}
In this paper, we continue with the results in \cite{Pg} and compute the group of
quasi-isometries for a subclass of split solvable unimodular Lie groups.  Consequently,
we show that any finitely generated group quasi-isometric to a member of the subclass has
to be polycyclic, and is virtually a lattice in an abelian-by-abelian solvable Lie group.
We also give an example of a unimodular solvable Lie group that is not quasi-isometric to
any finitely generated group, as well deduce some quasi-isometric rigidity results.
\end{abstract}

\tableofcontents

\section{Introduction}
A \emph{$(\kappa, C)$ quasi-isometry} $f$ between metric spaces $X$ and $Y$ is a map $f: X \rightarrow Y$ satisfying
\[ \frac{1}{\kappa} d(p,q) -C \leq d(f(p), f(q)) \leq \kappa d(p,q) + C \] \noindent with the additional property
that there is a number $D$ such that $Y$ is the $D$ neighborhood of $f(X)$.  Two quasi-isometries $f, g$ are
considered to be equivalent if there is a number $E>0$ such that $d(f(p), g(p)) \leq E$ for all $p \in X$. \smallskip

Let $G= \mathbb{R}^{m} \rtimes_{\varphi} \mathbb{R}^{n}$, $G'=\mathbb{R}^{m'} \rtimes_{\varphi'} \mathbb{R}^{n'}$ be connected,
simply connected non-degenerate unimodular split solvable groups (See section \ref{geometry of G} for definitions).  We say a map
from $G$ to $G'$ is \emph{standard}, if it splits as a product map that respects $\varphi$ and $\varphi'$
(See definition \ref{standard map}).  Also homomorphisms $\varphi$ is called
\emph{diagonalizable} if its image can be conjugated into the set of diagonal matrices.  The main result of this paper is the following
statement.\\

\textit{Theorem \ref{behaviors of phi} (abridged)}
\textit{Let $G$, $G'$ be non-degenerate, unimodular, split abelian-by-abelian solvable Lie groups, and $\phi: G \rightarrow G'$ a $\kappa, C$
quasi-isometry.  Then $\phi$ is bounded distance from a composition of a left translation and a standard
map.}

Consequently,
\textit{Corollary \ref{diagORnot}}
\textit{ If $\varphi$ is diagonalizable and $\varphi'$ isn't, then there is no quasi-isometry
between them.}

\begin{corollary} \label{QIgroup}
\[ \mathcal{QI}(G) = \left( \prod_{[\alpha]} Bilip(V_{[\alpha]}) \right) \rtimes \mbox{Sym}(G) \] \end{corollary}

\noindent Here $[\alpha]$ is an equivalence class of roots and $Sym(G)$ is a finite group,
analogous to the Weyl group in reductive Lie groups, that that reflects
the symmetries of $G$. (See section \ref{geometry of G}) \smallskip

When $\varphi$ is diagonalizable, as an application the work by Dymarz \cite{Dy} on quasi-conformal maps on the boundary of $G$,
we have  \begin{corollary} \label{showing polycyclic}
In the case that $\varphi$ is diagonalizable, if $\Gamma$ is a finitely generated group quasi-isometric to a $G$,
then $\Gamma$ is virtually polycyclic.  \end{corollary}

\subsection{Proof outline}
Our starting point is Theorem \ref{exisence of standard maps in small boxes}, and our
first task is to show that the $\mathbf{A}'$ part of the standard maps (See Definition \ref{standard map}) $f_{i}$ are affine.  This is done
in Section 3, where will see that the linear part is a scalar multiple of a finite order
element in $O(n)$ (where $n$ is the rank of $G$).  We also give interpretations of the linear and constant part of
$f_{i}$ in terms of properties of $G'$ and measure of certain sets in the box where $f_{i}$ was partially defined.
In Section 4, we show that the linear part of the $f_{i}$'s in different boxes have to be  the same up to scalar
multiple in the case that that the rank of $G$ is 2 or higher.  The rank $1$ case is the
same as the content in \cite{EFW2}.  The proof for higher rank case basically consists of
many rank $1$ argument as appeared in \cite{EFW2}.  In the last section, we put all
theses partially defined standard maps together to produce a splitting of the original
quasi-isometry.


\bold{Acknowledgement} I would like to thank Alex Eskin and David Fisher for their
incredible patience and assistance.  I also wish to thank Tullia Dymarz for helpful
conversations and coordinating results of her paper \cite{Dy} with this one, and to Mikhail Ershov for finding
me the reference containing the example of a solvable Lie group that admits no lattices.

\section{Preliminaries}
Here we recall the settings from \cite{Pg} and define new terms that will be used in this paper.

\subsection{Geometry of a certain class of solvable Lie groups} \label{geometry of G}
\bold{Non-degenerate, split abelian-by-abelian solvable Lie groups }
Let $\gothic{g}$ be a (real) solvable Lie algebra, and $\gothic{a}$ be a Cartan
subalgebra.  Then there are finitely many non-zero linear functionals $\alpha_{i}: \gothic{a}
\rightarrow \mathbb{C}$ called \emph{roots}, such that

\[ \gothic{g} = \gothic{a} \oplus \bigoplus_{\alpha_{i}} \gothic{g}_{\alpha_{i}} \]

\noindent where $\gothic{g}_{\alpha_{i}}=\{ x \in \gothic{g}: \forall t \in \gothic{a}, \exists n, \mbox{ such that }
(ad(t)-\alpha_{i}(t)Id)^{n}(x) = 0 \}$, $Id$ is the identity map on $\gothic{g}$, and
$ad: \gothic{g} \rightarrow Der_{\mathbb{R}}(\gothic{g})$ is the adjoint representation.
Let $\triangle$ denotes for the set of roots. Then $Aut(\gothic{a})$ acts on $\triangle$ in a natural way.  We define $Perm(\gothic{g})$ to
be the subgroup consisting $A \in Aut(\gothic{a})$ such that
\begin{enumerate}
\item it leaves the set of roots invariant, i.e. $A \triangle =\triangle$.

\item for every $\alpha \in \triangle$, dim$\gothic{g}_{\alpha \circ
A}=$dim$\gothic{g}_{\alpha}$.  \end{enumerate}

In this way, elements of $Perm(\gothic{g})$ induces a permutation on the set $\triangle$,
and we define $Sym(\gothic{g})$ to be the image of $Perm(\gothic{g})$ in the group of
permutations of $\triangle$.  For a generic $\gothic{g}$, its $Perm(\gothic{g})$ is
trivial.

We say $\gothic{g}$ is \emph{split abelian-by-abelian} if $\gothic{g}$ is a semidirect product
of $\gothic{a}$ and $\bigoplus_{i} \gothic{g}_{\alpha_{i}}$, and both are abelian Lie algebras;
\emph{unimodular} if the the roots sum up to zero; and \emph{non-degenerate} if the roots span $\gothic{a}^{*}$.  In particular,
non-degenerate means that each $\alpha_{i}$ is real-valued, and the number of roots is at least the dimension of
$\gothic{a}$.  Being unimodular is the same as saying that for every $t \in \gothic{a}$, the trace of
$ad(t)$ is zero.  We extend these definitions to a Lie group if its Lie algebra has these
properties, and write $Perm(G)$, $Sym(G)$ to mean $Perm(\gothic{g})$ and $Sym(\gothic{g})$ where $\gothic{g}$
is the Lie algebra of $G$. \smallskip

Therefore a connected, simply connected solvable Lie group $G$ that is non-degenerate, split abelian-by-abelian necessary takes the form
$G= \mathbf{H} \rtimes_{\varphi} \mathbf{A}$ such that  \begin{enumerate}
\item both $\mathbf{A}$ and $\mathbf{H}$ are abelian Lie groups.

\item $\varphi: \mathbf{A} \rightarrow Aut(\mathbf{H})$ is injective

\item there are finitely many $\alpha_{i} \in \mathbf{A}^{*} \backslash 0$ which together span
$\mathbf{A}^{*}$, and a decomposition of $\mathbf{H}=\oplus_{i} V_{\alpha_{i}}$

\item there is a basis $\mathcal{B}$ of $\mathbf{H}$ whose intersection with each of
$V_{\alpha_{i}}$ constitute a basis of $V_{\alpha_{i}}$, such that for each $\mathbf{t} \in \mathbf{A}$, $\varphi(\mathbf{t})$ with respect
to $\mathcal{B}$ is a matrix consists of blocks, one for each $V_{\alpha_{i}}$, of the form $e^{\alpha_{i}(\mathbf{t})} N(\alpha_{i}(t))$,
where $N(\alpha_{i}(t))$ is an upper triangular with 1's on the diagonal and whose off-diagonal entries are polynomials of $\alpha_{i}(t)$.
\smallskip  If in addition, $G$ is unimodular, then $\varphi(\mathbf{t})$ has determinant 1 for all
$\mathbf{t} \in \mathbf{A}$.  \end{enumerate}
\noindent The \emph{rank} of a non-degenerate, split abelian-by-abelian group $G$ is
defined to be the dimension of $\mathbf{A}$, and by a result of Cornulier \cite{Cornulier}, if two such groups are quasi-isometric,
then they have the same rank. \smallskip

By abuse of notation, we call $\triangle$ roots of $G$ as well and coordinatize its points as $((\mathbf{x}_{\alpha})_{\alpha \in \triangle}, \mathbf{t})$,
where $\mathbf{x}_{\alpha}=(x_{1,\alpha}, x_{2, \alpha}, \cdots, x_{dim(V_{\alpha}),\alpha} ) \in V_{\alpha}$, $\mathbf{t} \in
\mathbf{A}$.  A left invariant Finsler metric that is quasi-isometric to a left invariant
Riemannian is given by:

\begin{eqnarray*}
d \mathbf{t} +  \sum_{\alpha \in \triangle} e^{-\alpha(\mathbf{t})} \left( d \mathbf{x}_{\alpha} +
\sum_{j} P_{j,\alpha}(\alpha(\mathbf{t}))dx_{j,\alpha} \right) \end{eqnarray*} \noindent where $P_{j,\alpha}$ is a polynomial.
The following consequence is immediate.

\begin{lemma} \label{QI embedding}
If $G$ is non-degenerate, split abelian-by-abelian, then it can be QI embedded into $\prod_{\alpha \in \triangle}
\mathit{H}_{\mbox{dim}(V_{\alpha})+1}$, where $\mathit{H}_{s+1}=\mathbb{R}^{s}
\rtimes_{\psi} \mathbb{R}$ is a non-unimodular solvable Lie group determined by $\psi(t)=e^{t}N(t)$, where $N(t)$ is a
nilpotent matrix (upper triangular with 1's on the diagonal) with polynomial entries, equipped with a left-invariant Finsler
metric given by
\[ dt + e^{-t} \left( d\mathbf{x} + \sum_{j}P_{j}(t) dx_{j} \right) \] \noindent where $P_{j}(t)$ is
a polynomial. \end{lemma}

\begin{remark} When $\psi(t)$ is diagonal, $\mathit{H}_{s+1}$ is just the usual hyperbolic space. \end{remark}

To understand the geometry of $\mathit{H}_{s+1}$ better, we first note that the metric is bilipchitz
to one given by $dt + e^{-t}(1+\max_{j}P_{j}(t)) d \mathbf{x}$, which is quasi-isometric to
one given by $dt + e^{-t} Q(t) d\mathbf{x}$ for some polynomial $Q(t)$.  So a function
q.i. to the metric on $\mathit{H}_{s+1}$ is the following
\begin{equation}\label{distance}
d((x_{1},t_{1}), (x_{2},t_{2})) = \left \{ \begin{array}{ll}
                                     |t_{1}-t_{2}| & \mbox{ if $e^{-t_{i}}Q(t_{i}) |x_{1}-x_{2}| \leq 1 $ for some $i=1,2$};\\
                               U_{Q}(|x_{1}-x_{2}|) - (t_{1}+ t_{2}) & \mbox{ otherwise} \end{array} \right. \end{equation}
\noindent where $U_{Q}(|x_{1}-x_{2}|)=t_{0}$ satisfies
\[ e^{-t_{0}} Q(t_{0}) |x_{1}-x_{2}| = 1 \]

Since exponential grows faster than any polynomials, the function $U_{Q}$ has the
following property:
\begin{equation} \label{property of U}
\ln(x) - C_{Q} \leq    U_{Q}(x)    \leq    2 \ln(x) + C_{Q} \end{equation} \noindent for some constant $C$ depends
only on the polynomial $Q$. \medskip

Back to the description of $G$, we declare two roots equivalent if they are positive multiples of each other, and write $[\Xi]$ for the equivalence
class containing $\Xi \in \triangle$.  Moreover, for $\Xi_{i}, \Xi_{2} \in [\Xi]$, we say $\Xi_{1}$ less than $\Xi_{2}$
if $\Xi_{2}/ \Xi_{1} > 1$.  This makes sense because all roots in a root class are positive multiples of each other.  A left
translate of $V_{[\Xi]}=\oplus_{\sigma \in [\Xi]} V_{\sigma}$ will be called a \emph{horocycle of root class $[\Xi]$}.

A left translate of $\mathbf{H}$, or a subset of it, is called a
\emph{flat}.  For $p=(x_{\alpha})_{\alpha}$, $q=(y_{\alpha})_{\alpha}$ points in
$\mathbf{H}$, we compute subsets of $p \mathbf{H}$ and $q \mathbf{H}$ that are within
distance 1 of each other according to the embedded metric in Lemma \ref{QI embedding}, as the $p$ and $q$ translate of the
subset

\[  \bigcap_{\alpha \in \triangle : \ln(|x_{\alpha}-y_{\alpha}|) \geq 1}  \alpha^{-1} [U_{\alpha}(|x_{\alpha}-y_{\alpha}|), \infty ]
\subset \mathbf{A} \]

\noindent Since the roots sum up to zero in a non-degenerate, unimodular, split
abelian-by-abelian group, the set where two flats come together can be empty,i.e. the two flats have
no intersection.  If it is not empty, then the equation above says that it is an unbounded convex subset of $\mathbf{A}$ bounded by hyperplanes
parallel to root kernels.

\begin{definition} \label{standard map}
Let $G$, $G'$ be non-degenerate, split abelian-by-abelian Lie groups.  A map from $G$ to $G'$ or a subset of them, is called
\emph{standard map} if it takes the form $f \times g$, where $g: \mathbf{H} \rightarrow
\mathbf{H}'$ sends foliation by root class horocycles of $G$ to that of $G'$, and $f: \mathbf{A} \rightarrow
\mathbf{A}$ sends foliations by root kernels of $G$ to that of $G'$.   We will often refer to $f$ as the $\mathbf{A}'$ part
 of a standard map.  \end{definition}

\begin{remark} Note that when $G$ has at least $rank(G)+1$ many root kernels,
the condition on $f$ means that $f$ is affine, and when $G$ is rank 1, the condition on $f$ is
empty. \end{remark}

\subsection{Notations}

\subsubsection{General remarks about neighborhoods}

\bold{Neighborhoods of a set }
We write $B(p,r)$ for the ball centered at $p$ of radius $r$, and $N_{c}(A)$ for the $c$ neighborhood of the set
$A$.  We also write $d_{H}(A,B)$ for the Hausdorff distance between two sets $A$ and $B$.  If $\Omega \subset \mathbb{R}^{k}$ is a
bounded compact set, and $r \in \mathbb{R}$, we write $r \Omega$ for the bounded compact set that is scaled from $\Omega$ with respect
to the barycenter of $\Omega$. \smallskip

Given a set $X$, a point $x_{0} \in X$, the \emph{$(\eta, C)$ linear neighborhood of $X$ with respect to
$x_{0}$} is the set $\{ y, s.t. \exists \hat{x} \in X, d(y,\hat{x})=d(y,X) \leq \eta d(\hat{x},x_{0}) + C$.
Equivalently it is the set $\bigcup_{x \in X}B(x,\eta d(x,x_{0})+C)$.  By $(\eta, C)$ linear neighborhood of a set $X$, we mean
the $(\eta, C)$ linear neighborhood of $X$ with respect to some $x_{0}\in X$. \smallskip

If a quasi-geodesic $\lambda$ is within $(\eta, C)$ linear (or just $\eta$-linear) neighborhood of a geodesic
segment $\gamma$, where $\eta \ll 1$ and $C \ll \eta |\lambda|$, then we say that $\lambda$
\emph{admits a geodesic approximation} by $\gamma$.

\subsubsection{Notations used in split abelian-by-abelian groups}

Let $G=\mathbf{H} \rtimes \mathbf{A}$ stands for a non-degenerate, unimodular, split abelian-by-abelian group.  Fix a point $p \in G$.
We define the following:

\begin{itemize}
\item For a root class $[\alpha]$, we write $\mathit{l}_{[\alpha]}=\sum_{\xi \in [\alpha]} \xi$ be the sum of all roots in the equivalence
class.  Let $R_{G}$ be the set of all $\mathit{l}_{[\alpha]}$'s for $[\alpha]$ ranging over all root classes of $G$, and $G_{R}$ for the group of linear maps
that leaves $R_{G}$ invariant.  As $R_{G}$ is finite, $G_{R}$ is a subgroup of $O(n)$.

\item For $\alpha \in \triangle$ a root, we write $\alpha_{0} \in \mathbf{A}^{*}$ for the positive
multiple of $\mathit{l}_{[\alpha]}$ of unit norm with respect to the usual Euclidean inner product
on $\mathbf{A}$ and $\vec{v}_{[\alpha]} \in \mathbf{A}$ for the dual of $\alpha_{0}$.

\item Given $\vec{v} \in \mathbf{A}$, we define \begin{eqnarray*}
W_{\vec{v}}^{+} &=& \oplus_{\Xi(\vec{v}) >0} V_{\Xi} \\
W_{\vec{v}}^{-} &=& \oplus_{\Xi(\vec{v}) <0} V_{\Xi} \\
W_{\vec{v}}^{0} &=& \oplus_{\Xi(\vec{v})=0} V_{\Xi} \end{eqnarray*} \noindent We say a vector
$\vec{v} \in \mathbf{A}$ is \emph{regular} if
$\Xi(\vec{v}) \not=0$ for all roots $\Xi$, and a linear functional
$\ell \in \mathbf{A}^{*}$ is regular if its dual $\vec{v}_{\ell}$ is regular.

\item Let $\ell \in \mathbf{A}^{*}$, we define $W_{\ell}^{+}$, $W_{\ell}^{-}$,
$W_{\ell}^{0}$, as $W_{\vec{v}_ {\ell}}^{+}$, $W_{\vec{v}_ {\ell}}^{-}$, $W_{\vec{v}_
{\ell}}^{0}$ respectively, where $\vec{v}_{\ell} \in \mathbf{A}$ is the dual of $\ell$.

\item By \emph{the walls based at p}, we mean the set $p \bigcup_{\Xi} ker(\Xi)$.

\item The root kernels partition the unit sphere in $\mathbf{A}$ into into convex subsets called \emph{chambers}.  For vectors $\vec{u}, \vec{v}$ in the
interior of the same chamber $\mathfrak{b}$, $W^{+}_{\vec{u}}=W^{+}_{\vec{v}}$, and we define $W^{+}_{\mathfrak{b}}$ for this common
subspace of $\mathbf{H}$ and $W^{-}_{\mathfrak{b}}$ for its complement in $\mathbf{H}$
according to the root space decomposition, so that $\mathbf{H}=W^{+}_{\mathfrak{b}} \oplus W^{-}_{\mathfrak{b}}$.

\item By a \emph{geodesic segment through p}, we mean a set $p \overline{AB}$,
where $\overline{AB}$ is a directed line segment in $\mathbf{A}$.  By direction of a directed line segment in Euclidean space, we mean
a unit vector with respect to the usual Euclidean metric, and by direction of $p \overline{AB}$ we mean the direction of $\overline{AB}$.

\item For $i=2,3,.. rank(G)-1$, by \emph{$i$-hyperplane through p}, we mean a set $p S$, where $S \subset \mathbf{A}$ is an $i$-dimensional
linear subspace or an intersection between an $i$-dimensional linear subspace with a convex set.

\item Let $\pi_{A}:G  \longrightarrow \mathbf{A}$ be the projection
onto the $\mathbf{A}$ factor as $( \mathbf{x}, \mathbf{t} ) \mapsto \mathbf{t}$.

\item For each root $\alpha_{i}$, define $\pi_{\alpha_{i}}: G \longrightarrow  V_{\alpha_{i}} \rtimes \langle \vec{v}_{\alpha_{i}} \rangle $ as
$(\mathbf{x}_{1}, \mathbf{x}_{2}, \cdots \mathbf{x}_{|\triangle|}) \mathbf{t} \mapsto (\mathbf{x}_{i},
\alpha_{i}(\mathbf{t})\vec{v}_{\alpha_{i}})$.  We refer to negatively curved spaces
$V_{\alpha_{i}} \rtimes \langle \vec{v}_{\alpha_{i}} \rangle$ or $V_{[\alpha]} \rtimes \langle \vec{v}_{\alpha} \rangle$
as weight (or root) hyperbolic spaces.

\item For a regular vector $\vec{v} \in \mathbf{A}$, we define
$\Pi_{\vec{v}}: G \longrightarrow  \mathbf{H} \rtimes \langle \vec{v} \rangle$ as
$(\mathbf{x}, \mathbf{t}) \mapsto (\mathbf{x}, \langle \vec{v}, \mathbf{t} \rangle \vec{v})$,
where $\langle , \rangle$ is the standard inner product on $\mathbf{A}$.  In the rank 1 space
$\mathbf{H} \rtimes \langle \vec{v} \rangle$, the height function is given as
$\mathbf{H} \rtimes \langle \vec{v} \rangle  \stackrel{\pi_{\vec{v}}}{\longrightarrow}
\langle \vec{v} \rangle$

\item For a regular liner functional $\ell \in \mathbf{A}^{*}$ with unit norm, we define
$\Pi_{\ell}: G \rightarrow W^{-}_{\ell} \rtimes \mathbb{R}\vec{v}_{\ell}$ as
$(\mathbf{x}, \vec{t}) \mapsto ([\mathbf{x}]_{W^{-}_{\ell}}, \ell(\mathbf{t})
\vec{v}_{\ell})$.  \end{itemize}

\bold{Box associated to a compact convex set} Fix a net $\gothic{n}$ of $G$.
For $\alpha \in \triangle$, let $b(r) \subset V_{\alpha}$ be maximal product of intervals of size $r$, $[0,r]^{\mbox{dim}(V_{\alpha})}$.
Let $\Omega \subset \mathbf{A}$ be a convex compact set with non-empty interior whose barycenter is the identity
of $\mathbf{A}$.  We define the \emph{box associated to $\Omega$}, $\mathbf{B}(\Omega)$,
as the set $\prod_{j=1}^{\sharp} b(e^{\max(\alpha_{j}(\Omega))}) \Omega$.  We write

\begin{itemize}
\item $\mathcal{L}(\Omega)[m]$ (or $\mathcal{L}(\mathbf{B}(\Omega))[m]$) for the set of geodesics in $\mathbf{B}(\Omega)$
whose $\pi_{A}$ images begin and end at points of $\partial \Omega$
such that the ratio between its length and the diameter of $\Omega$
lies in the interval $[1/m, m]$.  This is a set of $\gothic{n}$ approximations of such geodesic segments, so is finite.

\item For $i=2,3,\cdots ,n$, write $\mathcal{L}_{i}(\Omega)[m_{i}]$ (or $\mathcal{L}_{i}(\mathbf{B}(\Omega))[m_{i}]$) for the set of $i$
dimensional hyperplanes in $\mathbf{B}(\Omega)$ such that the ratio between its diameter and the diameter of $\Omega$ lies in the
interval $[1/m_{i}, m_{i}]$.  This is a set of $\gothic{n}$ approximations of such bounded subsets in $i$-hyperplanes, so again is finite.

\item $\mathcal{P}(\Omega)$ (or $\mathcal{P}(\mathbf{B}(\Omega))$) for the set of points in $\mathbf{B}(\Omega)$.
That is $\mathcal{P}(\Omega)=\mathbf{B}(\Omega) \cap \gothic{n}$, so a finite set too.

\item  Let $S$ be an element of   $\bigcup_{i=2}^{n} \mathcal{L}_{i}(\Omega) \bigcup \mathcal{L}(\Omega) \bigcup
\mathcal{P}(\Omega)$.  We write $L(S)$, $L_{i}(S)$ for subset of $\mathcal{L}(\Omega)$, $\mathcal{L}_{i}(\Omega)$ contained or
containing $S$, and $P(S)$ for the subset of $\mathcal{P}(\Omega)$ contained in $S$. \end{itemize}

The following lemma shows that $G$ is amenable.  \begin{lemma} \label{boxes are folner}
Let $\Omega \subset \mathbf{A}$ be compact convex with non-empty interior.  Then,
$\mathbf{B}(r \Omega)$, $r \rightarrow \infty$ is a F\"{o}lner sequence. The volume ratio between $N_{\epsilon}(\partial (\mathbf{B}(r \Omega)))$
and $\mathbf{B}(r \Omega)$ is $O(\epsilon/ diam( \mathbf{B}(r \Omega))$\symbolfootnote[2]{because
the ratio of volumes of $\partial \Omega$ to $\Omega$ is roughly $\frac{1}{diam(\Omega)}$, and $r diam(\Omega)=diam(r
\Omega)$}.\end{lemma} \begin{proof} See Lemma \ref{Forste-boxes are folner} in \cite{Pg} \end{proof}

\begin{remark} \label{general folner}
The same calculation as above shows that for any set $\tilde{B}$ of the form $\Lambda \rtimes \Omega$,
where $\Lambda \subset \mathbf{H}$, $\Omega \subset \mathbf{A}$, the ratio of volumes of
$N_{\epsilon}(\partial \tilde{B})$ and that of $\tilde{B}$ is $O(\epsilon/
diam(\tilde{B}))$.   \end{remark}

\section{Shadows, slabs and coarsening} \label{shadow+slab}
We recall the following from \cite{Pg}.

\begin{theorem} \label{exisence of standard maps in small boxes}
Let $G$, $G'$ be non-degenerate, unimodular, split abelian-by-abelian Lie groups, and
$\phi:G \rightarrow G'$ be a $(\kappa, C)$ quasi-isometry.  Given $0< \delta, \eta < \tilde{\eta} < 1$, there exist numbers $L_{0}$, $m > 1$, $\varrho, \hat{\eta} < 1$ depending on $\delta$,
$\eta$, $\tilde{\eta}$ and $\kappa, C$ with the following properties: \smallskip

If $\Omega \subset \mathbf{A}$ is a product of intervals of equal size at least $mL_{0}$, then a tiling of $\mathbf{B}(\Omega)$ by isometric copies of $\mathbf{B}(\varrho \Omega)$
\[ \mathbf{B}(\Omega)= \bigsqcup_{i \in \mathbf{I}} \mathbf{B}(\omega_{i}) \sqcup \Upsilon \]

\noindent contains a subset $\mathbf{I}_{0}$ of $\mathbf{I}$ with relative measure at least $1-\nu$ such that

\begin{enumerate}

\item  For every $i \in \mathbf{I}_{0}$, there is a subset $\mathcal{P}^{0}(\omega_{i}) \subset \mathcal{P}(\omega_{i})$ of relative measure
at least $1-\theta$.

\item The restriction $\phi|_{\mathcal{P}^{0}(\omega_{i})}$ is within $\hat{\eta} diam(\mathbf{B}(\omega_{i}))$ Hausdorff neighborhood of a
standard map $g_{i} \times  f_{i}$. \end{enumerate}  Here, $\nu$, $\theta$ and $\hat{\eta}$ all approach zero as $\tilde{\eta}$, $\delta$ go to
zero. \end{theorem}


In this section, we focus on a particular standard map $g_{i} \times f_{i}$ supported on the subset $U_{i}$ of a good box $\mathbf{B}(\omega_{i})$, $i \in
\mathbf{I}_{0}$.  We will first show that the $f_{i}$ is affine for all ranks.  Then we will interpret its the constant and linear parts:
the linear part has to come from a finite set related to the geometry of $G'$, and the constant part depends on measure of certain subsets
in  $\mathbf{B}(\omega_{i})$.  We will drop the subscript $i$ from now on.

\subsection{Definitions}
In this subsection we define a list of objects that will be used for the remaining of this section.

\bold{Root class half planes }  A set of the form $p \mbox{   } \alpha_{0}^{-1}[-\infty, c]$ (resp.$p \mbox{   }\alpha_{0}^{-1}[c,\infty]$),
where $p \in \mathbf{H}$, $c \in \mathbb{R}$, is called a $[\alpha]$ negative
(resp. positive) half plane.   We write $\mathcal{H}^{-}_{[\alpha]}$ (resp. $\mathcal{H}^{+}_{[\alpha]}$) for the
set of $[\alpha]$ negative (resp. positive) half planes.  When we refer to a $[\alpha]$ half plane in a bounded set, we mean
$p \mbox{  } \alpha_{0}^{-1}([c,d])$, for some $p \in \mathbf{H}$, $c, d \in \mathbb{R}$.  We will also say that the length of this
$[\alpha]$ half plane is $|c-d|$. (\textit{remember here that the domain of roots are is
$\mathbb{A}$, not the entire group $G$})


\bold{Upper root boundary}  We define the \emph{upper boundary of root class $[\alpha]$ }, $\partial^{+}_{[\alpha]}$, as the quotient of
$\mathcal{H}^{+}_{[\alpha]}$ under the equivalence relation of bounded Hausdorff distance.

If two positive $[\alpha]$ half-planes $E_{p}$, $E_{q}$ where $p,q \in \mathbf{H}$, are bounded Hausdorff distance apart, then $p$, $q$ can only
differ by $V_{[\alpha]}$ coordinates.  This means each equivalence class can be identified with $V_{[\alpha]}$, and the collection of all
equivalence classes, $\partial^{+}_{[\alpha]}$, can be identified with $\oplus_{[\beta] \not= [\alpha]}
V_{[\beta]}$.

\bold{Lower root boundary}  We say two $[\alpha]$ negative half planes $H_{p}$, $H_{q}$ are equivalent if there is a sequence $H_{i} \in
\mathcal{H}^{-}_{[\alpha]}$ such that $H_{0}=H_{p}$, $H_{q}=H_{n}$ and any two successive $H_{i}$'s intersect at an unbounded convex set.
This is an equivalence relation because if $H_{p}$ is equivalent to $H_{q}$, and $H_{q}$ is equivalent to $H_{r}$, then concatenation of the
sequences used to connect the two pairs is a sequence that connects $H_{p}$ and $H_{r}$.
We define the \emph{lower boundary of $[\alpha]$} $\partial^{-}_{[\alpha]}$ as the quotient of $\mathcal{H}^{-}_{[\alpha]}$ under this
equivalence relation. \smallskip

We see that if $H_{p}$, $H_{q} \in \mathcal{H}^{-}_{[\alpha]}$ based at $p,q \in \mathbf{H}$ have non-empty
intersection, then $p$ and $q$ cannot differ by $V_{[\alpha]}$ coordinate.  On the other hand, if $p$ and $q$ differ only
in some $V_{[\beta]}$ coordinate, where $[\beta] \not= [\alpha]$, then $H_{p} \cap H_{q} \not =\emptyset$, so $H_{p}$ is equivalent
 to $H_{q}$ in this case.  This way, we see that the equivalence class containing $H_{p}$, $p \in \mathbf{H}$ are all those $H_{q} \in
\mathcal{H}_{[\alpha]}$, where $q \in \mathbf{H}$ differ from $p$ by some elements of $\oplus_{[\beta] \not=[\alpha]}
V_{[\beta]}$, and consequently, $\partial^{-}_{[\alpha]}$ an be identified with $V_{[\alpha]}$.   \smallskip

\bold{Measures on lower root boundaries}  If $p \in G$, we write $\pi^{-}_{[\alpha]}(p) \subset \partial^{-}_{[\alpha]}$ (resp.
$\pi^{+}_{[\alpha]}(p) \subset \partial^{+}_{[\alpha]}$) for the set of equivalence classes, each containing a minimal negative
(resp. positive) $[\alpha]$ half planes through a point that is at most distance $\rho$ away from $p$, where $\rho$ is the scale of
discretization.  Since $V_{[\alpha]}$ is the direct sums of $V_{\Xi}$, where $\Xi \in [\alpha]$, we will write $\pi^{-}_{\sigma}(p)$,
where $\sigma \in [\alpha]$, for the $V_{\sigma}$ coordinate of $\pi^{-}_{[\alpha]}(p)$.  For $A \subset G$,
we write $\pi^{*}_{[\alpha]}(A)= \bigcup_{p \in A} \pi^{*}_{[\alpha]}(p)$, where $* \in \{ + , - \}$. \smallskip

Since $\partial^{-}_{[\alpha]}$ is a homogeneous space (the subgroup $V_{[\alpha]} \subset \mathbf{H}$ acts faithfully and transitively on it),
it admits a Haar measure.  We normalize this measure $| \centerdot |$ by requiring that for each $\sigma \in [\alpha]$,
\begin{equation}\label{measure on boundary}
\left| \pi^{-}_{\sigma}(p) \right| e^{-\sigma(p)} =1 , \mbox{       } \forall p \in G \end{equation}

\bold{Upper and lower boundaries of a linear functional } We call the intersection of half planes corresponding to
two perpendicular linear functionals, a quarter plane.

We now define an equivalence relation on $\mathcal{H}^{+}_{\ell}$ (resp. $\mathcal{H}^{-}_{\ell}$) as follows.  Two
positive (resp. negative) $\ell$ half planes $H_{p}$, $H_{q}$ are equivalent if there is a sequence of
$H_{i} \in \mathcal{H}^{+}_{\ell}$ such that $H_{0}=H_{p}$, $H_{n}=H_{q}$ and the intersection between any two
successive $H_{i}$'s does not contain a quarter plane.  This is an equivalence relation because
if $H_{p}$ is equivalent to $H_{q}$, and $H_{q}$ is equivalent to $H_{r}$, then the
concatenation of the sequences used to connect the two pairs is a sequence that
establishes equivalence between $H_{p}$ and $H_{r}$.  We see that if the positive (resp. negative) $\ell$ half planes
$H_{p}, H_{q}$ based at $p,q \in \mathbf{H}$ are equivalent, then $p,q$ differ by an element of
$W^{-}_{\ell}$. (resp. $W^{+}_{\ell}$)

We define the \emph{upper boundary of $\ell$}, $\partial^{+}_{\ell}$, (resp. \emph{lower boundary of $\ell$},
$\partial^{-}_{\ell}$) as the quotient of $\mathcal{H}^{+}_{\ell}$ (resp.
$\mathcal{H}^{-}_{\ell}$) under this equivalence relation.  In light of the forgoing
discussion, we see that $\partial^{+}_{\ell}$ (resp. $\partial^{-}_{\ell}$) can be
identified with $W^{+}_{\ell}$.  (resp. $W^{-}_{\ell}$)

\bold{Measure on upper and lower boundaries of a linear functional}  Let $\ell$ be a
generic linear functional.  Since $\partial^{+}_{\ell}$ can be identified with
$W^{+}_{\ell}$ which itself is a direct sum of $V_{[\Xi]}$ where $\ell(\vec{v}_{\xi})
>0$, and each of $V_{[\Xi]}$ can be identified with $\partial_{[\Xi]}$, we can identify
$\partial^{+}_{\ell}$ with $\prod_{[\Xi]: \ell(\vec{v}_{\Xi}) >0} \partial_{[\Xi]}$, and
equip it with the product measures on the constituent root boundaries.  The same
procedure can be applied to $\partial^{-}_{\ell}$ to turn into a measure space.
\smallskip

If $p \in G$, we write $\pi^{-}_{\ell}(p) \subset \partial^{-}_{\ell}$ (resp.
$\pi^{+}_{\ell}(p) \subset \partial^{+}_{\ell}$) for the set of equivalence classes, each
containing a smallest negative (resp. positive) $\ell$ half planes through a point that
is at most distance $\rho$ away from $p$, where $\rho$ is the scale of discretization.
For $A \subset G$, and $* \in \{ +, - \}$, $\pi^{*}_{\ell}(A)$ is the union of $\pi^{+}_{\ell}(p)$, where $p$
ranges over all points of $A$.

\bold{Branching constant}  The branching constant $\mathbf{b}_{[\alpha]}$ of root class
$[\alpha]$ is the number such that $e^{\mathbf{b}_{[\alpha]} L}$ represents the number of $[\alpha]$ half planes
of length $L$ leaving a point.  It equals $\mathit{l}_{[\alpha]}/ \alpha_{0}$. \medskip

The branching constant $\mathbf{b}_{\ell}$, of a generic linear functional $\ell$, is a number such
that $e^{\mathbf{b}_{\ell} L}$ represents the number of $\ell$ half planes of length
$L$ leaving a point.  Its value is given by
\begin{equation}\label{branching constant defined}
\mathbf{b}_{\ell} = \sum_{\sigma: \sigma(\vec{v}_{\ell}) > 0} \sigma(\vec{v}_{\ell})
= \sum_{\sigma: \sigma(\vec{v}_{\ell}) < 0} \sigma(\vec{v}_{\ell}) \end{equation}

If $\tilde{\ell}$ is a linear functional whose norm is not 1,  we will write
$\mathbf{b}_{\tilde{\ell}}$ for $\mathbf{b}_{\tilde{\ell}/ \| \tilde{\ell} \|}$.

\bold{Distances on lower root boundaries}  Given $p,q \in \partial^{-}_{[\alpha]} \sim V_{[\alpha]}$,
let $t_{p,q}$ be the minimal $t \in \mathbb{R}$ such that there exists negative $[\alpha]$ half
planes in the equivalence class of $p$ and $q$ that are distance 1 (or $\rho$ if the scale of discretization is not $1$) at sets whose
$\pi_{A}$ projection is $\ell_{[\alpha]}^{-1}(t)$.

Fix a positive number $c$, we define a \emph{psudo distance} $D_{[\alpha]}$ between $p,q$
as \[ D_{[\alpha]}(p,q) = e^{c t_{p,q}} \]

\noindent Different choice of $c$ leads to quasi-symmetric equivalent metric.  In this way, the space
$(\partial^{-}_{[\alpha]}, D_{[\alpha]})$ becomes those whose quasi-conformal maps are studied in \cite{Dy}.



\bold{Shadows, slabs  }  Using the same root class as before, we define any subset $H$ of a left translate of $\oplus_{[\beta] \not=
[\alpha]}V_{[\beta]} \rtimes ker(\alpha_{0})$ a \emph{$[\alpha]$ block}.  Note that for any element $\sigma \in [\alpha]$,
$\sigma(H)$ is well-defined.  For $\rho>1$, we define the \emph{$\rho$-shadow of $H$}, $Sh(H, \rho)$, as the union of smallest
negative $[\alpha]$ half planes containing a point in $N_{\rho}H$.  For $h_{2} < h_{1} < \alpha_{0}(H)$, we define a \emph{slab of
$H$}, denoted by $Sl^{1}_{2}(H)$ as the intersection between $Sh(H,\rho)$ with $\alpha_{0}^{-1}([h_{2},h_{1}])$.  That
is, it's the subset of $Sh(H,\rho)$ that whose $\alpha_{0}$ values lies in between $h_{2}$ and $h_{1}$.

\bold{Generalized slabs  } For $E_{-} \subset \partial^{-}_{[\alpha]}$, $E^{+} \subset \partial^{+}_{[\alpha]}$,
$K \subset ker(\alpha_{0})$, $h_{2} < h_{1}$, we call a set \[ S(E_{-}, E^{+}, K, h_{2}, h_{1}) = \{ \left( \mathbf{x}_{[\alpha]},
(\mathbf{x}_{[\beta]})_{[\beta] \not= [\alpha]}, \mathbf{t} \right):  \mathbf{t} \in [h_{2},
h_{1}]K, \mbox{   } \mathbf{x}_{[\alpha]} \in E_{-}, \mbox{   } (\mathbf{x}_{[\beta]})_{[\beta] \not= [\alpha]} \in E^{+} \} \]

\noindent a \emph{generalized $[\alpha]$ slab}.  This generalizes the definition of slabs defined in the previous paragraph.


\bold{Coarsening } We define a process \emph{coarsening} as follows.  For $h \in
\mathbb{R}$, $E^{+} \subset \partial^{+}_{[\alpha]}$, \emph{the coarsening of $E^{+}$ by
$h$}, $\mathcal{C}_{h}(E^{+})$ is defined as the subset of $\partial^{+}_{[\alpha]}$ consisting of those equivalence classes that contains a
positive $[\alpha]$ half plane that has a non-empty intersection with a positive $[\alpha]$ half plane whose equivalence
class belongs to $E^{+}$ at a set whose $\pi_{A}$ projection belongs to $\alpha_{0}^{-1}[h, \infty]$. \smallskip

Similarly, for $E_{-} \subset \partial^{-}_{[\alpha]}$ a subset of the lower boundary, the coarsening of $E_{-}$ by $h$, $\mathcal{C}_{h}^{E_{-}}$,
is defined to be those equivalence classes of negative $[\alpha]$ half planes that contains an element which intersects non-empty with a negative
$[\alpha]$ half plane whose equivalence class belongs to $E_{-}$ at a set whose $\pi_{A}$
projection is a subset of $\alpha_{0}^{-1}([-\infty, h])$\symbolfootnote[2]{This is the same as the set of points in $\partial^{-}_{[\alpha]}$
that is distance $e^{h}$ from a point in $E_{-}$.}. \smallskip

Observe that as long as $h_{3} \leq h_{2}$, $h_{4} \geq h_{1}$, we have
\[ S \left( E_{-}, E^{+}, h_{2}, h_{1} \right) = S \left( \mathcal{C}_{h_{3}}(E_{-}),\mathcal{C}_{h_{4}}(E^{+}), h_{2}, h_{1} \right) \]

\begin{lemma}\label{lemma 3-5}
The number of $[\alpha]$ planes in $S=S(  \mathcal{C}_{h_{3}}(E_{-}), \mathcal{C}_{h_{4}}(E^{+}), K, h_{2}, h_{1})$ is comparable to
\[ \frac{Vol(S)}{|K| (h_{1}-h_{2}) } e^{\mathbf{b}_{[\alpha]}(h_{1}-h_{2})}  \]
\noindent That is, it is compatible to the area of the cross-section times
$e^{\mathbf{b}_{[\alpha]}(h_{1}-h_{2})}$. \end{lemma} \begin{proof} Counting the number of
$[\alpha]$ half planes really means you count the number of geodesics in $V_{[\alpha]} \rtimes \mathbb{R}(\vec{v}_{[\alpha]})$.
In this way, we see that the first term of the product is the size of a cross section in this projection image, and the product is the
no. of geodesics in the range. \end{proof}

\subsection{Improving almost product map}
Recall our setting from Theorem \ref{exisence of standard maps in small boxes}.

\[ \phi: \mathbf{B}(\omega) \rightarrow G' \]

\noindent is a quasi-isometry such that on a subset $U_{*} \subset \mathbf{B}(\omega)$ of relative measure at least $1-\theta$,
the restriction $\phi|_{U_{*}}$, is within $\hat{\eta} diam(\mathbf{B}(\omega))$ of a standard map $\hat{\phi}=g \times f$. \smallskip

If $U_{*}=\mathbf{B}(\omega)$, then it is clear that the image of a slab is a slab, and image of a generalized slab is also a
generalized slab.  However as $U_{*}$ is generally not $\mathbf{B}(\omega)$, it is not clear that for a $[\alpha]$ block $H$, and
$h_{2}< h_{1} < \alpha_{0}(H)$, there is an obvious
relation between $\phi(Sl_{2}^{1}(H))$ and $\hat{Sl}_{2}^{1}(H)$, which is defined as $S(g(\pi^{-}_{[\alpha]}(H)), g(\pi^{+}_{[\alpha]}(H)),
f(\pi_{A}(H) \times [h_{2},h_{1}]))$, other than
\begin{equation*} \phi(Sl_{2}^{1}(H) \cap U_{*}) \subset N_{\hat{\eta} diam(\mathbf{B}(\omega))}
\hat{\phi}(Sl_{2}^{1}(H) \cap U_{*}) \subset N_{\hat{\eta} diam(\mathbf{B}(\omega))} \hat{Sl}_{2}^{1}(H) \end{equation*}

\noindent We will show that by restricting to certain subset of $Sl_{2}^{1}(H)$ whose
coarsened version lies almost entirely in $U_{*}$, a reversal inclusion can be obtained. \smallskip

To this end, fix a $[\alpha]$ block $H$ in $\mathbf{B}(\omega)$, and let $h=\alpha_{0}(H)$.  As $f$ preserves
foliations by root kernels, the $f$ image of $ker(\alpha_{0})$ is the kernel of $\Xi_{0} \in \mathbf{A}^{*}$
for some root class $[\Xi]$.  We identify $\mathbf{A}/ker(\alpha_{0})$ with $\mathbb{R}$ by taking the $\alpha_{0}$ value of a
coset, and do so similarly for $\mathbf{A}/ker(\Xi_{0})$.  Then $f$ induces a map  $q: \mathbb{R} \rightarrow \mathbb{R}$ by sending $t$
to the $\Xi_{0}$ value of $f(t \mbox{ } ker(\alpha_{0}))$.  When the rank of $G$ is 1, this is just $f$. \smallskip

Now, fix heights  $h_{2} < h_{1} < h$, and define
\[ \hat{Sl}_{2}^{1}(H) =S( \mathcal{C}_{q(h_{1})} g(\pi^{-}_{[\alpha]}(H)), \mathcal{C}_{q(h_{2})} g(\pi^{+}_{[\alpha]}(H)),
f(\pi_{A}(H) \times [h_{2},h_{1}]))  \]

\noindent In the rest of this section, we write $B$ for the box $\mathbf{B}(\omega)=\left(\prod_{[\sigma]}A_{[\sigma]} \right) \omega$ where
$A_{[\sigma]} \subset V_{[\sigma]}$, $\omega \subset \mathbf{A}$. \smallskip

We refer to the parameters $\delta, \eta, \tilde{\eta}$ in Theorem \ref{exisence of standard maps in small
boxes} as `initial data', and show now that most slabs inherit the property that `intersection with $U_{*}$ possess relative large
measure'  For the remaining subsections, whenever $\gg $ appears, it means the ratio of the two quantities is more than $2\kappa$.

\begin{lemma} \label{lemma 4.1}
Given $0< \beta \ll \beta' \ll  1$, there exist constants $c_{1}$, $c_{2} \ll 1$ depending on our initial data and a subset
$E_{**} \subset \partial^{-}_{[\alpha]}(B)$ of relative large measure such that whenever $H$ is a $[\alpha]$ block which is at least
$2\kappa \beta' diam(\mathbf{B}(\omega))$ away from $\partial B$ and $\pi^{-}_{[\alpha]}(H) \cap E_{**}
\not=\emptyset$, then $|h_{1}(H) -h_{2}(H)| |\pi_{A}(H)| > \beta |\omega| $ implies
\begin{equation}\label{slab in domain} |Sl_{2}^{1}(H) \cap U_{*}| \geq (1-c_{2}) |Sl_{2}^{1}(H)|
\end{equation} \end{lemma}

\begin{proof}
Let $c_{2}$ be a constant to be chosen later.  Let $E_{1} \subset \partial^{-}_{[\alpha]}(B)$ be the subset such that for $x \in E_{1}$,
there exists a $[\alpha]$ block $H_{x}$ such that $x \in I_{x}=\pi^{-}_{[\alpha]}(H_{x})$ and equation (\ref{slab in domain}) fails. Then we
have a cover of $E_{1}$ by intervals $I_{x}$.  By Vitali covering there is a subset of $I_{k}$'s such that
$\sum_{k} \left| I_{k} \right| \geq 1/5 |E_{1}|$, and whose elements are strongly disjoint i.e. for $j \not= k$,
$d(I_{j}, I_{k}) \geq 1/2 \max(|I_{j}|)$, which means that the corresponding $[\alpha]$ block $H_{k}$'s are also disjoint.  By
construction $\left| Sl_{2}^{1}(H_{k}) \cap U^{c}_{*} \right| \geq c_{2} |Sl_{2}^{1}(H_{k}) |$. \smallskip

Summing over $k$ yields
\[ \left| B \cap U_{*}^{c} \right| \geq c_{2} \sum_{k} \left| Sl_{2}^{1}(H_{k}) \right|
\geq \frac{c_{2}}{2} \sum_{k} |h_{1}(H_{k}) - h_{2}(H_{k})| |\pi_{A}(H_{k})| |I_{k}| |\pi^{+}_{[\alpha]}(H_{k})| \]

\noindent As $| \pi^{+}_{[\alpha]}(H_{k}) |= | \pi^{+}_{[\alpha]}(B) |=\prod_{[\sigma] \not=[\alpha]} |  A_{[\sigma]}  |$,
 and $| B \cap U_{*}^{c} |  \leq \theta | \omega | | \pi^{-}_{[\alpha]}(B) | | \pi^{+}_{[\alpha]}(B) |$, we obtain
\[ |E_{1}| \leq 5 \sum_{k} |I_{k}| \leq \frac{10 \theta}{ \beta c_{2}} \left|  \partial^{-}_{[\alpha]}(B) \right| \]
\noindent Now choose $c_{2}$ appropriately.  \end{proof}

Note that as $\hat{\phi}$ is only partially defined, the definition of
$\hat{Sl}_{2}^{1}(H)$ sort of `fills up' $\hat{\phi}(Sl_{2}^{1}(H))$.  Therefore it is
not clear that most of $\hat{Sl}_{2}^{1}(H)$ lie close to the $\phi(U_{*})$.  We now
introduce some notations needed to describe a subset of $\hat{Sl}_{2}^{1}(H)$ which can
be controlled more easily.  Given a set $D \subset B$ and $A_{[\Xi]} \subset
\pi^{+}_{[\Xi]}(B)$, we write $A_{[\Xi]} \cap D$ for the intersection between $A_{[\Xi]}$ and $\pi^{+}_{[\Xi]}(D)$. We then define

\[ \tilde{Sl}_{2}^{1}(H, D):= S(\mathcal{C}_{q(h_{1})}(g(\pi^{-}_{[\alpha]}(H))),
\mathcal{C}_{q(h_{2})}(g(\pi^{+}_{[\alpha]}(H) \cap D)), f(\pi_{A}(H) \times [h_{2}, h_{1}]) ) \]

\noindent This next lemma says that for certain `slimmed down' version of
$\hat{Sl}_{2}^{1}(H)$, $\tilde{Sl}_{2}^{1}(H,D)$, most of them lies in $\phi(U_{*})$.

\begin{lemma} \label{lemma 4.2}
Given $\hat{\eta} < \beta \ll \beta' \ll \beta'' < 1$, there exist constants $c_{3}, c_{4}$ depending
on our initial data and a subset $E_{*} \subset \pi^{-}_{[\alpha]}(B)$ of relative large
measure with the following properties. \smallskip

Let $H_{0}$ be a $[\alpha]$ block in $B$ such that
\begin{enumerate}\renewcommand{\labelenumi}{\alph{enumi}.}
\item The distance between $H_{0}$ and $\partial B$ is at least $4\kappa \beta'' diam(\mathbf{B}(\omega))$.
\item The intersection $\pi^{-}_{[\alpha]}(H_{0}) \cap E_{*}$ is not empty.  \end{enumerate}

Suppose $H$ is a $[\alpha]$ block in $S=U_{*} \cap S(\pi^{-}_{[\alpha]}(H_{0}), \pi^{+}_{[\alpha]}(H_{0}),
\pi_{A}(H_{0}), \ell_{[\alpha]}(H_{0}), \ell_{[\alpha]}(H_{0}) + \beta'' diam(\mathbf{B}(\omega)))$ such that $\pi^{-}_{[\alpha]}(H)$ has
non-empty intersection with $E_{*}$.  Then, $\beta' |\omega| > |h_{1}(H)-h_{2}(H)||\pi_{A}(H)| > \beta
|\omega|$ implies
\begin{equation}\label{slab in range}
\left|  \tilde{Sl}_{2}^{1}(H) \cap \phi(U_{*})  \right| \geq (1-c_{4}) \left| \tilde{Sl}_{2}^{1}(H) \right| \end{equation} \noindent where $\tilde{Sl}_{2}^{1}(H)=\tilde{Sl}_{2}^{1}(H, S)$, and
$c_{4}$ approaches aero as our initial data approach zero's. \end{lemma}

\begin{proof}
The set $E_{*}$ will be constructed from $E_{**}$ that appeared in Lemma \ref{lemma 4.1}.

We first show that $\tilde{Sl}_{2}^{1}(H) \subset \phi(B)$.  Recall that $H_{0}$ is more
than $4 \kappa^{2} \beta'' diam(\mathbf{B}(\omega))$ away from the boundary of $B$.  This means that
$S$ is also more than $4 \kappa^{2} \beta'' diam(\mathbf{B}(\omega))$ away from $\partial B$.  By
definition  \begin{equation*}
 \tilde{Sl}_{2}^{1}(H) = S(\mathcal{C}_{q(h_{1})}(g(\pi^{-}_{[\alpha]}(H))),
\mathcal{C}_{q(h_{2})} g(\pi^{+}_{[\alpha]}(H) \cap S), f(\pi_{A}(H) \times [h_{2}(H), h_{1}(H)] )) \end{equation*}

By assumption, $|h_{1}(H)-h_{2}(H)| \leq \beta' \frac{|\omega|}{|\pi_{A}(H)|}$, so
$|h_{1}-h_{2}| \leq \beta' diam(\mathbf{B}(\omega))$.  Take $q_{0} \in \tilde{Sl}_{2}^{1}(H)$.  Then
$q_{0}$ is no further than $\beta' diam(\mathbf{B}(\omega))$ away from a point $q$ in the
generalized $[\alpha]$ slab \newline $S(g(\pi^{-}_{[\alpha]}(H)), g(\pi^{+}_{[\alpha]}(H) \cap S),
f(\pi_{A}(H)\times [h_{2}(H), h_{1}(H)] ))$.  This means that there is $p \in S \subset U_{*}$ such that
$\pi^{+}_{[\alpha]}(q)=\pi^{+}_{[\alpha]}(\hat{\phi}(p))$, so $d(q, \hat{\phi}(p)) \leq \beta' diam(\mathbf{B}(\omega))$.

By definition, $\hat{\phi}(p)$ lies on a $[\Xi]$ half plane that is $\hat{\eta} diam(\mathbf{B}(\omega))$ away from the image of a
$[\alpha]$ half-plane containing $p$. This means $d(q,\phi(S)) \leq 2\beta'
diam(\mathbf{B}(\omega))$, and therefore $d(q_{0}, \phi(S)) \leq 3 \beta' diam(\mathbf{B}(\omega))$. Since $d(S,
\partial B) > 4 \beta'' diam(\mathbf{B}(\omega))$,  it follows that $q_{0} \in \phi(B)$. \smallskip

Let $c_{3}$ be a constant to be chosen later.  Let $E_{2} \subset \partial^{-}_{[\alpha]}(B) - E_{1}$ be such that for $x \in E_{2}$ there is a $[\alpha]$ block $H_{x}$ such that $x \in I_{x}=\pi^{-}_{[\alpha]}(H_{x})$ and
equation (\ref{slab in range}) fails. Thus we have a cover of $E_{2}$ by intervals $I_{x}$.  By Vitali covering, we can find a
subset of $I_{k}$'s such that the the inequality opposite to above holds for each $H_{k}$
(instead of $H$) and $\sum_{k} |I_{k}| \geq (1/5) |E_{2}|$, and that $I_{k}$'s are
strongly disjoint\footnote{Note that we cannot proceed in the same manner as Lemma
\ref{lemma 4.1} from this point on because the $\partial^{+}_{[\alpha]}(S)$ is not a
full set: it's the intersection of this with something else, so we need to do something
else first. }.  (That is, for $j \not= k$, $d(I_{j}, I_{k}) \geq 1/2 \max(|I_{j}|,
|I_{k}|)$. In particular, this means $\tilde{Sl}_{2}^{1}(H_{k})$'s are disjoint as well. \medskip

We now claim that \begin{equation} \label{disjoint} \phi(Sh(H_{k}, O(1))^{c} \cap U_{*})
\cap \tilde{Sl}_{2}^{1}(H_{k}) = \emptyset
\end{equation}

\noindent Suppose not, then there is a $p \in Sh(H_{k}, O(\rho_{1}))^{c} \cap U_{*}$ such
that $\phi(p) \in \tilde{Sl}_{2}^{1}(H_{k})$.  The latter means that
$\pi^{-}_{[\alpha]}(\phi(p)) \in \mathcal{C}_{q(h_{1}(H))} g(\pi^{-}_{[\alpha]}(H_{k}))$, so
there is a $p' \in Sh(H_{k}, O(1)) \cap U_{*}$ such that
\begin{equation}\label{close}
t_{\pi^{-}_{[\Xi]}(\phi(p)), \pi^{-}_{[\Xi]}(\hat{\phi}(p'))} \leq q(h_{1}(H_{k})) + O(\hat{\eta} diam(\mathbf{B}(\omega)))  \end{equation}

\noindent However the fact that $p \in Sh(H_{k},O(1))^{c}$ and $p' \in Sh(H_{k}, O(1))$
means that for
\[  t_{\pi^{-}_{[\alpha]}(p), \pi^{-}_{[\alpha]}(p')}   > \alpha_{0}(H_{k})+ O(1) \]

\noindent which in turns means that
\[ t_{\pi^{-}_{[\Xi]}(\phi(p)), \pi^{-}_{[\Xi]}(\hat{\phi}(p'))}   > q(\alpha_{0}(H_{k})) + O(\hat{\eta} diam(\mathbf{B}(\omega)))
> q(h_{1}(H_{k})) + O(\hat{\eta} diam(\mathbf{B}(\omega))) \] \noindent  contradicting (\ref{close}).  So equation (\ref{disjoint}) holds. \\

Next we show that
\begin{equation} \label{measure of tilde}
\tilde{Sl}_{2}^{1}(H_{k}) \cap N_{O(\hat{\eta} diam(\mathbf{B}(\omega)))}
\hat{\phi}(Sl_{2}^{1}(H_{k}))^{c} \subset \phi(U_{*}^{c}) \end{equation}

\noindent Since $U_{*}^{c}$ has relative small measure in $B$, this means that the
measure of left hand side over all $k$'s from the strongly disjoint family remains small,
making the measure of $\sum_{k} N_{O(\hat{\eta} diam(\mathbf{B}(\omega)))} \hat{\phi}(Sl_{2}^{1}(H_{k}))$ a
lower bound for $\sum_{k} |\tilde{Sl}_{2}^{1}(H_{k})|$.

Suppose (\ref{measure of tilde}) is not true.  Then there is a $p \in
\tilde{Sl}_{2}^{1}(H_{k})$, such that $p \in N_{O(\hat{\eta} diam(\mathbf{B}(\omega)))} \phi(Sl_{2}^{1}(H_{k}))^{c}$ and $p \in \phi(U_{*})$.
However, by equation (\ref{disjoint}), the last two conditions means that $p \not \in \tilde{Sl}_{2}^{1}(H_{k})$, which contradicts assumption
that $p \in \tilde{Sl}_{2}^{1}(H_{k})$. \medskip

By choice, $\tilde{Sl}_{2}^{1}(H_{k})$'s satisfy the inequality opposite to equation
(\ref{slab in range}), so $|\tilde{Sl}_{2}^{1}(H_{k}) \cap \phi(U_{*}^{c}) | \geq c_{4}
|\tilde{Sl}_{2}^{1}(H_{k})|$.  As all the $\tilde{Sl}_{2}^{1}(H_{k})$'s are disjoint,
summing over $k$ yields \begin{eqnarray*}
|\phi(U_{*}^{c})| & \geq &  c_{4} \sum_{k} |\tilde{Sl}_{2}^{1}(H_{k})| \geq  c_{4}(1-\theta)\sum_{k} N_{O(\hat{\eta} diam(\mathbf{B}(\omega)))}
\hat{\phi}(Sl_{2}^{1}(H_{k}))\\
& \geq & c_{4}(1-\theta)(1-O(\hat{\eta})) (1-c_{2}) \sum_{k} |Sl_{2}^{1}(H_{k})| \\
&\geq& c_{4}(1-\theta)(1-O(\hat{\eta})) (1-c_{2}) \sum_{k} |\pi_{A}(H_{k})| |h_{1}(H_{k})-h_{2}(H_{k})| |I_{k}| |\pi^{+}_{[\alpha]}(H_{k})|
\end{eqnarray*}

\noindent The second line comes from Lemma \ref{lemma 4.1} and the fact that the volume ratio between $\epsilon$ neighborhood of $Sl_{2}^{1}(H)$ and $Sl_{2}^{1}(H)$ is
$1+O(\epsilon)$ in the second line, as noted in Remark \ref{general folner}. \smallskip

Since $\pi^{+}_{[\alpha]}(H_{k}) =\pi^{+}_{[\alpha]}(B)$, and $|\phi(U_{*}^{c})| \leq \theta |\omega| |\pi^{+}_{[\alpha]}(B)|
|\pi^{-}_{[\alpha]}(B)|$, we obtain
\begin{equation*} |E_{2}| \leq 5 \sum_{k} |I_{k}| \leq \frac{5 \theta}{c_{4}\beta (1-\theta)(1-O(\hat{\eta}))(1-c_{2})} |\pi^{-}_{[\alpha]}(B)| \end{equation*}  \end{proof}

\begin{corollary} \label{cross area preservation}
Let $H$ be a $[\alpha]$ block in $\mathbf{B}(\omega)$ satisfying the hypothesis of Lemma \ref{lemma 4.1} and \ref{lemma 4.2}, and $S$ as in Lemma \ref{lemma 4.2}.
Let $w_{1}, w_{2} \in \mathbb{R}$ be such that $\beta' |\omega| \geq |w_{1}-w_{2}| |\pi_{A}(H)| \geq \beta|\omega|$.  Then

\begin{equation} \label{upper bound}
|\mathcal{C}_{w_{1}}(g(\pi^{-}_{[\alpha]}(H)))| |\mathcal{C}_{w_{2}}(g(\pi^{+}_{[\alpha]}(H) \cap S)) | \geq
d |\pi^{-}_{[\alpha]}(H)| | \pi^{+}_{[\alpha]}(H) | \end{equation}
\noindent and
\begin{equation}\label{lower bound}
|\mathcal{C}_{w_{1}}(g(\pi^{-}_{[\alpha]}(H)))| |\mathcal{C}_{w_{2}}(g(\pi^{+}_{[\alpha]}(H) \cap S)) | \leq
b |\pi^{-}_{[\alpha]}(H)| | \pi^{+}_{[\alpha]}(H) | \end{equation}

\noindent where $d$ and $b$ depend only on $\kappa$, $C$.  \end{corollary}

\begin{proof}
Note that from the structure of $U$ and the fact that $\phi$ is a quasi-isometry, it
follows that for $z_{1}, z_{2} \in \pi_{A}(\mathbf{B}(\omega))$ we have
\[ \frac{1}{2\kappa} \left|  z_{1}-z_{2}  \right| - \hat{\eta} diam(\mathbf{B}(\omega)) \leq \left|  q(z_{1})-q(z_{2})  \right|
\leq 2\kappa \left|  z_{1}-z_{2}  \right| + \hat{\eta} diam(\mathbf{B}(\Omega)) \]

\noindent This means that $q$ is essentially monotone, so there exists $h_{1},h_{2}$ such
that $q(h_{1}(H))=w_{1}$, and $q(h_{2}(H))=w_{2}$.  We now apply Lemma \ref{lemma 4.1}
and \ref{lemma 4.2} to the resulting $Sl_{2}^{1}(H)$ and $\tilde{Sl}_{2}^{1}(H)$.
\smallskip

By equation (\ref{slab in range}) in Lemma \ref{lemma 4.2}, we know that
$|\tilde{Sl}_{2}^{1}(H) \cap \phi(U_{*}) | \geq (1-c_{4}) |\tilde{Sl}_{2}^{1}(H)|$.  The
structure of a standard map means that the ratio of measures of $\tilde{Sl}_{2}^{1}(H)
\cap \phi(U_{*})$ to that of $Sl_{2}^{1}(H) \cap U_{*}$ lies in $[1/2\kappa, 2\kappa]$.
These facts together shows that
\begin{equation*}
(1-c_{4}) |\tilde{Sl}_{2}^{1}(H)|  \leq  |\tilde{Sl}_{2}^{1} \cap \phi(U_{*})| \leq 2\kappa |Sl_{2}^{1}(H) \cap U_{*}|
 \leq  2\kappa |Sl_{2}^{1}(H)| \end{equation*}

\noindent Similarly,
\begin{equation*}
(1-c_{2}) |Sl_{2}^{1}(H)| \leq |Sl_{2}^{1}(H) \cap U_{*}| \leq 2\kappa |\tilde{Sl}_{2}^{1}(H) \cap \phi(U_{*}) |
\leq 2\kappa |\tilde{Sl}_{2}^{1}(H)| \end{equation*}

The claims now follow from the following volume formula.
\[ |Sl_{2}^{1}(H)|=|\pi^{-}_{[\alpha]}(H)| |\pi^{+}_{[\alpha]}(H) | |\pi_{A}(H)||h_{2}-h_{1}| \]
\[ |\tilde{Sl}_{2}^{1}(H)|= |\mathcal{C}_{q(h_{1})}g(\pi^{-}_{[\alpha]}(H))| |\mathcal{C}_{q(h_{2})} g(\pi^{+}_{[\alpha]}(H))|
|f(\pi_{A}(H))||q(h_{1})-q(h_{2})| \]  \end{proof}

\subsection{The constant part of $f$}
Continuing with the same notations from the previous subsection, we show in this subsection, that
$q:  \mathbf{A}/ker(\ell_{[\alpha]}) \rightarrow \mathbf{A}/ ker(\ell_{[\Xi]})$ is affine
and compute its constant term.  To compute the the constant term of this affine map,
we make use of Corollary \ref{cross area preservation}, and the property that the standard map $\hat{\phi}$
roughly preserves the number of root class half planes.\smallskip


\begin{lemma} \label{lemma 4-6}
Let $H$ be a $[\alpha]$ block in $B$.  Let $\mathcal{F}$, $\tilde{\mathcal{F}}$ denote the set of (maximal)
$[\alpha]$, $[\Xi]$ half planes in $Sl_{2}^{1}(H)$ and $\tilde{Sl}_{2}^{1}(H)$ respectively.  Then
\[ \log \left|  \mathcal{F}  \right| = \log \left|  \tilde{\mathcal{F}}  \right| + O(\hat{\eta} diam(B)) \] \end{lemma}

\begin{proof}
The claim is base on the fact that there is a bijection between $Sl_{2}^{1}(H) \cap U_{*}$ and $\tilde{Sl}_{2}^{1}(H) \cap \phi(U_{*})$.

Explicitly, let $\mathcal{F}'$ be the set of $[\alpha]$ half planes in $Sl_{2}^{1}(H)$ that are more than
$O(\hat{\eta} diam(B))$ away from $\partial B$, and spend at least $1-\sqrt{c_{2}}$ fraction of
their measure in $U_{*}$.  Then $\mathcal{F}'$ has a relative large measure in $\mathcal{F}$.

For each $\gamma \in \mathcal{F}'$ there is a $[\Xi]$ half plane $\hat{\gamma} \in \hat{\mathcal{F}}$ such that $\phi(\gamma \cap U_{*})$ is
within $\hat{\eta} diam(B)$ of $\hat{\gamma}$.  We define $\psi(\gamma)=\hat{\gamma}$.
Note $\psi$ is at most $e^{\hat{\eta} diam(B) + \sqrt{c_{2}}diam(B)}$ to one since two $[\alpha]$ planes with the same
$\hat{\phi}$ image must be within $\hat{\eta} diam(B)$ of each other whenever they are in $U_{*}$.  Inverse of $\psi$ is defined
similarly. \end{proof}

\begin{lemma} \label{lemma 4-7}
For all $z_{1}, z_{2} \in [z_{bot}, z_{top}]$, where $z_{top}=\max \alpha_{0}(B)$, and $z_{bot}=\min \alpha_{0}(B)$.
\[ q(z_{1})-q(z_{2}) = \frac{\mathbf{b}_{[\alpha]}}{\mathbf{b}_{[\Xi]}} (z_{1}-z_{2}) + O(\hat{\eta} diam(B)) \] \noindent
where $[\Xi]=f_{*}[\alpha]$  \end{lemma}

\begin{proof}
It is sufficient to check this for a $[\alpha]$ block $H$, and $h_{1},h_{2}$ satisfying the
hypothesis of Lemma \ref{lemma 4.1} and \ref{lemma 4.2}.  Let $\mathcal{F}$,
$\tilde{\mathcal{F}}$ denote the set of $[\alpha]$, $[\Xi]$ half planes in $Sl_{2}^{1}(H)$ and $\tilde{Sl}_{2}^{1}(H)$ respectively.
The number of $[\alpha]$ half planes in $Sl_{2}^{1}(H)$ is
\[ |\pi^{-}_{[\alpha]}(H)| \mbox{   } |\pi^{+}_{[\alpha]}(H)| \mbox{    }  e^{\mathbf{b}_{[\alpha]}(h_{1}-h_{2})} \]

\noindent while the number of $[\Xi]$ half planes is
\[ |\mathcal{C}_{q(h_{1})} g(\pi^{-}_{[\alpha]}(H))| \mbox{    }  |\mathcal{C}_{q(h_{2})} g(\pi^{+}_{[\alpha]}(H) \cap S)|
\mbox{    } e^{\mathbf{b}_{[\Xi]}(q(h_{1})-q(h_{2}))} \]

By Lemma \ref{lemma 4-6},
\[ |\mathcal{C}_{q(h_{1})} g(\pi^{-}_{[\alpha]}(H))| \mbox{    }  |\mathcal{C}_{q(h_{2})}
g(\pi^{+}_{[\alpha]}(H) \cap S)| \mbox{     } e^{\mathbf{b}_{[\Xi]}(q(h_{1})-q(h_{2}))} = |\pi^{-}_{[\alpha]}(H)|
\mbox{    } |\pi^{+}_{[\alpha]}(H)|e^{\mathbf{b}_{[\alpha]}(h_{1}-h_{2})}  \mbox{    } e^{ \hat{\eta} diam(B) } \]

\noindent After simplifying,
\[q(h_{1}) - q(h_{2}) = \frac{\mathbf{b}_{[\alpha]}}{\mathbf{b}_{[\Xi]}}(h_{1}-h_{2}) +
\frac{1}{\mathbf{b}_{[\Xi]}} \log \frac{|\pi^{-}_{[\alpha]}(H)| |\pi^{+}_{[\alpha]}(H)|}{|\mathcal{C}_{q(h_{1})} g(\pi^{-}_{[\alpha]}(H))|
|\mathcal{C}_{q(h_{2})} g(\pi^{+}_{[\alpha]}(H) \cap S)|} + \hat{\eta}diam(B) \]
\noindent The claim now follows because by Corollary \ref{cross area preservation},
\[ \frac{|\pi^{-}_{[\alpha]}(H)|  \mbox{     } |\pi^{+}_{[\alpha]}(H)|}{|\mathcal{C}_{q(h_{1})}
g(\pi^{-}_{[\alpha]}(H))| \mbox{     }  |\mathcal{C}_{q(h_{2})} g(\pi^{+}_{[\alpha]}(H) \cap S)|} =O(1) \]
\end{proof}

\begin{lemma} \label{leaving RG invariant}
The $\mathbf{A}'$ part of a standard map is affine.  Its natural action sends $R_{G}$ to $R_{G'}$,
hence can only take on one of finitely many possibilities. \end{lemma}

\begin{proof}
When $G$ has exact $rank(G)$ many root kernels, every point of $\mathbf{A}$ and $\mathbf{A}'$ is uniquely
determined by the intersection of $rank(G)$ many translates of root kernels, so Lemma \ref{lemma 4-7}
shows that $f$ is affine and takes foliations by root kernels of $G$ to that of $G'$. \smallskip

Let $A$ denote for the linear part of $f$, and $\sigma$ be the permutation that $f$
induces on the root classes. The existence of a standard map $O(\hat{\eta} diam(B))$
away from a quasi-isometry means that for $\vec{u}$ ranges over a large subset of
$\mathbb{S}^{n-1}$, $[\Xi]$ a root class, then $\mathit{l}_{[\Xi]}(\vec{u})>0$ if and only if $\mathit{l}_{\sigma([\Xi])}(A(\vec{u}))>0$, and
$\mathit{l}_{[\Xi]}(\vec{u})=0$ if and only if $\mathit{l}_{\sigma([\Xi])}(A(\vec{u}))=0$\symbolfootnote[2]{Note that in general,
$\mathit{l}_{[\Xi]} \circ A^{-1} = c \mathit{l}_{\sigma([\Xi])}$ for some positive $c$, not necessarily $1$.  All that we know is that $A$
induces a permutation on the root classes, not on the set $R_{G}$.}. So $\mathit{l}_{\sigma([\Xi])} \circ A = c_{[\Xi]}\mathit{l}_{[\Xi]}$ for
some $c_{[\Xi]}>0$. \medskip

Since $f$ is affine, the push-forward of $\ell \in \mathbf{A}^{*}$, $f_{*}(\ell)=\ell \circ A^{-1}$, is
an element of $\mathbf{A}'^{*}$. Take $\ell$ a regular linear functional of unit
norm, and $H$ an $[\alpha]$ block.  Then, inside of $Sl_{2}^{1}(H)$, the number of maximal sets of the form
$p \mbox{   } \ell^{-1}[c,d]$, $p \in \mathbf{H}$, $d-c=L$ is

\[ |\pi^{-}_{[\alpha]}(H)| \mbox{   }  |\pi^{+}_{[\alpha]}(H)|  e^{L \sum_{[\Xi]: \Xi(\vec{v}_{\ell}) >0}
\mathit{l}_{[\Xi]}(\vec{v}_{\ell}) } \]

\noindent while the number of $f_{*} \ell$ half planes in $\tilde{Sl}_{2}^{1}(H)$ is
\[ |\mathcal{C}_{q(h_{1})} g(\pi^{-}_{[\alpha]}(H))| \mbox{    }  |\mathcal{C}_{q(h_{2})}
g(\pi^{+}_{[\alpha]}(H) \cap S)| e^{L  \sum_{[\Xi]: \Xi(\vec{v}_{\ell}) >0} \mathit{l}_{\sigma ([\Xi])} (A\vec{v}_{\ell}) } \]

\noindent Simplifying using Lemma \ref{lemma 4-6} and Corollary \ref{cross area preservation}
yields  \[ \sum_{[\Xi]:\Xi(\vec{v}_{\ell})>0} \mathit{l}_{\sigma ([\Xi])} (A\vec{v}_{\ell})
= \sum_{[\Xi]:\Xi(\vec{v}_{\ell})>0} \mathit{l}_{[\Xi]}(\vec{v}_{\ell}) +  \hat{\eta} diam(B)  \]

\noindent  Since the map from $S^{n-1}$ to disjoint union of root classes defined by
sending $\vec{v}$ to $\{[\Xi]: \Xi(\vec{v}) > 0 \} \sqcup \{[\beta]: \beta(\vec{v})< 0 \}$ is constant on chambers, and each
chamber contains a basis of $\mathbf{A}$, we conclude that up to an error of $\hat{\eta} diam(\mathbf{B}(\omega))$,

\[  \sum_{[\Xi]: \Xi(\vec{v}_{\ell}) >0} \mathit{l}_{[\Xi]} = \sum_{[\Xi]: \Xi(\vec{v}_{\ell}) >0} \mathit{l}_{\sigma([\Xi])} \circ A=
\sum_{[\Xi]: \Xi(\vec{v}_{\ell}) >0} c_{[\Xi]} \mathit{l}_{[\Xi]} \]

\noindent In other words, we have two equations:
\begin{eqnarray*}
\sum_{[\alpha]: \alpha(\vec{v}) > 0 } (1-c_{[\alpha]}) l_{[\alpha]}
&=&0 \\ \sum_{[\beta]: \beta(\vec{v}) < 0 } (1-c_{[\beta]})
l_{[\beta]} &=&0 \end{eqnarray*}

\noindent  therefore $\tilde{R}=\{ (1-c_{[\alpha]}) \mathit{l}_{[\alpha]}, [\alpha]
\mbox{an equivalence class of roots} \}$ is a finite set of linear functionals whose sum
is zero and such that for any codimension 1 hyperplane, the
sum of those elements in $\tilde{R}$ lying entirely on a half plane is zero. Therefore
$\tilde{R}$ consist of zero linear functionals, so $c_{[\alpha]}=1$ for all root
equivalence classes $[\alpha]$.  But this means $\mathit{l}_{\sigma([\alpha])} \circ
A=\mathit{l}_{[\alpha]}$.  So $A$ is a linear map that sends $R_{G}$ to $R_{G'}$.
\end{proof}

\begin{remark}\label{branching constant the same}
Lemma \ref{leaving RG invariant} implies that the linear part of $f$, when viewed just as
a linear map on $\mathbb{R}^{n}$, where $n$ is the rank of $G$ (which is also the rank of
$G'$) is a scalar multiple of an element $A_{f} \in O(n)$, where $A_{f}$ has finite
order.  (This is because $f$ sends foliations by root kernels of $G$ to that of $G'$, and
elements of $R_{G}$ (resp. $R_{G'}$ ) are orthogonal to root kernels.  \end{remark}


The next proposition gives an interpretation to the constant part of $f$.
\begin{proposition} \label{map on heights}
Let $\beta \ll \beta' \ll \beta'' \ll 1$ be as in Lemma \ref{lemma 4.1}.  Let $S(E_{-}, E^{+}, K, h_{bot}, h_{top}) $ be a generalized
$[\alpha]$ slab in $B$.  Suppose $h_{bot} < z_{bot} < z_{top} < h_{top}$ with $4 \beta (h_{top}-h_{bot}) \leq (z_{top}-z_{bot}) \leq
\beta' (h_{top}-h_{bot})$, and $|h_{top}-z_{top}|, |z_{bot}-h_{bot}| > 4 \kappa^{2} \beta''
(h_{top}-h_{bot})$. \smallskip

Then there exists a set $S \subset B$ as in Lemma \ref{lemma 4.2} such that for all $z \in [z_{bot}, z_{top}]$,
\begin{equation}\label{eqn4.5}
q(z)= \frac{\mathbf{b}_{[\alpha]}}{\mathbf{b}_{f_{*}[\alpha]}}z -
\frac{1}{\mathbf{b}_{f_{*}[\alpha]}} \log \frac{|\mathcal{C}_{q(z_{bot})} g(\pi^{+}_{[\alpha]}(H) \cap S)|}
{|\pi^{+}_{[\alpha]}(H)|} + O(\hat{\eta} diam(B)) \end{equation} \end{proposition}

\begin{remark}
In all application of Proposition \ref{map on heights}, we change $q$ by $O(\hat{\eta} diam(B))$ in order to have equation (\ref{eqn4.5}) hold with no error term.
\end{remark}

\begin{proof}
We know from Lemma \ref{lemma 4-7} that
\[ q(z) = \frac{\mathbf{b}_{[\alpha]}}{\mathbf{b}_{f_{*}[\alpha]}}(z) + \left( q(z_{bot}) -
\frac{\mathbf{b}_{[\alpha]}}{\mathbf{b}_{f_{*}[\alpha]}}(z_{bot}) \right) + O(\hat{\eta} diam(B)) \]
\noindent We now find an alternative expression for the term in the bracket. \smallskip

Let $H_{top} \subset B$ be a $[\alpha]$ block such that $\alpha_{0}(H_{top}) = z_{top}$.
Then, according to equation (\ref{measure on boundary})

\[ \log   \frac{|\pi^{-}_{[\alpha]}(H)| e^{-\mathbf{b}_{[\alpha]} z_{top} } }
{ |\mathcal{C}_{q(h_{1})} g(\pi^{-}_{[\alpha]}(H))| e^{-\mathbf{b}_{f_{*}[\alpha]} q(z_{top}) }} = O(\hat{\eta} diam(B)) \]

\noindent Simplifying gives

\[  \log   \frac{|\pi^{-}_{[\alpha]}(H)|}{|\mathcal{C}_{q(h_{1})} g(\pi^{-}_{[\alpha]}(H)) |}
= \mathbf{b}_{[\alpha]} z_{top} -  \mathbf{b}_{f_{*}[\alpha]} q(z_{top}) +  O(\hat{\eta} diam(B)) \]

\noindent Together with Corollary \ref{cross area preservation} and Lemma \ref{lemma 4-7}
produces \begin{eqnarray*}
\frac{1}{\mathbf{b}_{f_{*}[\alpha]}} \log \frac{|\mathcal{C}_{q(h_{2})}g(\pi^{+}_{[\alpha]}(H) \cap S)|}{|\pi^{+}_{[\alpha]}(H)|}
&= & \frac{1}{\mathbf{b}_{f_{*}[\alpha]}} \log \frac{|\pi^{-}_{[\alpha]}(H)|}{|\mathcal{C}_{q(h_{1})} g(\pi^{-}_{[\alpha]}(H)) |} \\
&=& \frac{\mathbf{b}_{[\alpha]}}{\mathbf{b}_{f_{*}[\alpha]}} z_{top} - q(z_{top}) +
O(\hat{\eta} diam(B)) \end{eqnarray*}  \end{proof}

\begin{remark} \label{general section 3}
Now that we know the $\mathbf{A}'$ part of a standard map is affine, we can repeat this
entire section for half planes defined for a regular linear functional $\ell \in
\mathbf{A}^{*}$ using the analogous definitions of upper and lower $\ell$ boundaries
given in Section \ref{shadow+slab}. \end{remark}

\section{Aligning the linear part of standard maps}
Again, we refer to Theorem \ref{exisence of standard maps in small boxes}. By Lemma \ref{leaving RG invariant}, the linear
part of $f_{i}$'s that appeared in the conclusion of Theorem \ref{exisence of standard maps in small boxes}
sends $R_{G}$ of $G$, to $R_{G'}$ of $G'$.  A priori, the linear parts of the $f_{i}$'s do not have to be the same.  For a generic $G$ and $G'$,
the group of permutations between $R_{G}$ and $R_{G'}$ coming from linear maps would be trivial in which case the linear par of $f_{i}$'s
are identity. \smallskip

In this section, we show that when rank of $G$ is 2 or higher, the linear parts of the $f_{i}$'s have to be the same.  When $G$ is rank 1,
there are exactly two root classes and the argument follows exactly the same proof as in \cite{EFW2}, with the modification of
replacing $x$ and $y$ horocycles by left translates of horocycles corresponding to those two root classes. \medskip

We aim to prove the following by the end of this section.
\begin{theorem} \label{aligning horocycle}
Let $G, G'$ be non-degenerate, unimodular, split abelian-by-abelian Lie groups.  Let
$\phi: G \rightarrow G'$ be a $(\kappa, C)$ quasi-isometry.  Given $0< \delta, \eta< \tilde{\eta} < 1$, there exist numbers $L_{0}$ such
that if $\Sigma \subset \mathbf{A}'$ is a product of intervals of equal size at least $L_{0}$, then there is a subset
$U \subset \phi^{-1}(\mathbf{B}(\Sigma)) \subset G$ of relative measure at least $1-\theta$ such that for any root class horocycle
$p \mbox{   } V_{[\alpha]} \subset G'$,
\[ d(\phi^{-1}(p V_{[\alpha]} \cap \phi(U)), p' V_{[\beta]} \cap U)= O(\hat{\eta} diam(\Sigma)), \mbox{for some root horocycle }
 p' \mbox{  } V_{[\beta]} \subset G \] \noindent Here $\theta \rightarrow 0$, and $L_{0} \rightarrow \infty$ as $\delta$, $\eta$, $\tilde{\eta}$
 approach zero.  \end{theorem}


Choose $\delta, \eta \ll 1$.  By Lemma \ref{Forste-boxes are folner} and Remark \ref{Forste-general folner} in \cite{Pg}, $G$ is amenable
and boxes have small boundary compared to its volume, therefore the same is true of its image under $\phi^{-1}$, which means that
we can take a sufficiently large box $\mathbf{B}$ and apply Theorem \ref{exisence of standard maps in small boxes} to
$\phi^{-1}(\mathbf{B})$ to obtain a tiling of $\mathbf{B}$ by images of smaller boxes $\mathbf{B}_{i}$.
\begin{equation} \label{tileB}
\mathbf{B} = \bigsqcup_{i \in \mathbf{I}} \phi(\mathbf{B}_{i}) \sqcup \Upsilon \end{equation}

\noindent where $|\Upsilon| \leq \varrho |\mathbf{B}|$, such that there is a subset $\mathbf{I}_{0} \subset \mathbf{I}$ of relative
measure at least $1-\varkappa$, where $\varkappa \rightarrow 0$ as $\tilde{\eta} \rightarrow 0$, with the following properties.

For each $i \in \mathbf{I}_{0}$, there is a subset $U_{i} \subset \mathbf{B}_{i}$ of relative measure at least $1-\theta$,
such that $\phi|_{U_{i}}$ is $\hat{\eta}diam(\mathbf{B}_{i})$ away from a standard map $g_{i} \times f_{i}$, where $f_{i}$ is a
affine map respecting root kernels, and whose linear part preserves the set $R_{G}$.

Since each $U_{i}$ has a relative large measure, $U_{*} = \bigcup_{i \in \mathbf{I}_{0}} U_{i}$ relative large
measure at least $1-(\varrho + \theta)$ in $\phi^{-1}(\mathbf{B})$.  We write $\mathbf{B}_{i}=\mathbf{B}(\Omega_{i})$ where
$\Omega_{i} \subset \mathbf{A}$ is compact convex, and are all isomorphic to each other.
\medskip

For a subset $V \subset G'$, we write $\mathbf{I}(V)$ for those $i \in \mathbf{I}$ such that $\mathbf{B}_{i} \cap V \not= \emptyset$.
\bigskip

Theorem \ref{aligning horocycle} finishes the alignment step because of the following consequences.
\begin{corollary} \label{linear part the same}
In the conclusion of Theorem \ref{exisence of standard maps in small boxes}, the linear part of $f_{i}$'s,
$i \in \mathbf{I}_{0}$ are all the same.  \end{corollary}

\begin{proof}
Let $U_{*}=\bigcup_{i \in \mathbf{I}_{0}} U_{i} \subset G$.  Lemma \ref{aligning horocycle} implies that for boxes that
are within $O(\hat{\eta} diam(\mathbf{B}))$ of a left translate of $\mathbf{H}$, the $f_{i}$'s are the same.  Suppose for
some $x, y \in \phi^{-1}(\mathbf{B}) \cap U_{*}$ with $d(\pi_{A}(x), \pi_{A}(y)) \geq \hat{\eta} diam(\mathbf{B})$, the linear part of
of standard maps supported in neighborhoods of $x$ and $y$ corresponds to
distinct permutations $\sigma_{A_{x}}$, $\sigma_{A_{y}}$ from $R_{G}$ to $R_{G'}$.  So there must be a root class $[\alpha]$ such that
$[\beta]=\sigma_{A_{y}}([\alpha])$ is different from $[\Xi]= \sigma_{A_{x}}([\alpha])$.  Let $P$ be a codimension 1 hyperplane such that
$\vec{v}_{[\beta]}$ and $\vec{v}_{[\Xi]}$ lie on either side of it. \smallskip

Without loss of generality, we can take $xV_{[\alpha]}$, $yV_{[\alpha]}$ such that their intersections with $\mathbf{B}$ have
equal measure.  Let $x'V_{[\Xi]}$ be a horocycle such that $x'V_{[\Xi]} \cap \mathbf{B}$ is $O(\hat{\eta} diam(\mathbf{B})$ Hausdorff distance away
from $\phi(xV_{[\alpha]} \cap \phi^{-1}(\mathbf{B}))$, and $y'V_{[\beta]}$ be the corresponding horocycle for
$yV_{[\alpha]}$.  Since the distance between $\pi_{A}(x)$ and $\pi_{A}(y))$ is at least $\hat{\eta} diam(\mathbf{B})$, we also have
$d(\pi_{A}(x') , \pi_{A}(y')) > O(\hat{\eta} diam(\mathbf{B}))$. \medskip

Let $\mathcal{F}$ be the set of geodesic segments in direction $P^{\perp}$ (the unique direction perpendicular to $P$) with one
end point in $x'V_{[\Xi]} \cap \mathbf{B}$ and another in $y'V_{[\beta]} \cap \mathbf{B}$.  Then, elements of
$\mathcal{F}$ have length at most $O(diam(\mathbf{B}))$.  An element of $\phi^{-1}(\mathcal{F})$ on the other hand, is a path with one end point in $xV_{[\alpha]} \cap \phi^{-1}(\mathbf{B})$ and
another in $yV_{[\alpha]} \cap \phi^{-1}(\mathbf{B})$.  Since the distance between $\pi_{A}(x)$ and $\pi_{A}(y)$ is bigger than
$\hat{\eta} diam(\mathbf{B})$, we can assume without loss of generality that $| \alpha_{0}(x) - \alpha_{0}(y) |=
\Omega(\tilde{\epsilon}diam(\mathbf{B}))$, for some $\tilde{\epsilon} \leq \hat{\eta}$
such that $diam(\mathbf{B}) \ll e^{\tilde{\epsilon}diam(\mathbf{B})}$.  As the measure of $xV_{[\alpha]} \cap \mathbf{B}$ and
$yV_{[\alpha]} \cap \mathbf{B}$ are equal, this forces lengths of a large proportion of elements in $\phi^{-1}(\mathcal{F})$ to be at least
$e^{\tilde{\epsilon} diam(\mathbf{B})}$, which is a contradiction because the length of an element of
$\phi^{-1}(\mathcal{F})$ is $O(diam(\mathbf{B}))$.  \end{proof}

\begin{corollary} \label{strengthening part 2}
Given $0< \delta, \eta< \tilde{\eta} < 1$, there exist numbers $L_{0}$, $\hat{\eta}$ such that if $\Omega \subset \mathbf{A}$ is a product of
intervals of equal size at least $L_{0}$, then there is a subset $U \subset \mathbf{B}(\Omega)$ of relative measure at least
$1-\tilde{Q}$, and a standard map $\hat{\phi}=g \times f$ where $f$ is affine defined on it such that
\[ d(\phi|_{U}, \hat{\phi}) = O(\hat{\eta} diam(\mathbf{B}(\Omega))) \]

\noindent Here $\tilde{Q}$ and $\hat{\eta}$ approach zero as $\tilde{\eta}, \eta, \delta \rightarrow 0$. \end{corollary}

\begin{proof}
By Theorem \ref{exisence of standard maps in small boxes} and Corollary \ref{linear part the
same}, there is a tiling of $\mathbf{B}(\Omega)$:
\[ \mathbf{B}(\Omega)= \bigsqcup_{i \in \mathbf{I}} \mathbf{B}(\omega_{i}) \sqcup \Upsilon \]

\noindent where each $\omega_{i}$ is isometric to $\rho \Omega$ with the properties that there is a subset
$\mathbf{I}_{0} \subset \mathbf{I}$, with relative measure at least $1-\varkappa$ such that for $i \in \mathbf{I}_{0}$, there is a
subset $U_{i} \subset \mathbf{B}(\omega_{i})$, with relative measure at least $1-\theta$, and a standard map $g_{i} \times f_{i}$ supported on
it such that
\[ d(\phi|_{U_{i}}, g_{i} \times f_{i}) = \hat{\eta} diam(\mathbf{B}(\omega_{i})) \]

\noindent Furthermore, the linear parts of $f_{i}$'s are all the same.  Denote this common linear map by $A_{f}$.

We now basically proceed according to the same steps as in the proof to Theorem \ref{exisence of standard maps in small boxes} in
\cite{Pg}. \smallskip

Let $\mathcal{P}^{0}=\bigcup_{i \in \mathbf{I}_{0}} U_{i}$.  Since each $U_{i}$ has
relative measure at least $1-\theta$ in $\mathcal{P}(\mathbf{B}(\omega_{i}))$, therefore the measure
of $\mathcal{P}^{0}$ is at least $1-\theta'$ times that of $\mathcal{P}(\mathbf{B})$,
where $\theta' \rightarrow 0$ as $\tilde{\eta}$ approaches zero.  Let $\mathcal{L}^{0}=\{ \mathit{l} \in \mathcal{L}(\mathbf{B}): |\mathit{l} \cap \mathcal{P}^{0}| \geq (1-\theta'^{1/2}) |\mathit{l}| \}$.
By Chebyshev inequality, the size of $\mathcal{L}^{0}$ is at least $1-\theta''$ that of $\mathcal{L}(\mathbf{B})$ where
$\theta'' \rightarrow 0$ as $\tilde{\eta}$ approaches zero. \\

We now show that if $\mathit{l} \in \mathcal{L}^{0}$, then $\phi(\mathit{l})$ is close to a geodesic segment. \smallskip

First we claim that $\pi_{A} \circ \phi(\mathit{l})$ is close to a straight line segment in $\mathbf{A}'$.  By construction, if
$\mathit{l}_{i}=\mathit{l} \cap \mathcal{P}_{i}$ for some $i \in \mathbf{I}_{0}$, then $\pi_{A}(\phi(\mathit{l}_{i}))$ is within
$\hat{\eta}$-linear neighborhood of the line segment in direction $A_{f} \circ
\mathit{l}$, and makes an angle at least $\sin^{-1}(\tilde{\eta})$ with root kernels.  Since $\mathit{l}$ spend all but $\theta'^{1/2}$
proportion of its length in $\mathcal{P}^{0}$, we conclude that $\pi_{A}(\phi(\mathit{l}))$ is within
$(\hat{\eta}, \theta'^{1/2}|\mathit{l}|)$-linear neighborhood of a line segment $\hat{\mathit{l}}$ in direction $A_{f} \circ \mathit{l}$,
and makes an angle at most $\sin^{-1}(\tilde{\eta})$ with root kernels. \smallskip

Let $h: \phi(\mathit{l}) \rightarrow \hat{\mathit{l}}$ sends each point on $\phi(\mathit{l})$ to a point on
$\hat{\mathit{l}}$ that is closest to its $\pi_{A}$ image.  We now show that if for some $p,q \in \mathit{l} \cap \mathcal{P}^{0}$,
the line segment connecting $h(\phi(p))$ to $h(\phi(q))$ is orthogonal to $\hat{\mathit{l}}$,
then $d(\phi(p),\phi(q)) \leq \dot{\eta} |\phi(\mathit{l})|$, where $1 \gg \dot{\eta}/2 \gg \theta'^{1/2}$.

Suppose not, then we have two points with above property except that the distance
between them is bigger than $\dot{\eta} |\phi(\mathit{l})|$.  Then either $p,q \in \mathcal{P}_{i}$ for some $i \in \mathbf{I}_{0}$, or
$p \in \mathcal{P}_{\iota_{1}}, q \in \mathcal{P}_{\iota_{2}}$ for distinct $\iota_{1}, \iota_{2} \in \mathbf{I}_{0}$. \smallskip

The first case implies that $d(\phi(p), \phi(q)) \leq \hat{\eta} diam(\mathbf{B}_{i}) <
\dot{\eta} |\mathit{l}|$ so this case cannot happen.

So $p,q$ must belong to distinct $\mathbf{B}(\omega_{i})$'s.  In this case, the lower bound on the distance between $\phi(p)$ and $\phi(q)$ and the knowledge that
$\pi_{A}(\phi(\mathit{l}))$ is within $(\hat{\eta}, \theta'^{1/2}|\mathit{l}|)$-linear neighborhood of $\hat{\mathit{l}}$
means that there must be two points $x, y \in \mathcal{P}_{j} \cap \mathit{l}$ for
some $j \in \mathbf{I}_{0}$ such that the oriented line segment connecting $\pi_{A}(\phi(x))$ and $\pi_{A}(\phi(y))$ makes an angle
of at least $\pi-\sin^{-1}(\hat{\eta})$ with $\hat{\mathit{l}}$, but this is a contradiction to the Lemma \ref{linear part the same} which says
that the linear parts of $f_{i}$'s are the same. \smallskip

Since $\mathit{l}$ spend all but $\theta'^{1/2}$ proportion of its length in $\mathcal{P}^{0}$, by Lemma
\ref{Forste-how to show weakly monotone} in \cite{Pg}, we conclude that
$\phi(\mathit{l})$ is ($\hat{\eta}, \theta'^{1/2} |\phi(\mathit{l})|$) weakly monotone, and by Proposition
\ref{Forste-close to being straight} of \cite{Pg}, $\phi(\mathit{l})$ is within $(\hat{\eta}, \theta'^{1/2} |\phi(\mathit{l})|)$-linear
neighborhood of a geodesic segment in direction $A_{f} \circ \mathit{l}$. \medskip

So now the $\phi$ image of each element of $\mathcal{L}^{0}$ is close to a geodesic segment and we can apply
Lemma \ref{Forste-bare minimum for flats to flats} in \cite{Pg} to obtain a subset $\mathcal{F}^{0} \subset \mathcal{F}(\Omega)$ of relative
large measure such that if $\gamma \in \mathcal{F}^{0}$, $\phi(f)$ is within $\eta diam(\gamma)$ Hausdorff neighborhood of another flat.
We now apply the same argument as in the proof of Theorem \ref{exisence of standard maps in small boxes} using quadrilaterals
to obtain a standard $\hat{f} \times \hat{g}$ on $\mathcal{P}^{0}$, where linear part of $\hat{f}$
is $A_{f}$. \end{proof}

\subsection{$S$ graph, $\hat{S}$ graph and the $H$ graph}
We continue with the setting from equation (\ref{tileB}) and discretize $\mathbf{B}$ in this section so that it reflects the structures of
the standard maps $\hat{\phi}_{i}$, for $i \in \mathbf{I}_{0}$. \smallskip

\bold{The $S$ graph} Take $\mathcal{G}'$, a $\rho_{1}$ net in $G'$ (the range space) and connect $x,y  \in \mathcal{G}'$ by an edge if
their $\pi_{A}$ images are within $\rho_{1}$ of each other and $d(x,y) \leq 10 \rho_{1}$.  We metrize this graph by letting
lengths of edges be the distance between the corresponding points in $G'$, so all edges have length $O(\rho_{1})$.  We refer to the
discretization restricted to $\mathbf{B}$ as the $S$ graph.  \medskip

Recall that that each box $\mathbf{B}_{i}$ tiling the set
$\phi^{-1}(\mathbf{B})$ is isometric to $\mathbf{B}(\omega)$, the box associated to some convex compact set $\omega \subset \mathbf{A}$.
Let $\rho_{i}$,$i=2,3,4,5$, be numbers such that $\rho_{i} \ll \rho_{i+1} \ll  diam(\omega)$.
Fix $0 < \beta \ll  \beta' \ll \beta'' \ll 1$ such that $\hat{\eta} \ll \beta $.


\bold{The set $E_{\ell}(p)$ } Let $\ell$ be a regular linear functional of
$\mathbf{A}$ with unit norm.  For $p \in G$, we write $E_{\ell}(p)$ for the left coset $p
\mbox{   } W^{+}_{\ell} \rtimes ker(\ell)$.  A left coset of $W^{+}_{\ell} \rtimes ker(\ell)$ is also
the pre-image of a point under the projection map $G \rightarrow W^{-}_{\ell} \rtimes
\mathbb{R}\vec{v}_{\ell}$. \medskip

We say $E_{\ell}(p)$ is \emph{favourable} if
\begin{enumerate}
\item \[ \left|   E_{\ell}(p) \cap U_{*}  \right|  \geq \left(   1-\vartheta^{1/2} \right) \left|   E_{\ell}(p) \cap
    \phi^{-1}(\mathbf{B}) \right|    \]
    \noindent where $\vartheta$ is the relative proportion of $U_{*}^{c}$ in $\phi^{-1}(\mathbf{B})$.

\item For any point $q$ that is such that $\ell(\pi_{A}(p))-\ell(\pi_{A}(q)) =
    \rho_{5}$ \[ \left|   E_{\ell}(q) \cap U_{*}  \right| \geq \left(   1-\vartheta^{1/2}
     \right)  \left|  E_{\ell}(q)  \cap \phi^{-1}(\mathbf{B})  \right| \] \end{enumerate}

We say $E_{\ell}(p)$ is \emph{very favourable} if in addition,
for every $i \in \mathbf{I}_{0}$ such that $|E_{\ell}(q) \cap U_{i}| \geq
\left( 1-\theta^{1/3} \right) |E_{\ell}(q) \cap \mathbf{B}_{i}|$, the
intersection between $\pi^{-}_{\ell}(E_{\ell}(q) \cap \mathbf{B}_{i})$ and $\tilde{E}_{**}$ is
non-empty, where $\tilde{E}_{**}$ is the union of $E_{**}$ in those $\mathbf{B}_{i}$'s, $i \in \mathbf{I}_{0} \cap \mathbf{I}(E_{\ell}(p))$
from Lemma \ref{lemma 4.1} and Remark \ref{general section 3}. Here $\theta$ is an upper bound for $U_{i}^{c}$ in $\mathbf{B}_{i}$. \medskip

Since $U_{*}$ has relative large measure in $\phi^{-1}(\mathbf{B})$, it follows that if
we fix a chamber $\mathfrak{b}$, there is a left translate of $W^{+}_{\mathfrak{b}}$,
$x_{0}W^{+}_{\mathfrak{b}}$, and a function $\varkappa$ with the following properties:

\begin{enumerate} \renewcommand{\labelenumi}{\Alph{enumi}.}

\item The measure of $x_{0}W^{+}_{\mathfrak{b}} \cap U_{*}$ is at least $1-\varkappa$
times that of $x_{0}W^{+}_{\mathfrak{b}}$.

\item There is a subset $\mathcal{S}_{0} \subset \mathfrak{b}$ of relative measure at least
$1-\varkappa$ such that for every $\ell \in \mathcal{S}_{0}$, and $p \in U_{\mathfrak{b}}$, $E_{\ell}(p)$ is very favourable.

\end{enumerate}
\noindent where $\varkappa$ approaches zero as our initial data of $\delta$, $\eta$ and $\tilde{\eta}$ approach zero. \\

Fix such a left coset of $W^{+}_{\mathfrak{b}}$ and for every $\ell \in \mathcal{S}_{0}$, write $\tilde{H}_{\ell}$ for
$E_{\ell}(x_{0}) \cap \phi^{-1}(\mathbf{B})$ and $H_{\ell}$ for the subset of
$Sh(E_{\ell}(x_{0}), \rho_{5}) \cap \phi^{-1}(\mathbf{B})$ whose $\ell$ value is $\ell(x_{0})-\rho_{5}$.

\bold{The sets $\mathbf{I}_{g}(\ell)$, good and bad $\ell$ boxes}  For each $\ell \in
\mathcal{S}_{0}$, let $\mathbf{I}_{g}(\ell) \subset \mathbf{I}_{0}$ consisting of indices such that
\[ \left|  H_{\ell} \cap U_{i} \right| \cap \left(   1-\vartheta^{1/3}  \right) \left|   H_{\ell}  \right|  \]

A box $\mathbf{B}_{i}$ in the domain is called a \emph{good $\ell$ box} if $i \in \mathbf{I}_{g}(\ell)$, and bad otherwise.

\bold{Shadows of $H_{\ell}$ and $\phi(H_{\ell})$ } Fix a $\ell \in \mathcal{S}_{0}$.  Let
$h_{1}=\ell(H_{\ell})-\beta diam(\omega)$ and $h_{2}=\ell(H_{\ell}) - (\beta + \frac{\beta'}{2}) diam(\omega)$.  Since $E_{\ell}(x_{0})$ is
very favourable, Lemma \ref{lemma 4.1} says that with those chosen $h_{2}$ and $h_{1}$, for each $i \in \mathbf{I}_{g}(\ell)$,
\[ | Sl_{2}^{1}(H_{\ell}) \cap U_{i} | \geq (1-c_{2}) |Sl_{2}^{1}(H_{\ell})| \] \noindent
For each $j$ between $h_{2}$ and $h_{1}$, if we write $\rho(j)$ for the relative
proportion of $Sl_{2}^{1}(H_{\ell}) \cap \ell^{-1}(j) \cap U_{i}^{c}$ in
$Sl_{2}^{1}(H_{\ell}) \cap \ell^{-1}(j)$, the above condition means that

\[ \sum_{j=h_{2}}^{h_{1}} \rho(j) \leq 2c_{2} \] \noindent which means that $\rho(h_{0}^{i}) \leq 2\sqrt{c_{2}}$ for some
$h_{0}^{i} \in [h_{2},h_{1}]$.

The shadow of $H_{\ell}$, denoted by $W(H_{\ell})$, is now defined as the union of
$\cup_{i \in \mathbf{I}_{g}(\ell)} Sl_{2}^{1}(H_{\ell})  \cap \ell^{-1}(h_{0}^{i})$ and
$\cup_{i \in \mathbf{I}(H_{\ell}) \backslash \mathbf{I}_{g}(\ell)} Sl_{2}^{1} \cap
\ell^{-1}(h_{1})$.

Let $S=U_{*} \cap S\left( \pi^{-}_{\ell}(\tilde{H}_{\ell}), \pi^{+}_{\ell}(\tilde{H}_{\ell}), \pi_{A}(\tilde{H}_{\ell}),
 \ell(\tilde{H}_{\ell}),  \ell(\tilde{H}_{\ell}) + \beta'' diam(\omega) \right)$.  For each $i \in \mathbf{I}_{g}(\ell)$, we define $\hat{W}_{i}(H_{\ell})$
as $\tilde{Sl}_{2}^{1}(H_{\ell}, S) \cap (f_{*}\ell)^{-1}(q(h_{0}^{i}))$. The shadow of
$\phi(H_{\ell})$, denoted by $\hat{W}(H_{\ell})$ is now defined as the union of
$\hat{W}_{i}(H_{\ell})$'s, where $i$ ranges over $\mathbf{I}_{g}(\ell)$. \smallskip

\bold{Shadow vertices }  We now define a set of \textit{$\ell$ shadow vertices} in the
discretization of $\mathbf{B}$.  By shifting the discretization, we can assume that
$\hat{W}(H_{\ell})$ contains a $\rho_{1}$ net of $S$-vertices.  Every $S$-vertex in
$\hat{W}(H)$ is a shadow vertex.

Furthermore, a $\ell$ shadow vertex is \emph{good} if it lies in $\phi(U_{*})$, belonging to $\ell$ plane containing a point of
$\phi(U_{*})$, and is at least $10\kappa \beta'' diam(\omega)$ from $\partial \mathbf{B}$.

If a shadow vertex is not good, then it is bad.  We enlarge the set of bad shadow vertices by declaring any $S$-vertex in the
$\rho_{1}$ neighborhood of $\phi(U_{*}^{c} \cap W(H))$ a bad shadow vertex, even if it was a good shadow vertex by our previous
definition.

The bad shadow vertices in $N_{\rho_{1}} \phi(U_{*}^{c} \cap W(H))$ are not necessarily
close to $\hat{W}(H)$, even if they come from good boxes\symbolfootnote[2]{Because it's the bad set in the good box, which we have
no control over.  Anything termed 'bad' basically comes from sets that we have no control over: everything in bad boxes, and bad
sets $U_{i}^{c}$'s in good boxes.}.  While these bad shadow vertices are not well controlled, they make up a small proportion of all
shadow vertices and so do not interfere with the geometric argument in the next section.  See Lemma \ref{bad shadow are
small} below. \smallskip

For either good or bad boxes, the number of $\ell$ shadow vertices coming from $\mathbf{B}_{i}$ is proportional to the size of the
set $H_{\ell} \cap \mathbf{B}_{i}$.  The proportionality constant depends only on $\kappa$, $C$,
$G$ and $G'$.

\begin{lemma} \label{bad shadow are small}
For each $\ell \in \mathcal{S}_{0}$, there is a constant $c_{5}$ depending on our initial data
$\delta, \eta, \tilde{\eta}$ such that the proportion of bad shadow vertices is at most
$c_{5}$ which approach zero as our initial data go to zero's.  \end{lemma}

\begin{proof}
Bad shadow vertices are defined in two stages.  First we have the set $S_{1}$ of vertices
in $\hat{W}(H_{\ell})$ that are either within $10 \kappa \beta'' diam(\omega)$ of
$\partial \mathbf{B}$ or outside of $\beta' diam(\omega)$ neighborhood of a point in $\phi(U_{*})$
whose $\ell$ value is $h_{0}^{i}$ smaller than $\ell(H_{\ell})$.  That this set has small
measure in $\hat{W}(H_{\ell})$ follows from two facts.  First, the subset that are close
to $\partial \mathbf{B}$ has relative small measure by Lemma \ref{boxes are folner}.
\smallskip

Second, if the proportion of $S_{1}$ in $\hat{W}(H_{\ell})$ is $\theta$, then the set of
points in $\tilde{Sl}_{2}^{1}(H_{\ell}) \cap \phi(U_{*}^{c})$ contained in a $\ell$ half
plane through a point of $S_{1}$ has measure at most $\theta$ relative to
$\tilde{Sl}_{2}^{1}(H_{\ell})$.  However by Lemma \ref{lemma 4.2}, the proportion of
$\tilde{Sl}_{2}^{1}(H_{\ell}) \cap \phi(U_{*}^{c})$ in $\tilde{Sl}_{2}^{1}(H_{\ell})$ is
at most $c_{4}$.  Therefore $\theta \leq c_{4}$.

In the second stage, we enlarge the set of bad vertices in $\hat{W}(H)$ by adding the set
$N_{\rho_{1}} \phi(U_{*}^{c} \cap \hat{W}(H_{\ell}))$.  That this set has small measure
follows from our choice of $h_{0}^{i}$.  \end{proof}

\bold{The $\hat{S}$-graph } We now modify the $S$-graph near $\phi(\mathbf{H_{\ell}})$ so
that it reflects divergence property dictated by standard maps.

For $x \in W^{+}_{\ell}$, $y \in W^{-}_{\ell}$ and $t \in \mathbb{R}$, we write
$\gamma_{x,y}(t)$ for the set $(x,y)\mbox{  }\ell^{-1}(t)$, and $\gamma_{x,y}([c,d])$ for
the $\cup_{t \in [c,d]} \gamma_{x,y}(t)$, which is a $\ell$ half plane of length $|c-d|$. \smallskip

Let $K_{i}$ be the union of $\gamma_{x,y}([q(h_{0}^{i}),q(\ell(\tilde{H}_{\ell}))-
\rho_{5}/4 ])$'s that have non-empty intersection with $\hat{W}_{i}(H_{\ell})$\symbolfootnote[3]{Note that $\hat{W}_{i}(H_{\ell})$ is only defined for $i \in
\mathbf{I}_{g}(\ell)$.}.  We begin by replacing $K_{i}$ as a subset of the $S$ graph by
disjoint union of $\gamma_{x,y}$'s, then define the $\hat{S}$ graph by declaring a new set of vertices and a new incidence relation
on $K_{i}$. \smallskip

For each $t_{j} \in \frac{1}{\rho_{1}}(q(\ell(\tilde{H}_{\ell}))- \rho_{5}/4 -
q(h_{0}^{i}))$, call the $S$ vertices of $\gamma_{x,y}(t_{j}) \in K_{i}$
\emph{pre-vertices}.  Recall that for each linear functional $\Xi$, we can evaluate
$\Xi(p)$ where $p \in G$ to be the $\Xi$ value of $\pi_{A}(p)$.  In the range, we tile
left cosets of $W^{-}_{(f_{i})_{*}\ell}$ in $((f_{i})_{*}\ell)^{-1}(t_{j})$ by rectangles
$T_{-}$'s of diameter $10 \rho_{1}$; in the domain, we tile left cosets of $W^{+}_{\ell}$
in $\ell^{-1}(q_{i}^{-1}(t_{j}))$ by rectangles $T_{+}$'s of diameter $10 \kappa^{2}
\rho_{1}$.  We identify two vertices $p,q$ if
\begin{enumerate}
\item $p,q$ are in the same $T_{-}$.

\item $\phi_{i}^{-1}(p)$ and $\phi_{i}^{-1}(q)$ are in the same $T_{+}$ which has the
    property that
    \[ | \partial^{-}_{\ell}(T_{+}) \cap \tilde{E}_{**} | \geq 1/2
    |\partial^{-}_{\ell}(T_{+})| \]

    \noindent where $\tilde{E}_{**}$ is the union of $E_{**}$ coming from each $\mathbf{B}_{i}$
    , $i \in \mathbf{I}_{0}$ as in Lemma \ref{lemma 4.1}. \end{enumerate}

\noindent We also remove any edges in $K_{i}$ that ends at a bad shadow vertex.  A
$\hat{S}$ vertex is called \emph{regular} unless it arise from the procedure above, in
which case it is called \emph{irregular}. \\


In our original $S$ graph, every point has the same valence provided the vertex is not close to the boundary.  However, upon the
changes made for the $\hat{S}$ graph, the 'homogeneous'-ness of valence is not so clear. This next lemma says that we only change
the valence by bounded amount, so that $\hat{S}$ graph is essentially homogeneous away from boundary.

\begin{lemma} \label{uniform valence}
There exist constants $M_{l}$, $M_{u}$ depending only on $\kappa$, $C$ such that for any
two $\hat{S}$ vertices $v_{1}, v_{2}$, the ratio of numbers of $f_{*} \ell$ half planes
through them is bounded between $M_{l}$ and $M_{u}$.  \end{lemma}

\begin{proof}
Let $v$ be an irregular vertex, and let $z_{top}=\max (f_{*}\ell)(\mathbf{B})$.  In the
following calculations, we use $\approx$ to mean that up to a multiplicative error of
$e^{O(\hat{\eta} diam(\omega))}$.  First, the number of $f_{*} \ell$ half planes of length
$z_{top}-(f_{*}\ell)(v)$ containing $v$ and $(f_{*}\ell)^{-1}(z_{top})$ is $\approx
e^{\mathbf{b}_{f_{*}\ell}(z_{top}-f_{*}\ell(v))}$. \smallskip

As $v$ is an irregular vertex, there exists a $\ell$ block $H'$ in $\mathbf{B}_{i}$ such
that $v \in \phi_{i}(H')$ and $\partial^{-}_{\ell}(H')$ contains a point of $E_{**}$. The
number of $f_{*}\ell$ half planes containing $v$ and $\hat{W}_{i}(H_{\ell})$ (whose
$f_{*}\ell$ and $q_{i}$ value is $h_{0}^{i}$) is
\begin{align*}
\approx &|\mathcal{C}_{q_{i}(h_{0}^{i})} g_{i}(\pi^{-}_{\ell}(H'))| e^{-\mathbf{b}_{f_{*}\ell} h_{0}^{i}} &  \mbox{by (\ref{measure on boundary})} \\
\approx &|\mathcal{C}_{q_{i}(h_{0}^{i})} g_{i}(\pi^{-}_{\ell}(H'))|
\frac{|\mathcal{C}_{q_{i}(h_{2})} g(\pi^{+}_{\ell}(H') \cap S) |}{|\pi^{+}_{\ell}(H')|}
e^{-\mathbf{b}_{\ell}q_{i}^{-1}(h_{0}^{i})} & \mbox{ by Proposition \ref{map on heights}} \\
\approx &|\pi^{-}_{\ell}(H')| e^{-\mathbf{b}_{\ell} q_{i}^{-1}(h_{0}^{i})} & \mbox{ by \ref{cross area preservation}} \\
\approx &e^{\mathbf{b}_{\ell} \left(  \ell(H') - q_{i}^{-1}(h_{0}^{i})  \right) }   & \mbox{ by (\ref{measure on boundary})} \\
=&e^{\mathbf{b}_{f_{*}\ell}  \left(  q_{i}(\ell(H')) - h_{0}^{i}  \right) } & \mbox{ by
Proposition \ref{map on heights}} \end{align*}  \end{proof}

\bold{The $H$ graph }  We now define the $H$ graph as a subgraph of the $\hat{S}$ graph.
A good $\ell$ vertex is an irregular $\hat{S}$ vertex $v \in K_{i}$ such that
$f_{*}\ell(v)$ is within $O(\rho_{1})$ of $q_{i}(\ell(H_{\ell}))$ and the $W^{-}_{\ell}$
coordinate of $\phi_{i}^{-1}(v)$ is within $O(\rho_{1})$ of the $W^{-}_{\ell}$ coordinate
of $H_{\ell}$.  A bad $\ell$ vertex is a bad $\ell$ shadow vertices, which is always a
regular $\hat{S}$ vertices.

A vertex of the $H$ graph is either good or bad.  Good (resp. bad) $H$ vertices is the
union over $\ell \in \mathcal{S}_{0}$, all the good (resp. bad) $\ell$ vertices.  If a vertex comes
from the $\phi$ image of $x_{0}W^{+}_{\mathfrak{b}} \cap U_{*}$, then it is called a
$\mathfrak{b}$ vertex.

An edge of the $H$ graph is concatenation of edges in the $\hat{S}$ graph that connects
two $H$ vertices, or connect a good $H$ vertex with $\partial \mathbf{B}$.  An $\mathcal{E}_{0}$ edge is one that connects two
good $H$ vertices or one good $H$ vertex with $\partial \mathbf{B}$.


%

\subsection{Some geometric lemmas}
The idea behind the proof of Theorem \ref{aligning horocycle} is the following observation.
In $\mathit{H}_{n+1}$ (as in Lemma \ref{QI embedding}) suppose two travelers leave a common starting point via segments of
diverging vertical geodesics.  If they are to meet up again without having to travel for long, then they
can only do so in some neighborhood of the starting point. \smallskip

\bold{The projection $\Pi_{\ell}$ and the function $\rho_{\ell}(\centerdot, \centerdot)$}
Let $\Pi_{\ell}: G \rightarrow W^{-}_{\ell} \rtimes \mathbb{R} \vec{v}_{\ell}$ be the projection defined by $(\mathbf{x}, \mathbf{t}) \mapsto
([\mathbf{x}]_{W^{-}_{\ell}}, \ell(\mathbf{t}))$\symbolfootnote[3]{$[\mathbf{x}]_{W^{-}_{\ell}}$
denotes for the $W^{-}_{\ell}$ coordinate of $\mathbf{x}$}.  Let $H$ be a $\ell$ block, and write
$\rho_{\ell}(p,q)=(\Pi_{\ell}(p),\Pi_{\ell}(q))_{\Pi_{\ell}(H)}$ for the Gromov product of $\Pi_{\ell}(p)$ and $\Pi_{\ell}(q)$ with respect to
$\Pi_{\ell}(H)$ in the negatively curved space $W^{-}_{\ell} \rtimes \mathbb{R} \vec{v}_{\ell}$.

Recall that for three points $x,y,z$ in a metric space, the Gromov product is defined as
\[ (y|z)_{z} = \frac{1}{2} \left(  d_{X}(x,y) + d_{X}(x,z) - d_{X}(z,y) \right) \]

\noindent In a $\delta$ hyperbolic space $X$, the geodesic joining $y$ to $z$,
$\gamma_{yz}$, satisfies
\[ d_{X}(\gamma_{yz}, x) - \delta \leq (y|z)_{x} \leq d_{X}(\gamma_{yz},x) \]

For the remaining of this subsection, $H$ denotes for a $\ell$ block. Also, until the end
of Section 4, $\epsilon$ is a positive number less than 1 such that $O(diam(\mathbf{B})) \ll e^{\epsilon diam(\mathbf{B})}$,
and $\hat{\eta} diam(\omega) \ll \epsilon diam(\mathbf{B})$. The length of any path
considered in this section is less than $e^{\epsilon diam(\mathbf{B})}$.

In the next lemma, we list some properties of $\rho_{\ell}$, where $\approx$ is used to
denote two quantities whose ratio depends only on $\kappa, C$, $\ell$ and the space $G$.

\begin{lemma} \label{lemma 5.15} \noindent
\begin{enumerate}
\item Suppose $d(p',p) \ll d(p,H)$, $d(q',q) \ll d(q,H)$, and $\rho_{\ell}(p,q) \ll \min \{ d(p,H), d(q,H) \}$.
Then  \[ \rho_{\ell}(p,q) \approx \rho_{\ell}(p',q') \]

\item Suppose $\ell(p') < \ell(p)$, $\ell(q') < \ell(q)$, and each of the pairs
    $(p,p')$, $(q,q')$ can be connected by a geodesics having at least $\eta$
    proportion of the component in direction $\vec{v}_{\ell}$.

    If $d(p,H) , d(q,H) \gg \rho_{\ell}(p,q)$, then
    \[ \rho_{\ell}(p,q) \approx \rho_{\ell}(p',q') \]

\item If $\rho_{\ell}(p,q), \rho_{\ell}(q,q') \gg s$, then $\rho_{\ell}(p,q') \gg s$.
    \end{enumerate} \end{lemma}

\begin{lemma} \label{lemma 5.16}
Suppose $p,q \in G$ can be connected by a path $\hat{\gamma}$ of length less than
$e^{\epsilon R}$ such that
\begin{enumerate}
\item $\ell(\pi_{A}(p)), \ell(\pi_{A}(q)) \leq \ell(H) - \rho_{4} $

\item The $\ell$ values of $\hat{\gamma}$ decreases for at least $\epsilon R$ units at both
ends, and the $\ell$ values of points on the the remaining subsegment are no more than $\ell(H) - \epsilon R$. \end{enumerate}

Then, $\rho_{\ell}(p,q) \geq \Omega(\rho_{4})$. \end{lemma}

\begin{proof}
Let $p'$ and $q'$ be the closest points to $p$ and $q$ whose $\ell$ value first dips
below $\ell(H) -\epsilon R$.  Since the length of $\hat{\gamma}$ is less than $e^{\epsilon R}$, so is the length of
$\Pi_{\ell}(\hat{\gamma})$, which connects $\Pi_{\ell}(p)$ and $\Pi_{\ell}(q)$, as well
as the subsegment of $\Pi_{\ell}(\hat{\gamma})$ between $\Pi_{\ell}(p')$ and
$\Pi_{\ell}(q')$.  This means that if $\Pi_{\ell}(p) \not= \Pi_{\ell}(q)$, then any path
connecting $\Pi_{\ell}(p')$ to $\Pi_{\ell}(q')$ whose $\ell$ value stays $\epsilon R$
units below $\ell(H)$ would have length at least $e^{\epsilon R}$, contradicting the
assumption about the length of $\hat{\gamma}$.  \end{proof}

\begin{lemma}\label{lemma 5.17}
Let $p_{0}, q_{0}$ be good $\ell$ vertices, and $\gamma$ be an $\mathcal{E}_{0}$ edge
connecting them.  Let $p,q \in \gamma$ be points that are within the same good box as
$p_{0}$ and $q_{0}$ respectively.  Then the following holds:
\begin{enumerate}
\item Except near the end points, $\gamma$ never pass through any irregular $\ell$
    vertices.

\item we have $\rho_{\ell}(\hat{\phi}^{-1}(p), \hat{\phi}^{-1}(q)) \geq  \Omega(\rho_{4})$

\end{enumerate} \end{lemma}

\begin{proof}
Starting from $p_{0}$, let $p_{1}$ be the first place where $\gamma$ hits
$\hat{W}(H_{\ell})$.  As $\gamma$ is a $\mathcal{E}_{0}$ edge, it doesn't hit a bad
$\ell$ shadow vertex, then there exists a good $\ell$ shadow vertex $p'_{1} \in U_{*}
\cap W(H_{\ell})$ such that $\hat{\phi}(p'_{1})=p_{1}$ and $d(\phi^{-1}(p_{1}), p'_{1}) =
O(\hat{\eta} diam(\omega))$.  Note that $p'_{1}$ and $\phi^{-1}(p_{1})$ are both
$\Omega(\epsilon diam(\mathbf{B}))$ away from $\partial (\phi^{-1}(\mathbf{B}))$.

Let $p'_{2}=\phi^{-1}(p_{2})$ be the next point after $p'_{1}$ when $\phi^{-1}(\gamma)$
intersects $U_{*} \cap W(H_{\ell})$.  We know there must be such a point because the
length of $\gamma$ is less than $e^{\epsilon diam(\mathbf{B})}$, so whenever it moves
transverse to $H_{\ell}$, it must do so in $\ell^{-1}[h_{0}^{i}, \infty]$.

Since $\gamma$ does not hit a bad shadow vertex, $p'_{2} \in \ell^{-1}(h_{0}^{i})$ at $i
\in \mathbf{I}_{g}(\ell)$, so $p_{2}$ is a good shadow vertex.  Therefore continuing
$\gamma$ after $p_{2}$ hits a vertex in $\phi(U_{*}) \cap H_{\ell}$, which is a good
$\ell$ vertex, and must be $q_{0}$.

By Lemma \ref{lemma 5.15}
\[ \rho_{\ell}(\hat{\phi}^{-1}(p), \hat{\phi}^{-1}(q)) = \rho_{\ell}(\hat{\phi}^{-1}(p_{1}), \hat{\phi}^{-1}(p_{2}))
\approx \rho_{\ell}(\phi^{-1}(p_{1}), \phi^{-1}(p_{2})) \geq \Omega(\rho_{4}) \] \end{proof}


\begin{lemma}\label{lemma 5.18}
Suppose $p_{0}, q_{0}$ are good $\ell$ vertices and $\overline{p_{0}q_{0}}$ is a
$\mathcal{E}_{0}$ edge connecting them.  If for some $\rho_{1} \ll s \ll \rho_{3} \ll
\rho_{4}$, we have $p, q \in \overline{p_{0}q_{0}}$ in the same good box as $p_{0}$ and
$q_{0}$ satisfying
\[ (f_{i})_{*}\ell (p)-(f_{i})_{*}\ell(p_{0})=(f_{j})_{*}\ell(q)-(f_{j})_{*}\ell(q_{0}) =
s, \mbox{where } i,j \in \mathbf{I}_{g}(\ell) \]

\noindent Then, there is a a $\ell$ block $H'_{\ell}$ such that $p,q$ are within
$O(\rho_{1})$ of $\hat{\phi}(H'_{\ell})$. \end{lemma}

\begin{proof}
Let $p',q'$ be points on $\overline{p_{0}q_{0}}$ close to where it enters respective good
boxes.  By Lemma \ref{lemma 5.17}, $\rho_{\ell}(\hat{\phi}^{-1}(p'), \hat{\phi}^{-1}(q'))
\geq \Omega(\rho_{4})$.  Since $s \ll \rho_{3} < \rho_{4}$, by Remark \ref{branching constant the same}, we conclude that the
$\ell$ values of $\hat{\phi}_{i}^{-1}(p)$ and $\hat{\phi}_{j}^{-1}(q)$ are the same, both are $s$ lower than that of $H_{\ell}$. Since
$s < \rho_{3} \ll \rho_{4}$, $\Pi_{\ell}(\hat{\phi}_{i}^{-1}(p))$ and
$\Pi_{\ell}(\hat{\phi}_{j}^{-1}(q))$ are the same point, which is the same as saying that
$\hat{\phi}_{i}^{-1}(p)$ and $\hat{\phi}_{j}^{-1}(q)$ are within $O(\rho_{1})$ of a
$\ell$ block $H'_{\ell}$ as claimed.  \end{proof}

\begin{lemma} \label{illegal circuit}
Let $n$ be $4$ or $6$.  Suppose for $0 \leq i \leq n-1$, $p_{i}$ are $\hat{S}$ vertices
whose $\phi$ pre-images support standard maps.  Let $\overline{p_{i-1}p_{i}}$ be
subsegments of $\mathcal{E}_{0}$ edges in $H$ graph, where the indices are counted mod
$n$.

Let $r(p_{i})=\min\{ |(f_{i})_{*}\ell(v) -(f_{i})_{*}\ell(p_{i})|, v \mbox{ is a good
$\ell$ vertex.} \}$, where $\phi_{i}=g_{i} \times f_{i}$ is the standard map supported in
a neighborhood of $\phi^{-1}(p_{i})$.

Suppose there is an index $k$ such that $r(p_{k}) \ll \rho_{4}$, and for all $i \not= k$,
$r(p_{i}) > r(p_{k}) + 2 \rho_{1}$.  Then $\overline{p_{k}p_{k+1}}$ and
$\overline{p_{k}p_{k-1}}$ cannot have only $p_{k}$ in common.  \end{lemma}

\begin{proof}
we can assume $k=0$.  Let $H'_{\ell}$ be the $\ell$ block passing through
$\hat{\phi}^{-1}(p_{0})$.  By Lemma \ref{lemma 5.18}, we can consider $H'_{\ell}$ in
place of $H_{\ell}$.  Namely, we can replace the appearance of $H_{\ell}$ in the
definition of $\ell$ vertices by $H'_{\ell}$.  Let $p_{i}^{+}$ and $p_{i}^{-}$ be the
first and last time that $\overline{p_{i-1}p_{i}}$ leaves $\hat{W}(H_{\ell})$. \smallskip

\noindent By Lemma \ref{lemma 5.17}, for all $i \not=0$,
$\rho_{\ell}(\hat{\phi}^{-1}(p_{i-1}^{+}), \hat{\phi}^{-1}(p_{i}^{-})) \geq
\Omega(\rho_{4})$, and $\rho_{\ell}(\hat{\phi}^{-1}(p_{0}^{-}),
\hat{\phi}^{-1}(p_{0}^{+})) \leq \rho_{1}$.  This is a contradiction to Lemma \ref{lemma
5.15} \end{proof}


\subsection{Proof of Theorem \ref{aligning horocycle}}
Let $\mathbb{R}^{k} \rtimes_{\tau} \mathbb{R}$ be a rank 1 space with roots $\alpha_{i}$,
$-\beta_{j}$'s, where $\alpha_{i}, \beta_{j} >0$, and $\tau(t)$ is matrix consisting of blocks of the form
$e^{\alpha_{i}t}N_{i}(\alpha_{i}t)$, $e^{-\beta_{j}t}N_{j}(\beta_{j}t)$ where $N_{i}$, $N_{j}$ are
unipotent matrices with polynomial entries.  For the following two lemmas, let $B[T]$ denote for the subset
$\left( \prod_{i} \Omega_{i} \times \prod_{j} U_{j} \right) \rtimes [-T, T]$, where $\Omega_{i} \subset V^{\alpha_{i}}$, and
$U_{j} \subset V^{-\beta_{j}}$, and $\alpha_{i}, \beta_{j} > 0$, $|\Omega_{i}| e^{-\alpha_{i}T}, |U_{j}| e^{-\beta_{j}T} \geq 1$.
We call the set
\[ \left( \prod_{i} \Omega_{i} \times \prod_{j} U_{j} \right) \rtimes -T \bigcup
\left( \prod_{i} \Omega_{i} \times \prod_{j} U_{j} \right) \rtimes T \]

\noindent the \emph{top and bottom of B[T]}, denoted by $\bar{\partial}B[T]$.

\begin{lemma} \label{rank1generalBOX}
The total number of geodesics in $B[T]$ is $e^{T(\sum_{i}\alpha_{i} + \sum_{j}
\beta_{j})}\prod_{i,j} |\Omega_{i}| |U_{j}|$, and the number of geodesics in $B[T]$
through each vertex is $e^{T (\sum_{i}\alpha_{i} + \sum_{j} \beta_{j})}$.
\end{lemma}

\begin{proof}
To specify a geodesic in $B[T]$, we need to specify its coordinates in $\alpha_{i}$,
$\beta_{j}$ root spaces, and for every choice of $\alpha_{i}$ and $\beta_{j}$ coordinate,
there is a unique geodesic segment in $B[T]$ going from the top to the bottom.  The
number of different coordinates in $\alpha_{i}$ root spaces is $\prod_{i} |\Omega_{i}|
e^{\alpha_{i}T}$, and those in $\beta_{i}$ root spaces is $\prod_{j} |U_{j}|
e^{\beta_{j}T}$, so the number of geodesics is
\[ \left( \prod_{i} |\Omega_{i}| e^{\alpha_{i}T} \right) \left(\prod_{j} |U_{j}| e^{\beta_{j}T} \right) \]

We know that in $B[T]$
\[ \mbox{no. of geodesics } \times \mbox{no. of vertices on a geodesic } = \mbox{no.
vertices} \times \mbox{no. of geodesics through each vertex } \]

The number of vertices on each geodesic is $2T$, and the number of vertices is
$2T(\prod_{i} \Omega_{i})(\prod_{j} U_{j})$, and we now see the number of geodesics
through each point is indeed as claimed. \end{proof}

\begin{lemma}\label{counting lemma}
Hypothesis as in Lemma \ref{rank1generalBOX}.  Now further assume that $\mathbb{R}^{k} \rtimes_{\tau} \mathbb{R}$ is
unimodular, and let the common values of $\sum_{i} \alpha_{i}$ and $\sum_{j} \beta_{j}$ be $m$.  Let $X \subset
B[T]$ be a subset of vertices.  If $\mathcal{F}$ is a family of geodesics in $B[T]$ with
size $\sigma e^{2mT} \prod_{i,j} |\Omega_{i}| |U_{j}|$, where $\sigma \gg
\frac{2T}{e^{m\rho_{2}}}$, then there is a vertex $v \in X$, and two geodesics through $v$
in in the same direction that stay together for shorter than $\rho_{2}$ units.
\end{lemma}

\begin{proof}
Suppose the claim is not true. Then for each $X$ vertex $v$, every pairs of geodesics
through $v$ stay together at least $\rho_{2}$ units.  If $v$ is within $\rho_{2}$
neighborhood of top and bottom of $B[T]$, the number of geodesics through $v$ with
properties is $\frac{e^{(2T-h(v))m}}{e^{m\rho_{2}}}$, where $h$ is the height function on
$\mathbb{R}^{m} \rtimes \mathbb{R}$.  On the other hand, if $v$ is outside of $\rho_{2}$
neighborhood of top and bottom of $B[T]$, the number of geodesics through $v$ with this
property is  \[ \frac{e^{h(v)\sum_{i}\alpha_{i}}}{e^{m\rho_{2}}} \frac{e^{(2T-h(v))\sum_{j}\beta_{j}}}{e^{m\rho_{2}}}
= \frac{e^{2mT}}{e^{2m\rho_{2}}} \]

\noindent The number of $X$ vertices outside of $\rho_{2}$ neighborhood of the top and
bottom of $B[T]$ is at most $2(T-\rho_{2}) \prod_{i,j} |\Omega_{i}||U_{j}|$, and the
number of $X$ vertices within $\rho_{2}$ neighborhood of top and bottom of $B[T]$ is at
most $2\rho_{2} \prod_{i,j} |\Omega_{i}| |U_{j}|$.  So the number of geodesics satisfying
this scenario is at most

\[ \sum_{v \in N_{\rho_{2}}\bar{\partial}B[T]} \frac{e^{(2T-h(v))m}}{e^{m\rho_{2}}} +
\frac{e^{2mT}}{e^{2m\rho_{2}}} 2(T-\rho_{2}) \prod_{i,j} |\Omega_{i}||U_{j}|
\leq \frac{e^{2Tm}}{e^{m\rho_{2}}} (2T) \prod_{i,j} |\Omega_{i}| |U_{j}|  \]

\noindent Since the size of $\mathcal{F}$ is larger than this number, it is not possible
that every pairs of geodesics in $\mathcal{F}$ satisfy the scenario described above.  So
there must be a $X$ vertex $v$, and two geodesics in $\mathcal{F}$ that stay together for
less than $\rho_{2}$ units after passing through $v$.  \end{proof}

In the remaining of this subsection, given a regular vector $\vec{u}$, we
write $I_{\lambda, \vec{u}}(p)$, $p \in G'$, for the subset of $p W^{+}_{\vec{u}}$ that
can be reached by two geodesics of length $\lambda$ in direction $\vec{u}$ in the rank 1
space $G'_{\vec{u}}=\mathbf{H}' \rtimes \mathbb{R} \vec{u}$ containing $p$, and
$I'_{\lambda, \vec{u}}(p)$ for the subset of the left $W^{-}_{\vec{u}}$ coset that can be
reached from $p$ by a geodesic in direction $\vec{u}$ (or in direction $-\vec{u}$ as
viewed from $p$) of length $\lambda$. \bigskip

For the next two propositions, we make the following assumptions. \smallskip
\begin{enumerate}
\item Let $v$ be a $\mathfrak{b}$ vertex such that $\phi^{-1}(v)$ locally supports a standard map $\phi_{v}=f_{v} \rtimes g_{v}$.
Since $f_{v}$ is affine and permutes root classes of $G$ to root classes of $G'$, its linear part induces a permutation from the
chambers of $G$ to chambers of $G'$ and we write $(f_{v})_{*}$ for this permutation.

\item Suppose for $\ell \in (\mathbf{A})^{*}$ a regular linear functional, and $\vec{u} \in \mathbf{A}'$,
the vectors $\vec{v}_{\ell}$ and $(f_{v})_{*}^{-1} \vec{u}$ lie in a common chamber $\mathfrak{b}$ of
$G$. \end{enumerate}

\begin{proposition} \label{horocycle extension I}
If $\lambda$ is a number such that at least $\sigma$ fraction of geodesics leaving $v$ in
direction $\vec{u}$ are unobstructed by $H_{\ell}$ vertices for length at least $\lambda
+ \rho_{2}$, where $\sigma \gg 2 \lambda  / e^{\mathbf{b}_{\vec{u}} \rho_{2}}$.  Suppose $\sigma > O(\eta)$
for some $\eta < 1$, then at least $1-O(\eta)$ fraction of the $\hat{S}$ vertices in $I_{\lambda, \vec{u}}(v)$ are
$H_{\ell}$ vertices.  \end{proposition}

\begin{proof}
Let $E$ denote the set of geodesics leaving $v$ in direction $\vec{u}$ that are
unobstructed by $\ell$ vertices of length at least $\lambda + \rho_{2}$.  Let
$E_{\lambda}$ be the subset of $I'_{\lambda, \vec{u}}(v)$ passing through an element of
$E$.  By assumption, we have
\[ |E_{\lambda}| \geq \sigma e^{\mathbf{b}_{\vec{u}} \lambda} \]

\noindent Let $\mathcal{F}'_{0}$ be the union of geodesics leaving $E_{\lambda}$ in
direction $-\vec{u}$. (as viewed from $E_{\lambda}$).  Applying Lemma \ref{counting
lemma} to $\mathcal{F}'_{0}$, where $X$ means $\ell$ vertices, we see that either there
is a $\ell$ vertex whose $\vec{u}^{*}=\langle \vec{u}, \pi_{A} (\centerdot) \rangle$
value is at most $\rho_{2}$ from that of $v$, or that there is a $\ell$ vertex $w$ whose
$\vec{u}^{*}$ value differ from that of $v$ by more than $\rho_{2}$, and two geodesics in
$G_{\vec{u}}$ through $w_{1}$ that stay together for less than $\rho_{2}$ units after
passing through $w_{1}$.  Suppose the latter happens.  Let $x,y \in I'_{\lambda, \vec{u}}(v)$ be two upper end points
of those two geodesics, and $w_{1}$ be the first time that $\overline{wx}$ diverge from
$\overline{wy}$.  See Figure~\ref{fig:horoEXT1} below.

\begin{figure}[htbp]
\begin{center}

\input{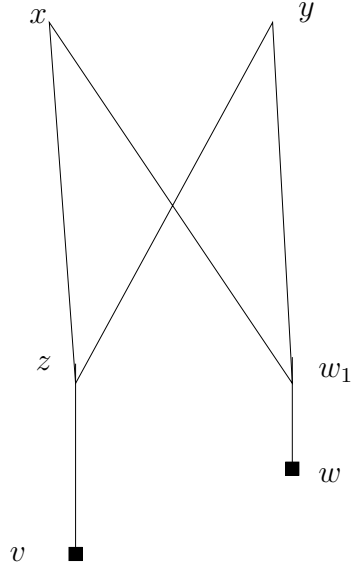}
\caption{The loop in the proof of Proposition \ref{horocycle extension I}.
Filled boxes are $H_{\ell}$ vertices. }
\label{fig:horoEXT1}

\end{center}
\end{figure}

Let $z$ be the first time that $\overline{vx}$ diverge from $\overline{vy}$.  Applying
Lemma \ref{illegal circuit} to the loop $z-y-w_{1}-x-z$ creates a contradiction.  So there
is a $\ell$ vertex whose $\vec{u}^{*}=\langle \vec{u}, \pi_{A} (\centerdot) \rangle$
value is at most $\rho_{2}$ from that of $v$.  That is, there is a $\ell$ vertex in the
$\rho_{2}$ neighborhood of $I_{\lambda, \vec{u}}(v)$. \smallskip

Let $U' \subset I_{\lambda, \vec{u}}(v)$ be those vertices that can be reached by two
elements of $\mathcal{F}'_{0}$.  Since every vertex in $I_{\lambda, \vec{u}}(v)$ can be
reached by at most $|E_{\lambda}|$ geodesics in $I'_{\lambda, \vec{u}}(v)$, it follows
that the relative measure of $U'$  in $I_{\lambda, \vec{u}}(v)$ is at least $1-O(\eta)$. \smallskip

Now suppose $w \in U'$, and let $x,y \in I'_{\lambda, \vec{u}}(v)$ be such that
$\overline{xw}, \overline{yw}$ are element of $\mathcal{F}'_{0}$ that are not obstructed
by $\ell$ vertices.  Applying Lemma \ref{illegal circuit} to the loop $v-x-w-y-v$, and
noting that $r(v)=0$, $r(x), r(y) \geq \rho_{2}$, it follows that $r(w)=0$ (otherwise
$r(v)$ would be the smallest, a contradiction), which means that $w$ is a $\ell$ vertex.
\end{proof}

\begin{proposition}\label{zero-one law}
(Hypothesis as in Proposition \ref{horocycle extension I})  Let $\mathcal{F}$ be the
union of geodesics in direction $\vec{u}$ leaving a point of $I_{\lambda, \vec{u}}(v)$.
Then at least $1-O(\eta)$ fraction of $\mathcal{F}$ are unobstructed by $H_{\ell}$
vertices for length $\lambda + \rho_{2}$.   \end{proposition}

\begin{proof}
Let $E_{\lambda} \subset I'_{\lambda, \vec{u}}(v)$, $U' \subset I_{\lambda, \vec{u}}(v)$
and $\mathcal{F}'_{0}$ be as in Proposition \ref{horocycle extension I}.  Note that the
measure of $U'$ is at least $1-O(\eta)$ that of $I_{\lambda, \vec{u}}(v)$.  \smallskip

Let $\mathcal{F}''$ be the set of geodesics leaving $U'$ in direction $\vec{u}$. Let $\mathcal{F}''_{long}$ be the set of
geodesics coming from extending elements of $\mathcal{F}''$ by extending $\rho_{4}$ units on the
$I_{\lambda, \vec{u}}(v)$\symbolfootnote[2]{Ideally, we would like to apply Lemma \ref{counting lemma} to the family $\mathcal{F}''$
in a rank 1 box of size $\lambda$ but in order to use illegal circuit we need the stub - vertical segment from $H_{\ell}$ instead
of being on $H_{\ell}$. Hence the choice of $H'_{\ell}$ lower down.}.  Let $H'_{\ell}$ be the $\ell$ block whose $\ell$ value is
$\rho_{4}$ less than $\ell(H_{\ell})$.  We call the resulting vertices $\ell ^{'}$ vertices if we replace occurrence of $H_{\ell}$
 in the definition of $\ell$ vertices by $H'_{\ell}$.

If all elements of $\mathcal{F}''_{long}$ are unobstructed by images of
$\ell^{'}$ vertices, then almost all elements of $\mathcal{F}$ are unobstructed by $\ell$
vertices, where `almost' here means with relative proportion at least $1-O(\eta)$.  Let $U'_{long}$ be the set of
$\ell^{'}$ vertices that are within $\rho_{4}$ of $U'$. \smallskip

We have that $|\mathcal{F}''_{long}| \geq (1-O(\eta)) e^{\mathbf{b}_{\vec{u}} (2\lambda +
\rho_{4})}$.  Applying Lemma \ref{counting lemma} allows us to conclude that either there
is a $\ell$ vertex whose $\vec{u}^{*}$ value is within $\rho_{2}$ to $\partial \mathbf{B}
\cap G'_{\vec{u}}$, or that there is a $\ell$ vertex $q$ whose $\vec{u}^{*}$ value differ
from that of $\partial \mathbf{B} \cap G'_{\vec{u}}$ by more than $\rho_{2}$ units, and
two elements of $\mathcal{F}''_{long}$ that stay together for less than $\rho_{2}$ units
after passing through $q$.

Suppose the latter scenario happens.  Then there are $w_{1}, w_{2} \in I_{\lambda,
\vec{u}}(v)$ and a $\ell^{'}$ vertex $q$ such that $\overline{w_{1}q}$,
$\overline{w_{2}q} \in \mathcal{F}''_{long}$.  Let $q_{*}$ be the first point where
$\overline{w_{1}q}$ and $\overline{w_{2}q}$ come together.  Then by assumption, the
$d(q,q_{*}) < \rho_{2}$.  Let $x_{1} \in I'_{\lambda, \vec{u}}(v)$ be the first point where geodesics in direction
$\vec{u}$ leaving $w_{1}$ and $v$ first meet, and let $x_{2} \in I'_{\lambda,
\vec{u}}(v)$ be similarly defined for $w_{2}$ and $v$.  Let $r(\centerdot)$ now denotes
for the distance to the closest $\ell^{'}$ vertex.  Then in the loop
$v-x_{1}-w_{1}-q_{*}-w_{2}-x_{2}-v$, (see Figure~\ref{fig:horoEXT2}) the $r$ value of all points but $q_{*}$ are at least
$\rho_{4}$, while $r(q_{*}) \leq \rho_{2}$, which is a contradiction by Lemma
\ref{illegal circuit}.

\begin{figure}[htbp]
\begin{center}

\input{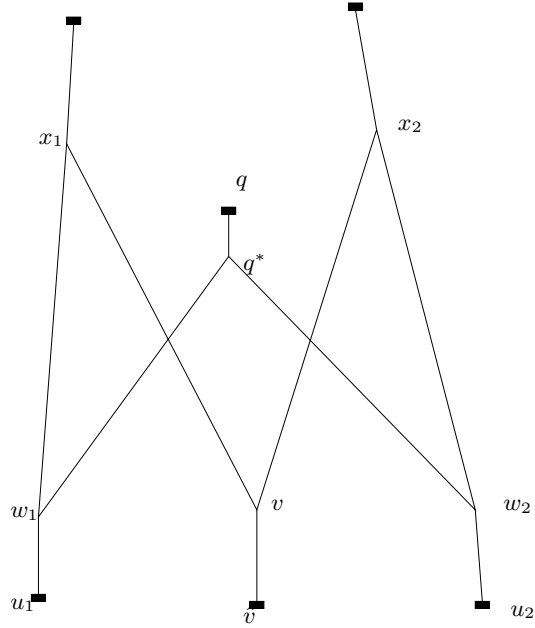}
\caption{The loop in the Proof of Proposition \ref{zero-one law}.  Filled boxes
are $H_{\ell'}$ vertices. }
\label{fig:horoEXT2}
\end{center}
\end{figure}

Therefore if elements of $\mathcal{F}''_{long}$ is to contain a $\ell'$ vertex, this vertex is within $\rho_{2}$ neighborhood
of $\partial \mathbf{B} \cap G'_{\vec{u}}$, which is just saying that no elements of
$\mathcal{F}''_{long}$ are obstructed by $\ell'$ vertices, therefore at least $1-O(\eta)$
proportion of elements in $\mathcal{F}$ are unobstructed by $\ell$ vertices by construction. \end{proof}

\begin{theorem} \label{horocycle extenion II}
Let $v$ be a $\mathfrak{b}$ vertex and $\hat{\phi}_{v}=f_{v} \times g_{v}$ be a standard map supported in a neighborhood of
$\phi^{-1}(v)$.  Let $\lambda_{0}=d(v, \partial \mathbf{B})$.  Then at least $1-O(\eta)$ proportion of vertices in
$v W^{+}_{(f_{v})_{*} \mathfrak{b}}$ are $\mathfrak{b}$ vertices, where $O(\eta)$ is the proportion of
geodesics leaving $\phi^{-1}(v)$ that admits a geodesic approximation of length  $\Omega(diam(\omega))$. \end{theorem}

\begin{proof}
Let $s_{\vec{u}}$ be the difference in $v_{\vec{u}}^{*}$ values between $\partial
\mathbf{B} \cap G'_{\vec{u}}$ and $v$. \smallskip

Fir a $\rho_{2} = \hat{\eta} s_{\vec{u}}$.  For every $\ell \in \mathcal{S}_{0}$, and $w \in I_{\lambda, \vec{u}}(v)$,
we let $f_{\ell}(w, \lambda)$ denotes the proportion of geodesics leaving $w$ that are
unobstructed by $H_{\ell}$ vertices for length at least $\lambda + \rho_{2}$.  Let
\[ f^{*}_{\ell}(v, \lambda) = \sup_{w \in I_{\lambda, \vec{u}}} f_{\ell}(w, \lambda) \]

\noindent In view of Proposition \ref{horocycle extension I} and \ref{zero-one law}, if $f^{*}_{\ell}(v, \lambda) \geq
O(\eta)$, then $f^{*}_{\ell}(v, \lambda) \geq 1-O(\eta)$. \smallskip

Then, either for all $\lambda \leq s_{\vec{u}}$, we have $f^{*}_{\ell} \geq 1-O(\eta)$; or that there is a maximal number
$\lambda_{\vec{u}, \ell} -1$ such that $f^{*}_{\ell}(v, \lambda_{\vec{u}, \ell}-1) \geq 1-O(\eta)$,
but $f^{*}_{\ell}(v, \lambda_{\vec{u}, \ell}) < O(\eta)$.  We are done if the latter does not
happen.  We now proceed to show that this is indeed the case.

In the second scenario, we know that $\lambda_{\vec{u}, \ell} \geq \hat{\eta} diam(\omega)$, and at least $1-O(\eta)$ proportion of
vertices in $I_{\lambda_{\vec{u},\ell}, \vec{u}}(v)$ and $I'_{\lambda_{\vec{u},\ell}, \vec{u}}(v)$ are $H_{\ell}$ vertices.
That is, they are $\phi$ images of $U_{*} \cap \left( x_{0} W^{+}_{\mathfrak{b}} \rtimes ker(\ell) \right)$. \smallskip

\bold{Claim: } If $\lambda_{\vec{u}, \xi} > \lambda_{\vec{u}, \ell}$ for $\xi, \ell \in \mathcal{S}_{0}$, then
$\lambda_{\vec{u}, \xi} - \lambda_{\vec{u}, \ell} > O(\hat{\eta} diam(\omega))$. \newline

\noindent Suppose not.  Then we will have subsets, one in $ker(\xi)$ and one in
$ker(\ell)$ that are within $O(\hat{\eta} diam(\omega))$ Hausdorff distance from each
other.  This can only happen if the subsets are within $O(\hat{\eta} diam(\omega))$ of
$ker(\xi) \cap ker(\ell)$. But this would mean that most of $I'_{\lambda_{\vec{u}, \ell},
\vec{u}}(v)$ come from $\phi$ images of $x_{0}W^{+}_{\mathfrak{b}} \rtimes (ker(\xi)\cap
ker(\ell))$, contradicting the assumption that $\lambda_{\vec{u}, \xi}$ is the minimal
height $t$ where most of the $I'_{t, \vec{u}}(v)$ are obstructed by $H_{\xi}$ vertices.\\

In this way, the image of the map $\mathcal{S}_{0} \rightarrow [0,s_{\vec{u}}]$ defined by sending
$\xi \rightarrow \lambda_{\vec{u}, \xi}$ is a $O(\hat{\eta} diam(\omega))$ discrete set.  Let $\hat{\lambda}_{\vec{u}}$ be the minimal
image value whose pre-images has positive measure.  This means that most elements of
$I'_{\hat{\lambda}_{\vec{u}}, \vec{u}}(v)$ and $I_{\hat{\lambda}_{\vec{u}}, \vec{u}}(v)$
are $\phi$ images of $U_{*} \cap x_{0}W^{+}_{\mathfrak{b}}$, thus the subset of
$\mathcal{S}_{0}$ consisting of elements $\xi$ such that $\lambda_{\vec{u},
\xi} > \hat{\lambda}_{\vec{u}}$ is empty.  Since for all $t < \hat{\lambda}_{\vec{u}}$,
the pre-images of $t$ in $\mathcal{S}_{0}$ has zero measure, this means not only does the
pre-images of $\hat{\lambda}_{\vec{u}}$ has positive measure, it has full measure. \smallskip

Now pick another direction $\vec{u}'$ in the same chamber as $\vec{u}$, but not in the
$\hat{\eta}^{1/2}$ neighborhood of the $\vec{u}$ orbit under the finite group of affine maps permuting $R_{g}$ to $R_{G'}$,
and repeat the same argument as above to obtain a number $\hat{\lambda}_{\vec{u}'}$ such that most of
$I'_{\hat{\lambda}_{\vec{u}'}, \vec{u}'}(v)$ come from $\phi$ images of
$x_{0}W^{+}_{\mathfrak{b}}$. \smallskip

Pick $y_{0} \in I'_{\hat{\lambda}_{\vec{u}}}(v)$ and $y \in
I'_{\hat{\lambda}_{\vec{u}'}}(v)$ so that each locally supports a standard map.  Take
two geodesics leaving $y_{0}$ in direction $\vec{u}$ (as viewed from $v$) that stay
together for $t_{\vec{u}}$ units (where $\hat{\eta}^{1/2}diam(\omega) \ll  t_{\vec{u}} \ll diam(\omega)$)
after they leave $y_{0}$, followed by a short segment $diam(\omega)$ away from $v$,
before joining the geodesics connecting $v$ to $y$. Let's say they stay together for
$t_{\vec{u}'}$ units before coming to a stop at $y$.  See Figure~\ref{fig:block1}

\begin{figure}[htbp]
\begin{center}

\input{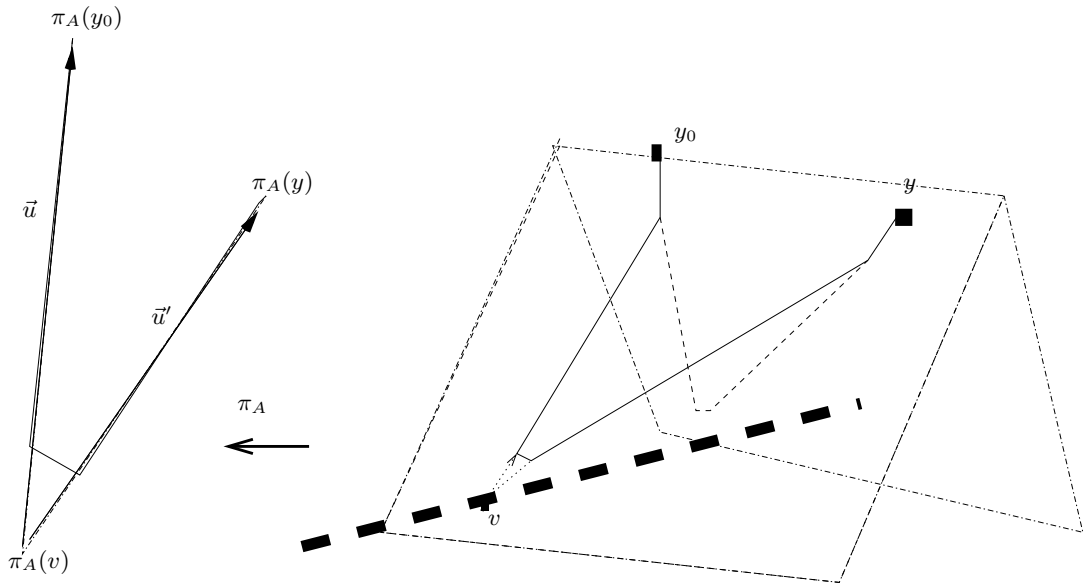}

\caption{The loop that prevents blocking.  Filled boxes
represent $\mathfrak{b}$ vertices. The left hand is the $\pi_{A}$ projection
image of the loop on the right.}
\label{fig:block1}
\end{center}
\end{figure}

As most of $I'_{\hat{\lambda}_{\vec{u}}}(v)$ and $I'_{\hat{\lambda}_{\vec{u}'}}(v)$ come
from $U_{*} \cap W^{+}_{\mathfrak{b}}$, for a full measure of $\ell \in \mathfrak{b}$,
Lemma \ref{illegal circuit} requires us to have $\Pi_{\ell} \phi_{y_{0}}^{-1}(t_{\vec{u}}
\vec{u}) \geq \Pi_{\ell} \phi_{y}^{-1} (t_{\vec{u}'} \vec{u}')$, as well as $\Pi_{\ell}
\phi_{y_{0}}^{-1}(t_{\vec{u}} \vec{u}) \leq \Pi_{\ell} \phi_{y}^{-1} (t_{\vec{u}'}
\vec{u}')$.  This means that $f_{y_{0}}^{-1}(t_{\vec{u}} \vec{u}) = f_{y}^{-1}
(t_{\vec{u}'} \vec{u}')$, where $f_{y_{0}}$ and $f_{y}$ are linear part of standard maps
$\phi_{y_{0}}$ and $\phi_{y}$.  That is, $t_{\vec{u}}/t_{\vec{u}'} \in [1/(1-\hat{\eta}),
1+ \hat{\eta}]$, and $t_{\vec{u}} \vec{u}$ lies in the $\hat{\eta}diam(\omega)$ neighborhood of the orbit of $t_{\vec{u}'} \vec{u}'$ under the finite group of
linear maps that permutes $R_{G}$ to $R_{G'}$.  But this contradicts our choice of
$\vec{u}'$ and $t_{\vec{u}}$.   \end{proof}

\begin{proof} \textit{of Theorem \ref{aligning horocycle}  }
As any root class horocycle is the intersection of finitely many left translates of $W^{+}_{\mathfrak{b}}$, where $\mathfrak{b}$ is a chamber, it suffices to show that the
analogue claim holds for left translates of $W^{+}_{\mathfrak{b}'}$ of $G'$ in place of left
translates of $V_{[\alpha]'}$ of $G'$. \smallskip

We start with $\mathbf{B}(\Sigma)$ sufficiently large so that we can apply Theorem \ref{exisence of standard maps in small
boxes} to $\phi^{-1}(\mathbf{B}(\Sigma))$ and obtaining a tiling as in equation
(\ref{tileB}).  Let $U_{*}=\bigcup_{i \in \mathbf{I}_{0}} U_{i}$.  Since $U_{*}$ has
large measure relative to $\phi^{-1}(\mathbf{B}(\Sigma))$, for a fixed chamber
$\mathfrak{b}$, we can find a large subset $U_{\mathfrak{b}} \subset U_{*}$ with the
property that for every point $p \in U_{\mathfrak{b}}$, there is a subset
$\mathcal{S}_{0} \subset \mathfrak{b}$ of relative proportion at least
$1-\vartheta$ such that $E_{\ell}(p)$ is very favourable for every $\ell \in
\mathcal{S}_{0}$. By constructing the corresponding $\hat{S}$ and $H$ graph, application
of Theorem \ref{horocycle extenion II} to a point $v \in \phi(U_{\mathfrak{b}})$ shows
that the $\phi^{-1}$ image of $v W^{+}_{\mathfrak{c}} \cap \mathbf{B}(\Sigma)$ is $O(\hat{\eta} diam(\mathbf{B}(\Sigma) )$ away from
a left translate of $W^{+}_{\mathfrak{b}}$, where $\mathfrak{c}$ is the image of $\mathfrak{b}$ under
the linear part of the standard map supported in a neighborhood of $\phi^{-1}(v)$. \end{proof}

\section{Patching}
In the previous section, we aligned the linear part of standard maps appeared in Theorem \ref{exisence of standard maps in small boxes} by showing that they are all the
same.  In this last section, we remove the condition that standard maps are only defined for a subset of relative large
measure and align the translational part of the standard maps by adopting the procedure used to achieve this in \cite{EFW1}.

\subsection{A weak version of an affine map}
We have by now seen that given a box, there is a subset of large measure supporting a standard map.  In this subsection, by
controlling the sizes of increasingly larger and larger boxes, we get rid of the 'subset of large relative measure' and extend the
result to all pairs of points $p,q \in G$ on the same flat.  The precise statement is the following:

\begin{theorem} \label{weak height preservation}
Let $G, G'$ be non-degenerate, unimodular, split abelian-by-abelian Lie groups and $\phi:
G \rightarrow G'$ be a $(\kappa, C)$ quasi-isometry between them.  Given $0< \delta, \eta \ll \tilde{\eta} <1 $, then there
exists $\tau < 1$, $M$ depending on $\delta$, $\eta$, $\tilde{\eta}$ and QI constants of
$\phi$ such that whenever $x, y $ belong to the same left coset of $\mathbf{H}$,
\begin{equation}\label{weak height eq}
\left| \pi_{A}(\phi(x)) - \pi_{A}(\phi(y)) \right| \leq \tau d(x,y) + M \end{equation}

\noindent where $\tau \rightarrow 0$, $M \rightarrow \infty$ as the input parameters approach zeros. \end{theorem}

\noindent The setup to the proof of theorem \ref{weak height preservation} follows the same sequence of steps as the analogue
result in Section 6.1 of \cite{EFW1}. \bigskip

Fix $0< \delta, \eta \ll \tilde{\eta} < 1$.  Let $\Omega \subset \mathbf{A}$ be a product of intervals of size $L_{0}$ with barycenter
located at the origin of $\mathbf{A}$.  By Corollary \ref{strengthening part 2}, there is a subset
$\mathcal{P}^{0} \subset \mathcal{P}(\Omega)$ of relative large measure which is the support of a standard map $g \times f$
where $f$ is affine.  Let $\vartheta \ll 1$ satisfies $\left| \mathcal{P}^{0} \right| \geq (1-\vartheta) \left| \mathcal{P}(\Omega) \right|$,
and set $\varrho=\sqrt{\vartheta}$.  Let $\mathbf{P}$ be a left translate of $\mathbf{H}$, and we can assume $\mathbf{P}$ contains the
identity element.  Let $R(\Omega)=\mathbf{B}(\Omega) \cap \mathbf{P}$, and $\mathbf{R}(\Omega)=\bigcup_{g \in R(\Omega)} g(\varrho \Omega)$.
 The following is a rehash of Corollary \ref{strengthening part 2}.

\begin{corollary} \label{implication of refinement}
There is a standard map $f \times g$ where $f$ is affine with linear part $A_{f}$, defined on $\mathcal{P}^{0} \subset \mathcal{P}(\mathbf{R}(\Omega))$, with
$\left| \mathcal{P}^{0} \right| \geq (1-\varrho) \left| \mathcal{P}(\mathbf{R}(\Omega)) \right|$, such that
$d(\phi|_{\mathcal{P}^{0}}, f \times g) \leq \hat{\eta} diam(\mathbf{B}(\Omega))$.

Furthermore, if $p \in \mathcal{P}^{0}$, there is a subset $L^{0}(p) \subset L(p)$ of relative large measure such that the
$\phi$ image of every element $\zeta \in L^{0}(p)$ is within $\hat{\eta}$-linear neighborhood of a geodesic
segment in direction that of $A_{f} \circ \zeta$, whose direction makes an angle of at most $\sin^{-1}(\tilde{\eta})$ with root kernels.
\end{corollary}

\bold{ The tiling.  } Choose $\hat{\eta} \ll \varsigma \ll 1$. For
each $j \in \mathbb{N}$, set $\Omega_{j}=(1+\varsigma)^{j} \Omega$.
We tile $\mathbf{P}$ by $R(\Omega_{j})$, where each tile is denoted
by $R_{j,\iota}$, $\iota \in \mathbb{N}$. For $p \in G$, we write
$R_{j}[x]$ for the tile in the $j$-th tiling containing the point in
$\mathbf{P}$ that lies on the same flat as $x$.

\bold{ Warning. }  Despite the fact that
$\Omega_{j+1}=(1+\varsigma)\Omega_{j}$, the number $R_{j,k}$'s
needed to cover a rectangle $R_{j+1}[p]$ is very large, on the order
of $e^{\sum_{i=1}^{|\triangle|} \varsigma \max\{\alpha_{i}(\Omega)\}}$.

\bold{ The sets $U_{j}$ }  For each tile $R_{j,k}$ in the $j$-th
tiling of $\mathbf{P}$, Corollary \ref{implication of refinement}
produces a subset $\mathcal{P}^{0}_{j,k} \subset
\mathcal{P}(\mathbf{R}_{j,k})$ of relative large measure. We set
\[ U_{j} =\bigcup_{k \in \mathbb{N}} \mathcal{P}^{0}_{j,k} \]

\noindent In view of Corollary \ref{implication of refinement}, for
any $x \in U_{j}$,
\begin{equation} \label{in the same box}
\sup_{y \in \mathbf{R}[x] \cap U_{j}} \left| \pi_{A}(\phi(x))-\pi_{A}(\phi(y))  \right| \leq \hat{\eta} diam(\mathbf{B}(\Omega_{j}))
\end{equation}

\noindent  Recall that $\triangle'$ is the set of roots in $G'$, and we write
$n$ for the rank of $G$ (which is also the rank of $G'$).  We also have the following generalization:

\begin{lemma} \label{two small in a big}
For any $x \in U_{j}$, and $y \in \mathbf{R}_{j+1}[x] \cap U_{j}$,
\[ \left| \pi_{A}(\phi(x)) - \pi_{A}(\phi(y)) \right| \leq   (|\triangle'|)(n) 4 \hat{\eta} diam(\mathbf{B}(\Omega_{j})) \] \end{lemma}

\begin{proof}
WLOG, we can assume that there is a horocycle $H_{[\alpha]}$, $[\alpha]$ an equivalence class of roots, that intersect both
$\mathbf{R}[x] \cap U_{j}$ and $\mathbf{R}[y] \cap U_{j}$.  (If not, then we can find a sequence of points $x=p_{0}, p_{1}, \cdots ,
p_{l}=y$ where $l \leq |\triangle'|$ such that for each pair of consecutive points, there is a horocycle intersecting
$\mathbf{R}[p_{\iota}] \cap U_{j}$ and $\mathbf{R}[p_{\iota+1}] \cap U_{j}$).  Let $x_{1} \in H_{[\alpha]} \cap \mathbf{R}[x] \cap
U_{j}$, $y_{1} \in H_{[\alpha]} \cap \mathbf{R}[y] \cap U_{j}$ be points of intersection.  Therefore by equation (\ref{in the same box}),
\[ \left| \pi_{A}(\phi(x)) - \pi_{A}(\phi(x_{1})) \right| \leq \varrho diam(\mathbf{B}(\Omega_{j})),
\hspace{3mm} \left| \pi_{A}(\phi(y)) - \pi_{A}(\phi(y_{1})) \right| \leq \varrho diam(\mathbf{B}(\Omega_{j})) \]

\noindent  Since $d(x_{1},y_{1}) \leq diam(\mathbf{B}(\Omega_{j+1}))$, for $\iota=1,2$ we can find geodesic segments $\gamma_{x_{1},\iota}$,
$\gamma_{y_{1},\iota}$ leaving $x_{1},y_{1}$ respectively such that $Q=\{ \gamma_{x_{1}, \iota}, \gamma_{y_{1}, \iota} \}_{\iota=1,2}$
is a $0$-quadrilateral.  Additionally, because $x_{1}, y_{1} \in U_{j} \cap H_{[\alpha]}$, we can assume for $\iota=1,2$,
$*=x_{1},y_{1}$, the subsegment $\hat{\gamma}_{*, \iota} \subset \gamma_{*,\iota}$ containing $*$ and satisfies $\left|
\hat{\gamma}_{*, \iota} \right| = \frac{1}{1+ \varsigma} \left| \gamma_{*,\iota} \right|$ admit geodesic approximation to its $\phi$
image.  That is, $\phi(\hat{\gamma}_{*,\iota})$ is within $\eta |\hat{\gamma}_{*,\iota} |$ Hausdorff neighborhood of another
geodesic segment. \smallskip

Let $\mathit{l}_{*,\iota}$ be a geodesic approximation to $\phi(\hat{\gamma}_{*,\iota})$. Then angle between the direction of
$\mathit{l}_{*,\iota}$ and that of $\hat{\gamma}_{*,\iota}$ is at most $\sin^{-1}(\hat{\eta})$.  Therefore by modifying each
$\mathit{l}_{*,\iota}$ by an amount at most $\hat{\eta} 2\kappa diam(Q) \leq \hat{\eta} diam(\mathbf{B}(\Omega_{j+1}))$, we can
assume $\mathit{l}_{*,\iota}$ all have parallel directions.  Since $\varsigma \ll 1$, the four geodesic segment $\mathit{l}_{*,\iota}$
do constitute a quadrilateral $\tilde{Q}$ and Lemma \ref{Forste-structure of a quadrilateral} of \cite{Pg} applied to $\tilde{Q}$ yields
\begin{equation*} \left| \pi_{\vec{v}} \circ \Pi_{\vec{v}}
(\phi(x_{1})) - \pi_{\vec{v}} \circ \Pi_{\vec{v}} (\phi(y_{1}))
\right| \leq \hat{\eta} diam(\mathbf{B}(\Omega_{j}) \end{equation*}

\noindent where $\vec{v}$ is parallel to edge directions of $\tilde{Q}$.  Therefore
\begin{equation*} \left| \pi_{\vec{v}} \circ
\Pi_{\vec{v}} (\phi(x)) - \pi_{\vec{v}} \circ \Pi_{\vec{v}} (\phi(y)) \right| \leq \hat{\eta} diam(\mathbf{B}(\Omega_{j}) + 2
\varrho diam(\mathbf{B}(\Omega_{j})) \leq  4 \hat{\eta} diam(\mathbf{B}(\Omega_{j})) \end{equation*}

\noindent since $\varrho \leq \hat{\eta}$.  The claim now follows by constructing quadrilaterals whose image is approximated by
quadrilaterals whose edge direction ranges over at least $n$ many linearly independent directions. \end{proof}

\begin{lemma} \label{one big one small}
Suppose $p \in \mathbf{R}_{j}[x] \cap U_{j}$, $q \in \mathbf{R}_{j+1}[x] \cap U_{j+1}$.  Then,
\[ \left| \pi_{A}(\phi(x)) - \pi_{A}(\phi(y)) \right| \leq (4|\triangle| n +2) \hat{\eta} diam(\mathbf{B}(\Omega_{j+1}))  \]  \end{lemma}

\begin{proof}
Note that the projection of both $\mathbf{R}_{j+1}[x] \cap U_{j}$ and $\mathbf{R}_{j+1}[x] \cap U_{j+1}$ to $R_{j+1}[x]$ have relative
measure up at least $1-\varrho$.  Therefore we can find $p' \in \mathbf{R}_{j+1}[x] \cap U_{j}$, $q' \in \mathbf{R}_{j+1}[x] \cap U_{j+1}$,
and $\left| \pi_{A}(\phi(p')) - \pi_{A}(\phi(q')) \right| \leq \varrho diam(\mathbf{B}(\Omega_{j+1}))$.   Now by Lemma \ref{two small in a big},
\[ \left| \pi_{A}(\phi(p)) - \pi_{A}(\phi(p')) \right| \leq (|\triangle|)n4 \hat{\eta} diam(\mathbf{B}(\Omega_{j})) \]

\noindent and by equation (\ref{in the same box})
\[  \left| \pi_{A}(\phi(q')) - \pi_{A}(\phi(q)) \right| \leq \hat{\eta} diam(\mathbf{B}(\Omega_{j+1})) \]

\noindent the claim now follows by triangle inequality. \end{proof}

\begin{proof} \textit{to Theorem \ref{weak height preservation} }
We have
\[ R_{0}[x] \subset R_{1}[x] \subset R_{2}[x] \cdots \]

\noindent and

\[ R_{0}[y] \subset R_{1}[y] \subset R_{2}[x] \cdots \]

\noindent There exists $N$ with diam$(\mathbf{B}(\Omega_{N}))$ is comparable to $d(x,y)$ (after possibly shifting the $N$'s grid by a
bit) such that $R_{N}[x]=R_{N}[y]$.  Now for $0 \leq j \leq N$, pick $x_{j} \in \mathbf{R}_{j}[x] \cap U_{j}$, $y_{j} \in
\mathbf{R}_{j}[y] \cap U_{j}$.  We may assume that $x_{N}=y_{N}$. By Lemma \ref{one big one small}:
\begin{eqnarray*}
|\pi_{A}(\phi(x_{0}))-\pi_{A}(\phi(y_{0}))| & \leq &
\sum_{j=0}^{N-1} |\pi_{A}(\phi(x_{j+1})) - \pi_{A}(\phi(x_{j}))| +
\sum_{j=0}^{N-1}|\pi_{A}(\phi(y_{j+1}))-\pi_{A}(\phi(y_{j}))|\\
&\leq & 2 (4 |\triangle| n + 2) \sum_{j=0}^{N-1} \hat{\eta}
diam(\mathbf{B}(\Omega_{j+1})) \\
& \leq & 4 (2 |\triangle| n + 1) \frac{\hat{\eta}}{\varsigma}
diam(\mathbf{B}(\Omega_{N})) \end{eqnarray*}

\noindent where last inequality comes from $diam(\mathbf{B}(\Omega_{j+1}))=(1+ \varsigma) diam(\mathbf{B}(\Omega_{j}))$.  Now since
$x_{0} \in R_{0}[x]$ $\left| \pi_{A}(\phi(x)) - \pi_{A}(\phi(x_{0})) \right| \leq 2\kappa d(x,x_{0}) \leq 2\kappa diam(\mathbf{B}(\Omega))=M $,
and similarly $\left| \pi_{A}(\phi(y)) - \pi_{A}(\phi(y_{0})) \right| \leq M$. The claim now follows by noting that we chose
$\frac{diam(\mathbf{B}(\Omega_{N}))}{d(x,y)} \in [1/2, 2]$ and $\varsigma \gg \hat{\eta}$ \end{proof}

\subsection{Consequence of weak height preservation - flats go to flats }
Theorem \ref{weak height preservation} is the first statement we have that places no additional constraints on the points other than
their natural relation in $G$. In this subsection we show that as a first consequence, the image of a flat is within $O(1)$ of another flat,
which eventually culminating in the proof of theorem \ref{behaviors of phi} in the next subsection.

\begin{proposition} \label{FTF} $\phi$ sends each flat to within $O(1)$ of another flat. \end{proposition}

We now proceed to establish some necessary observations.
\begin{lemma} \label{weak affine map}
There is a linear map $A_{0}: \mathbf{A} \rightarrow \mathbf{A}'$ and numbers $\hat{\tau} < 1$, $\hat{M} >0$ such that for any $x, y \in G$,
\begin{equation}\label{weak affine eq}
\left|(\pi_{A} (\phi(x)) - \pi_{A}(\phi(y))) - (A_{0}(\pi_{A}(x)) - A_{0}(\pi_{A}(y))) \right| \leq \hat{\tau} d(x,y) + \hat{M}
\end{equation} \end{lemma}

\begin{proof}
Let $\mathbf{B}$ be a box such that $diam(\mathbf{B})/d(x,y) \in [1/2,2]$.  By Corollary \ref{strengthening part 2}, there is a
subset $\mathcal{P}^{0} \subset \mathcal{P}(\mathbf{B})$ of relative measure at least $(1-\sqrt{\hat{\eta}})$ that supports a standard
map which is $\hat{\eta}diam(\mathbf{B})$ away from $\phi|_{\mathcal{P}^{0}}$.  Write
$A_{0}$ for the linear part of the $\mathbf{A}'$ part of the standard map.  Without loss of generality we can assume that
$x \mathbf{H} \cap \mathcal{P}^{0} \not= \emptyset$, and $y \mathbf{H} \cap \mathcal{P}^{0} \not= \emptyset$.
Let $\hat{x} \in x\mathbf{H} \cap \mathcal{P}^{0}$, and $\hat{y} \in y \mathbf{H} \cap \mathcal{P}^{0}$.  Then by
Theorem \ref{weak height preservation}
\[ \left| \left(  \pi_{A}(\phi(\hat{y})) - \pi_{A}(\phi(y)) \right) -
\left(   A_{0}(\pi_{A}(\hat{y})) - A_{0}(\pi_{A}(y)) \right) \right| \leq \tau d(\hat{y},y) + M\]

\noindent since $\pi_{A}(\hat{y})=\pi_{A}(y)$.  By Corollary \ref{strengthening part 2} we also have

\[ \left| (\pi_{A}(\phi(\hat{x})) - \pi_{A}(\phi(\hat{y}))) - (A_{0}(\pi_{A}(\hat{x})) - A_{0}(\pi_{A}(\hat{y}))) \right|
\leq \hat{\eta} diam(\mathbf{B}) \]

\noindent and by Theorem \ref{weak height preservation} again, we have
\[ \left| (\pi_{A}(\phi(x)) - \pi_{A}(\phi(\hat{x}))) - (A_{0}(\pi_{A}(x))-A_{0}(\pi_{A}(\hat{x}))) \right| \leq \tau d(x, \hat{x}) + M \]

\noindent Summing all three equations and apply triangle inequality to the left hand side we have
\begin{eqnarray*}
\left| (\pi_{A}(\phi(x))-\pi_{A}(\phi(y))) - (A_{0}(\pi_{A}(x))- A_{0}(\pi_{A}(y))) \right| &\leq& (2\tau + \eta) diam(\mathbf{B}) + 2M \\
&\leq&  2(2\tau + \hat{\eta}) d(x,y) + 2M \end{eqnarray*}

\noindent since $\hat{\eta}$ and $\tau$ depends on our initial $\epsilon, \delta$, and approach zero as those initial input approach zero,
we can assume that $2(2\tau + \hat{\eta}) < 1$, and we set $\hat{\tau}=2(2\tau + \hat{\eta})$, $\hat{M}=2M$.  \end{proof}

\begin{corollary} \label{the same piA}
There is a number $M_{0}$ such that if $p,q \in G$ are two points on the same flat and $\pi_{A}(\phi(p))=\pi_{A}(\phi(q))$, then
$d(x,y)\leq M_{0}$.  \end{corollary}

\begin{proof}
In equation (\ref{weak affine eq}) substitute $x,y$ by $p,q$ and note that since $p,q$ lies on the same flat, $\left| \pi_{A}(p) -
\pi_{A}(q) \right| = d(p,q)$. So we have \[ (1-\hat{\tau}) d(p,q) \leq \hat{M} \]

\noindent to which the result follows.   Alternatively, this can be obtained from Theorem \ref{weak height preservation} applied to the
inverse map $\phi^{-1}$ and $\phi(p)$, $\phi(q)$, for then equation (\ref{weak height eq}) becomes
\begin{eqnarray*}
\left| \pi_{A} \circ \phi^{-1} (\phi(p)) - \pi_{A} \circ \phi^{-1}(\phi(q)) \right| & \leq & \tau d(\phi(p), \phi(q)) + M
\leq \tau \kappa d(p,q) + M+C \\
\left| \pi_{A}(p) - \pi_{A}(q) \right| = d(p,q) &\leq& \tau \kappa d(p,q) + M+C \\
d(p,q) &\leq& \frac{M+C}{1-\tau \kappa} \end{eqnarray*} \end{proof}

A subset $L$ of $\mathbf{A}' \simeq \mathbb{R}^{n}$ is called a 'grid' if it is the image
of an injective homomorphism $\psi: \mathbb{Z}^{n} \rightarrow \mathbf{A}'$.  A line in $L$ refers to a subset of the form:
$\{ \psi(\mathbf{c}+ t \mathbf{u}) : t \in \mathbb{Z} \}$ for some $\mathbf{c}, \mathbf{u} \in \mathbb{Z}^{n}$, and each
coordinate of $\mathbf{u}$ is either $+1, -1$ or $0$.  A grid is said to be good if none of its lines are parallel to root kernels.

\begin{lemma} \label{preleude to FTF}
Let $\{ x_{i} \} \subset G'$ be a sequence of points with the following properties:
\begin{enumerate}
\item $\pi_{A}(x_{j}) \not= \pi_{A}(x_{i})$ if $i \not= j$. \item $\{ \pi_{A}(x_{i}) \} \subset \mathbf{A}'$ is a good grid.
\item for any subsequence $\{x_{i_{j}} \}$ such that $\{ \pi_{A}(x_{i_{j}}) \}$ is a line, $\{x_{i_{j}} \}$ is within $O(1)$ of a
(bi-infinite) geodesic.  \end{enumerate}

Then $\{ x_{i} \}$ is within $O(1)$ of a flat. \end{lemma}

\begin{proof}
Write $\{ \pi_{A}(x_{i}) \}=L$.  Let $\{ x_{1_{j}} \}$, and $\{ x_{2_{j'}} \}$ be two subsequences such that their $\pi_{A}$ images
are two parallel lines, and let $\mathit{l}_{1}$, $\mathit{l}_{2}$ be two geodesics within $O(1)$ of $\{ x_{1_{j}} \}$ and $\{
x_{2_{j}} \}$ respectively.  We note to every $x_{1_{j}}$, there is a $x_{2_{j}}$ such that $\pi_{A}(x_{2_{j}})$ is closest to
$\pi_{A}(x_{1_{j}})$ amongst $\{ \pi_{A}(x_{2_{j}}) \}$, and furthermore, there is a line in $L$ containing $\pi_{A}(x_{1_{j}})$
and $\pi_{A}(x_{2_{j}})$, and by assumption, this means that there is a geodesic $\hat{\mathit{l}}_{j}$ (whose direction is the same
for all $j$) within $O(1)$ of $x_{1_{j}}$ and $x_{2_{j}}$.  Now if $\mathit{l}_{1}$ and $\mathit{l}_{2}$ are not in the same flat, then
for any root $\Xi$, $\pi_{\Xi}(\mathit{l}_{1})$ and $\pi_{\Xi}(\mathit{l}_{2})$ fork out with respect to the orientation
we previously fixed on $\mathit{l}_{1}$ and $\mathit{l}_{2}$.  When $x_{1_{j}}$, $x_{2_{j}}$ are far away from the fork point, this
causes a contradiction because the existence of $\hat{\mathit{l}}$ mean that $\pi_{\Xi}(x_{1_{j}})$ and $\pi_{\Xi}(x_{2_{j}})$ can be
connected by a vertical geodesic in $V_{\Xi}$, but they lie far away from the forking point of two geodesics.
Therefore if $\{ x_{i_{j}} \}$ is a subsequence for which their $\pi_{A}$ image is a affine 2-subspace, then $\{ x_{i_{j}} \}$ lie
within $O(1)$ of an affine 2-subspace in a flat. \medskip


Now suppose whenever $\{ x_{i_{j}} \}$ is a subset whose $\pi_{A}$ image is an affine $I$-subspace, $\{ x_{i_{j}} \}$ is within $O(1)$
of a flat.  Let $\{ x_{1_{j}} \}$, $\{ x_{2_{j}} \}$ be two subsets such that $\{ \pi_{A}(x_{1_{j}}) \}$, $\{ \pi_{A}(x_{2_{j}}) \}$ are
two parallel $I$-hyperplane, and $\mathit{h}_{1}$, $\mathit{h}_{2}$ be two affine $I$ subspace within $O(1)$ of $\{ x_{1_{j}} \}$ and
$\{ x_{2_{j}} \}$ respectively. \smallskip

Then we know for every $x_{1_{j}}$, there is a $x_{2_{j}}$ such that there is a line in $L$ containing $\pi_{A}(x_{1_{j}})$ and
$\pi_{A}(x_{2_{j}})$, so by assumption, $x_{1_{j}}$, $x_{2_{j}}$ are within $O(1)$ of a (straight) geodesic $\hat{\mathit{l}}_{j}$.
Furthermore we can assume WLOG that the direction of $\hat{\mathit{l}}_{j}$ are the same for all $j$.  Therefore if
$\mathit{h}_{1}$ and $\mathit{h}_{2}$ lie on two different flats, then for some root $\Xi$, $\pi_{\Xi}(\mathit{h}_{1})$ and
$\pi_{\Xi}(\mathit{h}_{2})$ are two vertical geodesic that fork apart. Therefore for $x_{1_{j}}$, $x_{2_{j}}$ such that
$\pi_{\Xi}(x_{1_{j}})$, $\pi_{\Xi}(x_{2_{j}})$ lying very far from the fork point where $\pi_{\Xi}(\mathit{h}_{1})$ and
$\pi_{\Xi}(\mathit{h}_{2})$ diverge from each other, this is a contradiction to the existence of $\hat{\mathit{l}}_{j}$ within
$O(1)$ of $x_{1_{j}}$ and $x_{2_{j}}$, for the latter would imply that $\pi_{\Xi}(x_{1_{j}})$ and $\pi_{\Xi}(x_{2_{j}})$ are within
$O(1)$ of a vertical geodesic in $V_{\Xi}$.  \end{proof}

\begin{proof}\textit{to Proposition \ref{FTF}} \newline
Fix a flat $\mathcal{F} \simeq \mathbf{A}$.  The composition
$\pi_{A} \circ \phi |_{\mathcal{F}}$ is homotopic to the identity
map when we consider $\mathcal{F}$ as $\mathbf{A}$, so $\pi_{A}
\circ \phi|_{\mathcal{F}}$ preserves homologies.  By equation
(\ref{weak affine eq}), it takes balls to neighborhood of balls, so
it preserves homologies relative to the complement of a ball,
therefore $\pi_{A} \circ \phi|_{\mathcal{F}}$ is onto. \medskip

Let $L$ be a good grid in $\mathbf{A}'$ under a group isomorphism $\psi$.  We can further make sure that for each basis $e_{i} \in \mathbb{Z}^{n}$,
$| \psi(e_{i}) - \psi(\vec{0}) | = 4M_{0}$, where $M_{0}$ is as in Corollary \ref{the same piA}.  The same corollary
also implies that for every point $\mathbf{b} \in L$, the subset $s_{\mathbf{b}}=\{ x \in \phi(\mathcal{F}): \pi_{A}(x) =\mathbf{b}\}$ of $G'$
is contained in a ball of radius at most $M_{0}$.  Therefore for distinct points $\mathbf{b}, \mathbf{d} \in L$,
$s_{\mathbf{b}} \cap s_{\mathbf{d}} = \emptyset$. \smallskip

Let $\{ x_{j} \} \subset \phi(\mathcal{F})$ be a subset such that $x_{j} \in s_{\mathbf{b}}$ for some $\mathbf{b} \in L$.
We choose $x_{j}$'s so that if $j \not= i$, $x_{j} \in s_{\mathbf{b}}$, $x_{i} \in s_{\mathbf{c}}$ for $\mathbf{b} \not= \mathbf{c}$.
Now let $\{ x_{i_{j}} \}$ be a subsequence whose $\pi_{A}$ images is a line in $L$, and $\{ y_{i_{j}} \} \subset \mathcal{F}$ be their
$\phi$-preimages in $\mathcal{F}$.  Then by equation (\ref{weak affine eq}), $d(y_{i_{j}},y_{i_{j+1}}) \leq \hat{C}(L)$, for all
$j$, for some constant $\hat{C}$ depending on $L$.  This means that that $d(x_{i_{j}}, x_{i_{j+1}}) \leq 2\kappa \hat{C}$ for all $j$.
Apply (i) of Lemma \ref{Forste-lemma A1-2} to $\{ x_{i_{j}} \}$ shows that the sequence is within $O(\hat{C})$ of a (bi-infinite) geodesic.
The claim now follows by applying Lemma \ref{preleude to FTF} to $\{ x_{j} \}$.  \end{proof}


\subsection{Consequences of flats go to flats}

The primarily reason for introducing root boundaries $\partial_{[\alpha]}$ for each root
class $[\alpha]$ and the (psudo) distance function is that a quasi-isometry of $G$
induces quasi-similarities among $(\partial_{[\alpha]}, D_{[\alpha]})$'s.

Recall that a map $F: X \rightarrow Y$ between metric spaces is called a
$(N,K)$-\emph{quasi-similarities} if
\[ N/K d(x,y) \leq d(F(x), F(y)) \leq NK d(x,y) \]

\noindent When $K=1$, the map is called a \emph{similarity}.  We write the group of
quasi-similarities of a space by $QSim(X)$.

\begin{proposition} \label{QI preserving flats}
If $\psi: G=\mathbf{H} \rtimes_{\varphi} \mathbf{A} \rightarrow G'=\mathbf{H}' \rtimes_{\varphi'} \mathbf{A}'$ is a quasi-isometry such that the image of left
translate of $\mathbf{A}$ is within $O(1)$ neighborhood of a left translates of $\mathbf{A}'$.  Then
\begin{enumerate}
\item $\psi$ sends foliations by root kernels of $G$ to the foliations by root kernels of $G'$.
\item The number of root classes of $G$ and $G'$ are the same.

\item $\psi$ sends foliations by root class horocycles of $G$ to foliations by root class horocycles of $G'$.

\item $\psi$ is within $O(1)$ of a product map $f \times g$, where $f: \mathbf{A} \rightarrow \mathbf{A}'$ is affine whose linear part
is some scalar multiple of a finite order element of $O(n)$; while $g=(g_{1},g_{2}, \cdots g_{s_{0}})$, each
$g_{i}$ is a bilipschitz map from $V_{[\alpha]}$ to $V_{f_{*}[\alpha]}$, where $f_{*}$ is the bijection from root classes of $G$ to root
classes of $G'$ as induced by the linear part of $f$.

\item For each root class $[\Xi]$, if we list the roots $\xi_{1} < \xi_{2} < \xi_{3} \cdots
\xi_{l}$, then $f$ induces a map on the roots such that $f_{*} \xi_{1} < f_{*} \xi_{2}
\cdots f_{*} \xi_{l}$, where $\{ f_{*} \xi \} = f_{*}[\Xi]$.  Furthermore, with respect
to this order of roots, $g|_{V_[\Xi]}: V_{[\Xi]} \rightarrow V_{f_{*}[\Xi]}$ has the form

\[ (x_{1}, x_{2}, \cdots x_{l} ) \mapsto ( g_{1}(x_{1},x_{2}, \cdots x_{l}), \cdots   g_{l-1}(x_{l-1},x_{l}) ,
g_{l}(x_{l}) ) \]

\noindent such that for each $i$, any $x_{j} \in V_{\xi_{j}}$ for $j > i$, the map $\bullet \mapsto g_{i}(\bullet, x_{i+1}, \cdots x_{l})$
is a $(e^{f_{*}\xi_{i}(\mathbf{t}_{0})}, 2\kappa e^{c})$ quasi-similarity, where $\mathbf{t}_{0}$ is the constant part of $f$, and $c$ is a
constant depending on $\psi$.


\item $\psi$ induces an quasi-similarity between root boundaries of $G$ to root boundaries of $G'$. \end{enumerate} \end{proposition}


\begin{proof}
Since two flats are within a finite Hausdorff distance of each other if and only if they are the same flat, our assumption means that
a flat $\hat{\mathcal{F}}$ within $O(1)$ Hausdorff neighborhood of $\psi(\mathcal{F})$ for $\mathcal{F}$ a flat, is unique.

If two flats $\mathcal{F}_{1}, \mathcal{F}_{2}$ have no-empty intersection (i.e. come together) then
$\hat{\mathcal{F}}_{1} \cap \hat{\mathcal{F}}_{2} \not= \emptyset$, and $d_{H}(\phi( \partial (\mathcal{F}_{1} \cap \mathcal{F}_{2})), \partial
(\hat{\mathcal{F}}_{1} \cap \hat{\mathcal{F}}_{2})) \leq M_{0}$. \medskip

The boundary of the set where two flats come together viewed as a subset of $\mathbf{A}$ is a finite union of hyperplanes parallel root
kernels.  If for two flats $\mathcal{F}$, $\mathcal{F}'$, $\partial( \mathcal{F} \cap \mathcal{F}')$ is a
union of hyperplanes parallel to at least two roots kernels, then there are flats $\mathcal{F}_{1}, \mathcal{F}_{2}, \cdots
\mathcal{F}_{k}$ such that $\mathcal{F} \bigcap \mathcal{F}' = \bigcap_{\iota=1}^{k} \mathcal{F} \bigcap
\mathcal{F}_{\iota}$.\smallskip

Therefore if $\partial(\mathcal{F}_{1} \cap \mathcal{F}_{2})$ is a hyperplane parallel to just one root kernel, then so is
$\partial(\hat{\mathcal{F}}_{1} \bigcap \hat{\mathcal{F}}_{2})$. \smallskip

Fix a flat $\mathcal{F}$.  Since every hyperplane in $\mathcal{F}$ parallel to a root kernel is the boundary of $\mathcal{F} \cap
\mathcal{F}'$ for some other flat $\mathcal{F}'$, it follows that if $h_{1},h_{2} \subset \mathcal{F}$ are two hyperplanes parallel to a
common root kernel, then there exists hyperplanes $\hat{h}_{1},\hat{h}_{2} \subset \hat{\mathcal{F}}$ parallel to a common root
kernel such that $d_{H}(\phi(h_{i}), \hat{h}_{i}) \leq A$ for $i=1,2$. In other words, $\phi|_{\mathcal{F}}: \mathcal{F}
\rightarrow \hat{\mathcal{F}}$ sends hyperplanes parallel to the kernel of a root, to hyperplanes parallel to the kernel of some
other root.  As $\psi$ is coarsely onto, it follows that the number of root kernels of
$G$ and $G'$ are the same.

For the remaining claims, we consider two cases separately.

\bold{Case I - $G$ has rank(G) many root kernel}

Since $G$ is unimodular, the sum of all its roots is zero.  But with only $rank(G)$ many root kernels, this means that there is a basis
$\{ \alpha_{i} \}$ of $\mathbf{A}^{*}$ such that every root of $G$ is of the form $c \alpha_{i}$ where $c \in \mathbb{R} - \{ 0 \}$.
By abuse of notation, we will write $V^{+}_{\alpha_{i}}$ to be the direct sum of root spaces where the root is positive multiples of
$\alpha_{i}$, even though $\alpha_{i}$ itself might not be a root, and similarly, $V^{-}_{\alpha_{i}}$ denote for direct sum of root spaces
where the root is negative multiples of $\alpha_{i}$. Let $\{ \beta_{i} \} \subset \mathbf{A}'^{*}$ be the analogous linear functionals
for $G'$.  In this way, $\mathbf{H}=\bigoplus_{i} V^{+}_{\alpha_{i}} \oplus V^{-}_{\alpha_{i}}$, and similarly
$\mathbf{H}'= \bigoplus_{i} V^{+}_{\beta_{i}} \oplus V^{-}_{\beta_{i}}$. \smallskip

First we show that left translates of $\mathbf{H}$ are taken to $O(1)$ Hausdorff neighborhood of left translates of $\mathbf{H}'$.
Take $p, q \in V^{\star}_{\alpha_{1}}$, $\star \in \{+ ,- \}$.  We construct a quadrilateral $Q$ (see Definition \ref{Forste-quadrilateral}
in \cite{Pg}) with $p,q$ as two of its vertices using left translates of $\bigcap_{j \not= 1} ker(\alpha_{j})$.  Since flats are taken to
$O(1)$ neighborhoods of flats, $\psi(Q)$ is within $O(1)$ neighborhood of a quadrilateral $\hat{Q}$ whose edges are left translates of
$\bigcap_{j \not=l} ker(\beta_{j})$ for some $l$, and therefore $\psi(p)$, $\psi(q)$ are within $O(1)$ neighborhood of a left translate of
$V^{\star}_{\beta_{l}}$, $\star \in \{+,- \}$.

More generally, for any two points $x, y \in \mathbf{H}$, we can find a sequence of
points $p_{i} \in \mathbf{H}$, such that $p_{0}=x$, $p_{s}=y$, and for all $i$, $p_{i}$,
$p_{i+1}$ differ only in a $V^{\star}_{\alpha_{i}}$, $\star \in \{ +, = \}$ coordinate.
The previous argument then implies $\psi(x)$ and $\psi(y)$ are within $O(1)$ of a left
translate of $\mathbf{H}'$, because the same is true of $\psi(p_{i})$ and $\psi(p_{i+1})$ for
all $i$. \medskip

Next we show that the restriction of $\psi$ to different flats agree. If $p, q \in \mathbf{H}$ are points such that for each
$\alpha_{i}$, these two points differ in only one of $V^{+}_{\alpha_{i}}$, $V^{-}_{\alpha_{i}}$ coordinate (but not both), then the
intersection of flats based on them, $p \mathcal{F}$, $q \mathcal{F}$, intersect at an unbounded set whose boundary is a union of hyperplane
parallel to kernels of $\alpha_{i}$'s that appear in the coordinate difference between $p$ and $q$.  Furthermore, we can find a third point
$x \in \mathbf{H}$ that differ in one of $V^{+}_{\alpha_{i}}$, $V^{-}_{\alpha_{i}}$ coordinate with each of $p$ and $q$ so that
$p \mathcal{F} \cap q \mathcal{F} \subset p \mathcal{F} \cap x \mathcal{F}$, and
$p \mathcal{F} \cap q \mathcal{F} \subset q \mathcal{F} \cap x \mathcal{F}$.  In this way, we see that
$\phi|_{p \mathcal{F}}$ is the same as $\phi|_{q \mathcal{F}}$.

More generally, for any two points $x, y \in \mathbf{H}$, we can find a finite number of
points $p_{i} \in \mathbf{H}$ such that $p_{0}=x$, $p_{l}=y$, such that for all $i$,
$p_{i}$ differ from $p_{i+1}$ in one of $V^{+}_{\alpha_{j}}$, $V^{-}_{\alpha_{j}}$
coordinate (but not both).  The previous argument then implies $\phi|_{x \mathcal{F}}$ is
the same as $\phi|_{x \mathcal{F}}$ (as $\phi|_{p_{i} \mathcal{F}}$ is the same as
$\phi|_{p_{i+1} \mathcal{F}}$). \medskip

Denote the common value of the restriction of $\psi$ to any flat as $f$.  Then we know
that $f$ induces a bijection between $\{ \alpha_{i} \}$ and $\{ \beta_{i} \}$ which are
numbered such that for all $i$, $\psi$ sends left translates of kernel of $\alpha_{i}$ to left
translates of kernel of $\beta_{i}$.  As left translate of $ker(\alpha_{i})$ (resp.
$ker(\beta_{i})$) can be identified with a rank 1 space $G_{\alpha_{i}}=\mathbf{H} \rtimes \mathbb{R}(\vec{v}_{\alpha_{i}})$
(resp. $G'_{\beta_{i}}=\mathbf{H}' \rtimes \mathbb{R}(\vec{v}_{\beta_{i}})$), $\psi$ induces a
quasi-isometry from $G_{\alpha_{i}}$ to $G'_{\beta_{i}}$ that sends left translates of
$\mathbf{H}$ to $O(1)$ neighborhood of left translates of $G'_{\beta_{i}}$.  By
Proposition 5.8 of \cite{Benson}, we conclude that the $f$-induced map from
$\mathbf{A}/ker(\alpha_{i})$ to $\mathbf{A}'/ ker(\beta_{i})$ is bounded distance from an
affine map. \smallskip

Repeat the above argument to all other $\alpha_{i}$'s shows that $f$ is $O(1)$ away from
being affine that respects root kernels of $G$ and $G'$.

\bold{Case II - $G$ has at least rank(G)+1 many root kernels. }
In this case, the restriction to a flat, $\phi|_{\mathcal{F}}$ preserves at least $n+1$ many parallel
families of hyperplanes, thus forcing $\phi|_{\mathcal{F}}$ to be affine (when $\mathcal{F}$ and $\hat{\mathcal{F}}$ are viewed as
$\mathbf{A}$ and $\mathbf{A}'$ respectively.) that respects root kernels.  \smallskip

So by identifying $\mathcal{F}$ and $\hat{\mathcal{F}}$ with $\mathbb{R}^{n}$, $\phi|_{\mathcal{F}}$ can be written as
$\mathbf{t} \mapsto M_{\mathcal{F}}(\mathbf{t}) + \mathbf{t}_{0}^{\mathcal{F}}$, where $M_{\mathcal{F}} \in GL(n)$ is
a linear map preserving root kernels.  This implies (i), that images of (straight) geodesics are within $O(1)$ Hausdorff neighborhood of
straight geodesics.  At this stage, both the affine map $M_{\mathcal{F}}$ and the translation $\mathbf{t}_{0}^{\mathcal{F}}$
depend on the flat $\mathcal{F}$. \medskip

Fix a root equivalence class $[\Xi]$.  Let $\vec{v} \in \mathbb{R}^{n}$ be a direction such that $\alpha(\vec{v}) \not= 0$
for all roots $\alpha$.  Then for any two points $x, y$ in the same left coset of $V_{[\Xi]}$, there is a quadrilateral $Q$ with $x,y$
as two of its vertices, and whose edges are all in direction $\vec{v}$.  WLOG we can assume $\Xi(\vec{v}) < 0$, so $x,y$ are in
the same left coset of $W^{-}_{\vec{v}}$.  Since $\psi$ takes geodesics to $O(1)$ neighborhood of geodesics, it follows that
$\psi(Q)$ is within $O(1)$ of another (0) quadrilateral.  Write $\mathcal{F}_{x}=x \mathbf{A}$, $\mathcal{F}_{y}=y \mathbf{A}$ for
the flats based at $x$ and $y$ respectively.  Then $\psi(Q)$ close to a quadrilateral means that $M_{\mathcal{F}_{x}}(\vec{v}) =
M_{\mathcal{F}_{y}}(\vec{v})$, and the dot product between $\vec{v}$ and $\mathbf{t}_{0}^{\mathcal{F}_{x}}$, as well as
$\mathbf{t}_{0}^{\mathcal{F}_{y}}$ are the same.  More over, writing $\vec{u}= M_{\mathcal{F}_{x}}(\vec{v})$, by
Lemma \ref{Forste-quadrilateral word in rank 1} in \cite{Pg}, $\psi(x)$, $\psi(y)$ are within $O(1)$ of a left coset of $W^{-}_{\vec{u}}$.  \smallskip

We can find at least $n$ $\vec{v}_{i}$'s, such that $\cap_{i}W^{\Xi(\vec{v}_{i})/|\Xi(\vec{v}_{i})|}_{\vec{v}_{i}} = V_{[\Xi]}$ and none of them
lies in a root kernel.  Repeat the same argument as before yields that $M_{\mathcal{F}_{x}}=M_{\mathcal{F}_{y}}=M$,
$\mathbf{t}_{0}^{\mathcal{F}_{x}}=\mathcal{t}_{0}^{\mathcal{F}_{y}}$. That is, the restriction of $\psi$ to the flats based at $x,y$ are
within $O(1)$ of each other.  WLOG, we can assume that we have so many $\vec{v}_{i}$'s, such that
$\cap_{i}W^{\Xi(\vec{v}_{i})/|\Xi(\vec{v}_{i})|}_{M(\vec{v}_{i})} = V_{[\beta]}$ for some root equivalence class $[\beta]$, which shows
that $\psi(x)$, $\psi(y)$ are within $O(1)$ of a left translate of $V_{[\beta]}$ for some root class $[\beta]$.  \smallskip

Since any two points on the same left translate of $\mathbb{R}^{m}$ (or $\mathbf{H}$) can be 'connected' by finitely many points, for
which successive pairs of points lie on a left translate of $V_{[\alpha]}$ for some root class $[\alpha]$, it follows that
$\psi$ sends a left translate of $\mathbf{H}$ to within $O(1)$ neighborhood of another left translate of $\mathbf{H}$. Since the
$\psi$ restricted to flats based at two points that lie on a common horocycle agree, we also have now that the restriction of $\psi$ to
any two flats agree. That is, the restriction of $\psi$ to any flat does not depend on the base point of the flat.
\\

So now we know that regardless of the number of root kernels, $\psi$ splits into $f \times g$, where $f: \mathbf{A} \rightarrow \mathbf{A}'$
affine and respects root kernels, while $g: \mathbf{H} \rightarrow \mathbf{H}'$ takes root class horocycles to root class horocycles.
Furthermore, the permutation on root classes induced by $f$ and $g$ agree.\smallskip

We now proceed to show that the $f$ actually induces a bijection between roots of $G$ and $G'$ (not just root
classes). \smallskip

Since we know now the $\psi$ restricted to any flat is the map $\mathbf{t} \mapsto M(\mathbf{t})+\mathbf{t}_{0}$.  This means that
(straight) geodesics are taken to straight geodesics, and we can compare the rate of divergence between two geodesics in the same direction
but at based at different points of $\mathbf{H}$. \smallskip

Specifically, take $\Xi$ a root, and let $p,q$ be two points on a common $\Xi$ horocycle. Pick some $\vec{v}
\in \mathbf{A}$, and let $\mathit{l}_{x}=x(t \vec{v})$, $\mathit{l}_{y}=y(t \vec{v})$ be two geodesic rays in direction
$\vec{v}$ leaving $x,y$ respectively.  Then $\psi(\mathit{l}_{x})$, $\psi(\mathit{l}_{y})$ are within $O(1)$ of $\mathit{l}_{x'}=x'(t
M(\vec{v}))$ and $\mathit{l}_{y'}=y'(t M(\vec{v}))$ respectively, where $x', y'$ are within $O(1)$ of a left translate of
$V_{[\beta]}$. \smallskip

The rate of divergence between $\mathit{l}_{x}$ and $\mathit{l}_{y}$ is $e^{t\Xi(\vec{v})}$; the divergence rate between
$\mathit{l}_{x'}$ and $\mathit{l}_{y'}$ is $e^{t \xi(M(\vec{v}))}$ for some $\xi \in [\beta]$. Since the two rates are QI to each
other, there must be some $\tilde{\beta} \in [\beta]$ such that $\Xi(\vec{v})=\tilde{\beta}(M(\vec{v}))$.  But $\vec{v}$ is
arbitrary, so $\Xi=\tilde{\beta} \circ M$, $x',y'$ are in the same left translate of $\oplus_{\xi \in [\beta]: \xi < \tilde{\beta}}
V_{\xi}$. \smallskip

This means that $f$ induces a bijection $\sigma_{f}$ between roots of $G$ and $G'$.  So that for each
root $\alpha$, $\psi$ sends left translates of $\oplus_{\xi \in [\alpha]: \xi \leq \alpha} V_{\xi}$, to $O(1)$ Hausdorff
neighborhoods of left translates of direct sum of root spaces where the roots are in the
same root class as $\sigma_{f}(\alpha)$ but less than $\sigma_{f}(\alpha)$. \smallskip

For $x, y \in V_{\Xi}$, write $g_{\Xi}(x), g_{\Xi}(y) \in V_{\sigma(\Xi)}$ for the $V_{\sigma(\Xi)}$
, and for each $\xi < \sigma(\Xi)$, write $\tilde{g}_{\xi}(x), \tilde{g}_{\xi}(y) \in
V_{\xi}$ for the $V_{\xi}$ component of $g(x)$ and $g(y)$, so that $g(x)=g_{\Xi}(x) + \sum_{\xi} \tilde{g}_{\xi}(x)$,
$g(y)=g_{\Xi}(y) + \sum_{\xi} \tilde{g}_{\xi}(y)$.  Pick a $\mathbf{t} \in \mathbf{A}$, then the distance between $(\mathbf{t},x)$ and
$(\mathbf{t},y)$ with respect to path metric in $\mathbf{t} \mathbf{H}$ is $e^{-\Xi(\mathbf{t})}|x-y|$.  The $\psi$ images of
$(\mathbf{t},x)$, $(\mathbf{t},y)$ is $c$ away from $(M\mathbf{t}+\mathbf{t}_{0}, g(x))$ and $(M \mathbf{t} + \mathbf{t}_{0},
g(y))$, so we have the following inequality:

\begin{eqnarray*}
\frac{1}{2 \kappa } e^{-c} e^{-\Xi(\mathbf{t})} |x-y | & \leq &
\left( e^{-\sigma(\Xi)(M \mathbf{t} + \mathbf{t}_{0})} P_{\Xi}(\sigma(\Xi)(M \mathbf{t} + \mathbf{t}_{0}))  |g_{\Xi}(x)-g_{\Xi}(y)| \right) \\
& + & \left( \sum_{\xi< \sigma(\Xi)}e^{-\xi(M \mathbf{t} + \mathbf{t}_{0})} P_{\xi}(\xi(M \mathbf{t} + \mathbf{t}_{0}))
|\tilde{g}_{\xi}(x)-\tilde{g}_{\xi}(y)| \right) \\
& \leq & 2\kappa e^{c} e^{-\Xi(\mathbf{t})} |x-y | \\ \end{eqnarray*}

Since $\sigma(\Xi) \circ M = \Xi$, dividing by $\Xi P_{\Xi}$ on both sides and let
$\Xi(\mathbf{t}) \rightarrow \infty$, and noting that for any polynomial $Q$, $\lim_{t \rightarrow \infty}
Q(t+t_{0})/Q(t)$ is bounded above by a number independent of $t_{0}$, so we end up with

\[ \frac{1}{2 \kappa } e^{-c} e^{\sigma(\Xi)(\mathbf{t}_{0})} \left| x-y \right| \leq
\left| g_{\Xi}(x) - g_{\Xi}(y) \right| \leq 2\kappa e^{c} e^{\sigma(\Xi)(\mathbf{t}_{0})} \left|x-y \right| \]

\noindent so the restriction of $g|_{V_{\Xi}}: V_{\Xi} \rightarrow V_{\sigma(\Xi)}$ is bilip with bilip constants
$2\kappa e^{c} e^{\sigma(\Xi)(\mathbf{t}_{0})}$, $1/2\kappa e^{-c} e^{\sigma(\Xi)(\mathbf{t}_{0})}$.

To summarize, $\psi$ is $O(1)$ from a map of the form $(\mathbf{x}, \mathbf{t}) \mapsto (g(\mathbf{x}), mA_{f}(\mathbf{t}) + \mathbf{t}_{0})$,
where $m> 0$, $A_{f}: \mathbf{A} \rightarrow \mathbf{A}'$ is a finite order element in $O(n)$ that preserves foliations by root kernels, while
$g: \mathbf{H} \rightarrow \mathbf{H}'$ sends root horocycles to root horocycles and furthermore respects the graded foliations
in each root horocycle. \medskip

Let $\sigma$ be the permutation on root classes induced by $M$.  Then as $\psi$ sends negative $[\alpha]$ half planes
to bounded neighborhood of negative $\sigma([\alpha])$ half planes, $\psi$ induces a map
from $\partial^{-}_{[\alpha]}$ to $\partial^{-}_{\sigma ([\alpha])}$ for every root class
$[\alpha]$ of $G$.

Furthermore, the map $\hat{q}: \mathbb{R} \sim \mathbf{A}/ker(\alpha_{0}) \rightarrow \mathbf{A}/ker(\sigma(\alpha)_{0})\sim \mathbb{R}$ is
bounded distance from an affine with linear term as $m$ and constant term as
$\sigma(\alpha)_{0}(\mathbf{t}_{0})$, we now show the induced map on lower root
boundaries is a quasi-similarity.

Take $p,q \in \partial_{[\alpha]} \sim V_{[\alpha]}$.  Then, definition of $D_{[\alpha]}(p,q)$ says that
\[ D_{[\alpha]}(p,q) = e^{m t_{p,q}} \] \noindent under $\phi$, we have

\[ D_{\sigma([\alpha])}(g(p),g(q)) = e^{\hat{q}(t_{p,q})} \] \noindent therefore
\begin{eqnarray*}
mt_{p,q} + \sigma(\alpha)_{0}(\mathbf{t}_{0}) - c & \leq & \hat{q}(t_{p,q}) \leq mt_{p,q} + \sigma(\alpha)_{0}(\mathbf{t}_{0}) + c \\
e^{m t_{p,q}} \mbox{    }   \frac{e^{\sigma(\alpha)_{0}(\mathbf{t}_{0})}}{e^{c}} & \leq & e^{\hat{q}(t_{p,q})}
\leq e^{m t_{p,q}}  \mbox{      }  e^{\sigma(\alpha)_{0}(\mathbf{t}_{0})}e^{c} \\
\frac{1}{e^{c}} \left( e^{\sigma(\alpha)_{0}(\mathbf{t}_{0})} D_{[\alpha]}(p,q) \right) & \leq &
D_{\sigma([\alpha])}(g(p), g(q)) \leq  e^{c} \left(e^{\sigma(\alpha)_{0}(\mathbf{t}_{0})} D_{[\alpha]}(p,q) \right)
\end{eqnarray*} \end{proof}

\begin{theorem} \label{behaviors of phi}
Let $G$, $G'$ be a non-degenerate, unimodular, split abelian-by-abelian solvable Lie group, and $\phi: G \rightarrow G'$ a $(\kappa, C)$
quasi-isometry.  Then $\phi$ is bounded distance from a composition of a left translation followed
by a map of the form $(\mathbf{x}, \mathbf{t}) \rightarrow (g(\mathbf{x}), f(\mathbf{t})$,
where $f$ is affine whose linear part is a positive of a finite order element $A_{f} \in
O(n)$ (n is the rank of $G$) that preserves foliations by root kernels, while $g=(g_{1}, g_{2}, \cdots, g_{\sharp})$,
$g_{i}$ is a bilip map from $V_{[\alpha_{i}]}$ to $V_{[\alpha_{i}]}$ with bilip constants depending only on $\kappa,
C$. \end{theorem}

\begin{proof}
Apply Proposition \ref{QI preserving flats} in light of Proposition \ref{FTF}.  \end{proof}

\begin{corollary}\label{diagORnot}
Let $G=\mathbf{H} \rtimes_{\varphi} \mathbf{A}$ be a a non-degenerate, unimodular, split
abelian-by-abelian group such that $\varphi(\mathbf{A})$ is diagonalizable, while $G'=\mathbf{H}' \rtimes_{\varphi^{'}} \mathbf{A}'$
is another non-degenerate, unimodular, split abelian-by-abelian group such
that $\varphi'(\mathbf{A}')$ is not diagonalizable.  Then $G$ and $G'$ are not
quasi-isometric. \end{corollary}

\begin{proof}
If there were, then Theorem \ref{behaviors of phi} implies that geodesics are taken to
geodesics and the induced height function on the geodesics are affine, which means that
we can compare rates of divergence between two geodesics in the same direction. In $G'$,
we would detect exponential polynomial growth while in $G$, only exponential growth can
be detected, and those two growth types are not q.i. to each other. \end{proof}

When the homomorphism appeared in the semidirect expression of a non-degenerate,
unimodular, abelian-by-abelian solvable Lie group is diagonalizable, uniform subgroups
of quasi-similarities of its root boundaries are analyzed in \cite{Dy}, and in this case we are able to
say something about an arbitrary finitely generated group in its quasi-isometric class.

\begin{corollary} \label{showing polycyclic}
Let $G=\mathbf{H} \rtimes_{\psi} \mathbf{A}$ be a non-degenerate, unimodular, split abelian-by-abelian solvable Lie group
where $\psi$ is diagonalizable.  If $\Gamma$ is a finitely generated group quasi-isometric to
$G$, then $\Gamma$ is virtually polycyclic. \end{corollary}

\begin{proof}
Let $\varphi: \Gamma \rightarrow G$ be a $(\kappa, C)$ quasi-isometry.  For each $\gamma
\in \Gamma$, write $L_{\gamma}$ for the left translation of $\gamma$, and $\tilde{L}_{\gamma}=\varphi
\circ L_{\gamma} \circ \varphi^{-1}$.  Then $\tilde{\Gamma}=\{ \tilde{L}_{\gamma}  \}_{\gamma \in \Gamma}$ constitute an uniform subgroup of
$QI(G)$, all with the same q.i. constants (they are all $(\kappa, C)$ quasi-isometries).

By Theorem \ref{behaviors of phi}, each $\tilde{L}_{\gamma}$ induces a permutation on root classes.  Therefore the map from
$\tilde{\Gamma}$ into the permutations on root classes of $G$ is a well-defined homomorphism, whose kernel, $\tilde{\Gamma}_{0}$ is finite
index in $\tilde{\Gamma}$.

Since $\tilde{\Gamma}$ is a uniform subgroup of $QI(G)$, by Proposition \ref{QI preserving flats}
$\tilde{\Gamma}_{0}$ is a uniform subgroup of $\prod_{[\alpha]} QSim(\partial^{-}_{[\alpha]})$.  Applying Theorem 2 in
\cite{Dy} to the image of $\tilde{\Gamma}_{0}$ in each $QSim(\partial^{-}_{[\alpha]})$
factor, we can conjugate $\tilde{\Gamma}_{0}$ into $\prod_{[\alpha]} ASim(\partial^{-}_{[\alpha]})$, the group of almost similarities.
Denote the image of $\tilde{\Gamma}_{0}$ in
$\prod_{[\alpha]} ASim(\partial^{-}_{[\alpha]})$ by $\hat{\Gamma}_{0}$.  Note that $\hat{\Gamma}_{0}$ and $\tilde{\Gamma}_{0}$ are isomorphic.

For each $\tilde{L}_{\gamma} \in \tilde{\Gamma}_{0}$, write $g_{[\Xi], \gamma}$ for the almost similarity on $(\partial^{-}_{[\Xi]},
D_{[\Xi]})$ and $t_{[\Xi], \gamma}$ the corresponding similarity constant, as induced by the image of $\tilde{L}_{\gamma}$ in $\hat{\Gamma}_{0}$.

\bold{Claim:} For each $\tilde{L}_{\gamma} \in \tilde{\Gamma}_{0}$, there is a $\mathbf{s}_{\gamma} \in \mathbf{A}$ such
that $t_{[\Xi], \gamma} = e^{\Xi_{0}(\mathbf{s}_{\gamma})}$ for each root class $[\Xi]$. \smallskip

We know that for each $\tilde{L}_{\gamma}$, there is a $\mathbf{t}_{0, \gamma} \in \mathbf{A}$ such that
$\tilde{L}_{\gamma}$ induces $(e^{\Xi_{0}(\mathbf{t}_{0, \gamma})}, e^{c})$ quasi-similarity $\tilde{g}_{[\Xi],\gamma}$ on
$(\partial_{[\Xi]}, D_{[\Xi]})$.  Theorem 2 of \cite{Dy} says that we can find $F_{[\Xi]} \in Bilip(\partial_{[\Xi]}, D_{[\Xi]})$ with bilip
constant $K'$ such that for every root class $[\Xi]$ and every $\tilde{L}_{\gamma} \in \tilde{\Gamma}_{0}$,

\[ F_{[\Xi]}\tilde{g}_{[\Xi], \gamma} F_{[\Xi]}^{-1}= g_{[\Xi],\gamma} \]

\noindent Therefore $t_{[\Xi], \gamma}$, the similarity constant of $g_{[\Xi], \gamma}$ satisfies

\begin{equation} \label{scaling constant}
t_{[\Xi], \gamma} \in [e^{\Xi_{0}(\mathbf{t}_{0, \gamma})}
\frac{1}{e^{c}K'^{2}}, e^{\Xi_{0}(\mathbf{t}_{0, \gamma})}e^{c}K'^{2} ] , \mbox{for all root classes }[\Xi] \end{equation}

\noindent Since the sum of roots is zero and
\[ 0 = \sum_{\alpha \mbox{ roots}} \alpha = \sum_{[\Xi]} \Xi_{0}
\frac{\mathit{l}_{[\Xi]}}{\Xi_{0}} \] \noindent it follows that

\[ \prod_{[\Xi]} \left( e^{\Xi_{0}(\mathbf{t}_{0, \gamma})} \right)^{\frac{\mathit{l}_{[\Xi]}}{\Xi_{0}}} = 1 \]

\noindent therefore
\begin{equation}\label{canonical product}
\prod_{[\Xi]} \left(  t_{[\Xi], \gamma} \right)^{\frac{\mathit{l}_{[\Xi]}}{\Xi_{0}}}
=1 \end{equation}

\noindent (because the left hand side lies in an interval those end points are constants independent of $\gamma$,
so if the left hand side was not 1, then the product for sufficiently high powers of
$\gamma$ escape the interval) \smallskip

On the other hand, for a generic linear functional $\ell$, we know that
\[
0 = \sum_{[\Xi]} \mathit{l}_{[\Xi]}(\vec{v}_{\ell}) = \sum_{[\Xi]} \Xi_{0}(\vec{v}_{\ell}) \frac{\mathit{l}_{[\Xi]}}{\Xi_{0}}
= \sum_{[\Xi]} \Xi_{0}(\mathbf{t}_{0,\gamma}) \frac{\Xi_{0}(\vec{v}_{\ell})}{\Xi_{0}(\mathbf{t}_{0,\gamma})} \frac{\mathit{l}_{[\Xi]}}{\Xi_{0}}
\]

\noindent By equation (\ref{scaling constant}) this means that
\begin{equation}\label{product in other direction}
\prod_{[\Xi]} \left( t_{[\Xi], \gamma} \right)^{\frac{\Xi_{0}(\vec{v}_{\ell})}{\Xi_{0}(\mathbf{t}_{0,\gamma})} \frac{\mathit{l}_{[\Xi]}}{\Xi_{0}}} =1
\end{equation}

\noindent By letting $\ell$ ranging over a subset of positive measure, equation (\ref{product in other
direction}) and (\ref{canonical product}) means that the $t_{[\Xi], \gamma}$ must be of
the form $e^{\Xi_{0}(\mathbf{s}_{\gamma})}$ for some $\mathbf{s}_{\gamma} \in \mathbf{A}$\footnote{When rank of $G$ is 2 or higher,
$t_{[\Xi],\gamma}$ equals to $e^{\Xi_{0}(\mathbf{t}_{0,\gamma})}$; but in rank 1 all that we can say is that it is of
the form $e^{\Xi_{0}(\mathbf{s}_{\gamma})}$ where $\mathbf{s}_{\gamma}$ might not be the
same as $\mathbf{t}_{0,\gamma}$}. \smallskip

This means that for each $\tilde{L}_{\gamma} \in \tilde{\Gamma}_{0}$, $\{ g_{[\Xi],\gamma} \}_{[\Xi]}$ determines
an element $\psi_{\gamma} \in QI(G)$ of the form
\[ \psi_{\gamma}((\mathbf{x}_{[\Xi]})_{[\Xi]}, \mathbf{t}) = \left(
( g_{[\Xi],\gamma}(\mathbf{x}_{[\Xi]}) )_{[\Xi]}, \mathbf{t} + \mathbf{s}_{\gamma} \right) \]

\noindent and we can define a homomorphism $h: \hat{\Gamma}_{0} \rightarrow \mathbf{A}$ as $\gamma
\mapsto \mathbf{s}_{\gamma}$.

The kernel of $h$ consist of elements with no
translations, so they leave the subgroup $\mathbf{H}$ invariant.

Since $\Gamma$ is quasi-isometric to $G$, the quasi-action of $\tilde{\Gamma}_{0}$ on $G$ is proper,
which means $\hat{\Gamma}_{0}$ and $ker(h)$ quasi acts properly on $\prod_{[\Xi]} (\partial^{-}_{[\Xi]},
D_{[\Xi]})$.  Now by Theorem 18 of \cite{Dy}, $\Gamma$ is virtually polycyclic.
\end{proof} \bigskip

In a group $G$, an element $x \in G$ is called \emph{exponentially distorted} if there are numbers $c, \epsilon$ such that for all
$n \in \mathbb{Z}$,
\[ \frac{1}{c} \log(|n| + 1) - \epsilon \leq \| x^{n} \|_{G} \leq c \log(|n|+1) + \epsilon
\]

\noindent where $\|x^{n} \|_{G}$ is the distance between the identity and $x^{n}$ in $G$.  In the case of a connected,
simply connected solvable Lie group $G$, Osin showed in \cite{Osin} that the set of
exponentially distorted elements forms a normal subgroup $R_{\emph{exp}}(G)$.

%
%

\begin{lemma}\label{splitting implied}
Let $G$ be a connected, simply connected solvable Lie group such that $$ 1 \rightarrow
R_{\emph{exp}}(G) \rightarrow G \rightarrow \mathbb{R}^{s} \rightarrow 1 ,$$ where
$R_{\emph{exp}}(G)$ is abelian.  Then the above sequence splits and $G$ is a semidirect
product of $R_{\emph{exp}}(G)$ and $\mathbb{R}^{s}$. \end{lemma}

\begin{proof}
Let $\gothic{h}$ be a Cartan subalgebra of $\gothic{g}$, the Lie algebra of
$G$.  Then $\gothic{v}$, the Lie algebra of $R_{\emph{exp}}(G)$, is generated
by root spaces in the decomposition of $\gothic{g}$ with respect to $\gothic{h}$.  Since
this is abelian, it means that $\gothic{g}$ is a semidirect product of $\gothic{h}$ and
$\gothic{v}$.  Since $\gothic{g}/\gothic{v}$ is abelian, $\gothic{h}$ is abelian.
\end{proof}

\begin{corollary} \label{showing lattice in almost me}
Let $G=\mathbf{H} \rtimes_{\psi} \mathbf{A}$ be a non-degenerate, unimodular, split abelian-by-abelian solvable Lie group
where $\psi$ is diagonalizable.  If $\Gamma$ is a finitely generated group quasi-isometric to
$G$, then $\Gamma$ is virtually a lattice in a unimodular semidirect product of $\mathbf{H}$ and
$\mathbf{A}$.  \end{corollary}

\begin{proof}
By Corollary \ref{showing polycyclic}, $\Gamma$ contains a finite index subgroup that is
polycyclic.  By a theorem of Mostow (Theorem 4.28 in \cite{Raghu}) which says that a
polycyclic group contains a finite index subgroup that embeds as a lattice in a
connected, simply connected Lie group, we have $\mathcal{L}$, a connected, simply
connected solvable Lie group to which $\Gamma$ is virtually a lattice of.  The crux of the proof
consists of showing that $\mathcal{L}$ satisfies the short exact sequence in Lemma \ref{splitting
implied}, and the argument is practically that of Section 4.3 in \cite{Dy} with minor modifications.
We reproduce the skeleton of the proof below, and refer the readers to the relevant sections in \cite{Dy} for details. \\

We are now going to construct a continuous homomorphism $\tilde{h}: \mathcal{L} \rightarrow
\mathbf{A} = \mathbb{R}^{n}$ that is onto, whose finitely generated kernel not only is quasi-isometric to
$\mathbf{H}$ but also coincides with the exponential radical of $\mathcal{L}$.  As any
finitely generated in the same quasi-isometric class as $\mathbb{R}^{n}$ must be
virtually $\mathbb{R}^{n}$, by Lemma \ref{splitting implied}, $\mathcal{L}$ is virtually a semidirect product of $\mathbf{A}$
and $\mathbf{H}$.  If this semidirect product were not unimodular, then $\mathcal{L}$
would be non-amenable, which is a contradiction because amenability is preserved under
quasi-isometry. \\

The $\tilde{h}$ is going to be the composition of the following three homomorphisms:
\begin{enumerate} \renewcommand{\labelenumi}{\Alph{enumi}.}
\item $\mathcal{L} \rightarrow \prod_{[\alpha]} QSim(\partial^{-}_{[\alpha]})$
\item Conjugation of a uniform subgroup of $\prod_{[\alpha]}
QSim(\partial^{-}_{[\alpha]})$ into $AIsom(G)$, where $AIsom(G)$ is the set of all maps
of the form $ \psi_{\gamma}((\mathbf{x}_{[\Xi]})_{[\Xi]}, \mathbf{t}) = \left(
( g_{[\Xi],\gamma}(\mathbf{x}_{[\Xi]}) )_{[\Xi]}, \mathbf{t} + \mathbf{s}_{\gamma} \right)
$.

\item $h: \prod_{[\alpha]} AIsom(\partial^{-}_{[\alpha]}) \rightarrow \mathbf{A}=\mathbb{R}^{n}$ \end{enumerate}

\bold{Homomorphism A.} We can assume without loss of generality, that $\Gamma$ itself is
a lattice in $\mathcal{L}$.  We start with the following construction which can be found in Section 3.2
of \cite{Furman}.  Choose some open subset $E \subset \mathcal{L}$ with compact
closure, such that $\mathcal{L}$ is the union of left translates of $E$ by $\Gamma$.
Also fix a function $p: \mathcal{L} \rightarrow \Gamma$ such that $x \in p(x)E$ for every
$x \in \mathcal{L}$.  Then by defining $q_{h}: \Gamma \rightarrow \Gamma$ as
$q_{h}(\gamma)=p(h\gamma)$ for every $h \in \mathcal{L}$, we obtain a homomorphism from
$\mathcal{L}$ into $QI(\Gamma)$.  Since $\Gamma$ is quasi-isometric to $G$, conjugating
by the quasi-isometry between $\Gamma$ and $G$, we obtain a homomorphism from
$\mathcal{L}$ into $QI(G)$, where the images have uniform quasi-isometric constants.  By
(v) of Proposition \ref{QI preserving flats}, we can realize $QI(G)$ as a subgroup of $\prod_{[\alpha]} QSim(\partial^{-}_{[\alpha]})$.
By passing to a finite index subgroup of $\mathcal{L}$ if necessary, we now have the homomorphism $A$ from
$\mathcal{L}$ to $\prod_{[\alpha]} QSim(\partial^{-}_{[\alpha]})$, whose image is a uniform subgroup of quasi-similarities.
Continuity of homomorphism $A$ follows from Proposition 26 of \cite{Dy} where continuity in each factor was obtained. \smallskip

\bold{Homomorphism B. }  Theorem 2 of \cite{Dy} says that we can conjugate the image of
homomorphism A. into $\prod_{[\alpha]} ASim(\partial^{-}_{[\alpha]})$.  That elements of $\prod_{[\alpha]}
ASim(\partial^{-}_{[\alpha]})$ can be realized as elements of $AIsom(G)$ follows from the
\bold{Claim} in the proof of Corollary \ref{showing polycyclic}.  Homomorphism B. is
continuous because conjugation is continuous.

\bold{Homomorphism C. } The definition of $AIsom(G)$ means that we have a well-defined
homomorphism into $\mathbf{A}=\mathbb{R}^{n}$, which is our homomorphism C.  Now if $q_{i}$ is a sequence in
$AIsom(G)$ approaching to identity, then the map each one of them induces on the
$\mathbf{A}$ factor also has to approach that of what the identity does.  Since the
identity map produces no change in the $\mathbf{A}$ factor, it follows that the image of
$q_{i}$'s under homomorphism C. approaches $\vec{0} \in \mathbf{A}$.\medskip

Since $\Gamma$ is quasi-isometric to $G$, the quasi-action\footnote{Conjugating each left translation
of $\Gamma$ by the quasi-isometry between $\Gamma$ and $G$ gives a quasi-action on $G$} of $\Gamma$,
and therefore $\mathcal{L}$, on $G$ is cobounded, it follows that $\tilde{h}$ must be
onto because it is continuous. \\

We now claim that $R_{\emph{exp}}(\mathcal{L})= ker(\tilde{h})$.  To this end, we need
the following from \cite{Osin}.

\begin{lemma} (\textit{Lemma 2.1 in \cite{Osin}}) \newline
Suppose $G$, $H$ are locally compact groups generated by some symmetric compact
neighborhoods of the identities, $\| \dot \|_{G}$, $\| \dot \|_{H}$are canonical norms on
$G$ and $H$, and $dis_{G}$, $dist_{H}$ are the induced metrics.  Assume $\phi: G
\rightarrow H$ is a continuous surjective homomorphism, then there is a constant $K$ such
that \[ dist_{H}(\phi(g_{1}), \phi(g_{2})) \leq K dist_{G}(g_{1},g_{2}) \] \end{lemma}
\medskip

Applying the lemma above to $\tilde{h}$ gives us that
\[ \| \tilde{h}(\gamma) \| \leq K |\gamma|_{\mathcal{L}}, \forall \gamma \in \mathcal{L}
\]

Now let $\gamma \in R_{\emph{exp}}(\mathcal{L})$ such that $| \tilde{h} |=c$.
Then for any $n$ \[ cn = | \tilde{h}(\gamma^{n}) | \leq K \left| \gamma^{n}
\right|_{\mathcal{L}} = K \log(n + 1) \]

\noindent So we must have $\tilde{h} = \vec{0}$, hence $R_{\emph{exp}}(\mathcal{L})
\subset ker(\tilde{h})$. \smallskip

Conulier showed in \cite{Cornulier} that for a connected, simply connected solvable Lie group $X$, the
\emph{asymptotic dimension}, defined as the dimension of $X/R_{\emph{exp}}(X)$ is a quasi-isometric invariant.  This means that
\[ \mbox{dim } \mathcal{L}/R_{\emph{exp}}(\mathcal{L}) = \mbox{dim }
G/R_{\emph{exp}}(G) = \mbox{dim } \mathbf{A} = n \]

\noindent However as $\tilde{h}$ is onto, the dimension of $\mathcal{L}/ker(\tilde{h})$
also equals $n$.  So $ker(\tilde{h})$ cannot be strictly bigger than
$R_{\emph{exp}}(\mathcal{L})$. \medskip

By construction, $\mathcal{L}$ quasi-acts properly on $G$ as a uniform group of
quasi-isometries, which means $ker(\tilde{h})$ quasi-acts properly on $\mathbf{H}$ as a uniform
group of quasi-similarities, so $ker(\tilde{h})$ is finitely generated by Proposition 20
in \cite{Dy}.  Fix a $p \in G$.  Then $\gamma \mapsto B \circ A (\gamma)(p)$ is a quasi-isometry from $\mathcal{L}$ to $G$.  Here $B,A$ refers
to the homomorphisms mentioned above.  The restriction of this map to $ker(\tilde{h}) =
R_{\emph{exp}}(\mathcal{L})$ produces a quasi-isometric embedding of
$R_{\emph{exp}}(\mathcal{L})$ into $\mathbf{H}$.  However since the cohomological
dimension is a quasi-isometry invariant \cite{Gersten}, the dimension of
$R_{\emph{exp}}(\mathcal{L})$ must equal that of $\mathbf{H}$.  Now by theorem 7.6 of
\cite{Benson}, this embedding must be coarsely onto, which means
$R_{\emph{exp}}(\mathcal{L})$ is quasi-isometric to $\mathbf{H}$, so $R_{\emph{exp}}(\mathcal{L})$ must be virtually
$\mathbf{H}$ since the latter is abelian.  \end{proof}

\bold{Example of a unimodular solvable Lie group not Q.I. to any finitely generated
groups}

We need the following result which is stated in \cite{EFW1}, and whose proof is finished in \cite{Dy}.
\begin{theorem}\label{showing lattice} \textit{Theorem 1 in  \cite{Dy}}
If the rank of $G$ is 1, then a finitely generated group $\Gamma$ quasi-isometric to $G$
is virtually a lattice in $G$. \end{theorem}

Let $G$ be a rank 1 group with weights $1,1,-2$.  It has no lattice\footnote{because if an elements of
$SL_{3}(\mathbb{Z})$ has two distinct eigenvalues, one of them repeated twice, then they
have to be $-1,-1,1$.  See \cite{Hasegawa}. But the diagonal matrix with those eigenvalues as entries is not conjugate
to the action of $\mathbb{R}$ on $\mathbb{R}^{3}$ i.e. the diagonal matrix with entries $e^{1}, e^{1},
e^{-2}$.} so Theorem \ref{showing lattice} says that it cannot be
quasi-isometric to any finitely generated groups.

\end{document}